\documentclass[a4paper,twoside,reqno]{amsart}
\usepackage{fullpage}

\title[Gysin sequences and $SU(2)$-symmetries of $C^*$-algebras]{Gysin sequences and $SU(2)$-symmetries of $C^*$-algebras}
\date{\today}

\author{Francesca Arici, Jens Kaad}
\address{Mathematical Institute, Leiden University, P.O. Box 9512, 2300 RA
Leiden, the Netherlands}
\email{f.arici@math.leidenuniv.nl}

\address{Department of Mathematics and Computer Science, The University of Southern Denmark, Campusvej 55, DK-5230 Odense M, Denmark}
\email{kaad@imada.sdu.dk}

\usepackage[all,cmtip]{xy}
\usepackage{slashed}
\usepackage{amsmath,amscd}
\usepackage{amssymb,amsfonts,mathrsfs}
\usepackage[driver=pdftex,margin=3cm,heightrounded=true,centering]{geometry}
\usepackage[colorlinks=true,linkcolor=black,citecolor=black]{hyperref}

\usepackage{color}

\usepackage[normalem]{ulem}

\usepackage{todonotes}
\usepackage{graphicx}

\theoremstyle{plain}

\newtheorem{thm}{Theorem}[section]
\newtheorem{prop}[thm]{Proposition}
\newtheorem{lemma}[thm]{Lemma}
\newtheorem{cor}[thm]{Corollary}

\theoremstyle{definition}
\newtheorem{dfn}[thm]{Definition}

\newtheorem{rem}[thm]{Remark}

\newtheorem{example}[thm]{Example}

\theoremstyle{plain}

\numberwithin{equation}{section}

\newcommand{\alpheqn}[1][\relax]{
     \refstepcounter{equation}
     \if#1\relax \relax
       \else \label{#1}
     \fi  
     \setcounter{saveeqn}{\value{equation}}%
    \setcounter{equation}{0}%
    \renewcommand{\theequation}{\thealphequation}}
\newcommand{\reseteqn}{\setcounter{equation}{\value{saveeqn}}%
     \renewcommand{\theequation}{\thearabicequation}}

\providecommand{\mathscr}{\mathcal} % a priori mathscr is mathcal
\IfFileExists{mathrsfs.sty}{
\usepackage{mathrsfs}}         %  \mathscr produces rsfs fonts
   {%  else try euler fonts
     \IfFileExists{eucal.sty}{
        \usepackage[mathscr]{eucal}} % \mathscr produces Euscript fonts,
   {% else both unavailable, do nothing, mathscr remains mathcal
   }
}

\newcommand{\Cs}{$C^*$-}

\newcommand{\cd}{\cdot}

\newcommand{\ot}{\otimes}
\newcommand{\hot}{\widehat \otimes}

\newcommand{\op}{\oplus}
\newcommand{\bop}{\bigoplus}
\newcommand{\opo}{\oplus \ldots \oplus}
\newcommand{\olo}{\otimes\ldots\otimes}
\newcommand{\plp}{+ \ldots +}

\newcommand{\we}{\wedge}

\newcommand{\ci}{\circ}

\newcommand{\ti}{\times}
\newcommand{\nn}{\mathbb{N}}
\newcommand{\zz}{\mathbb{Z}}

\newcommand{\rr}{\mathbb{R}}
\newcommand{\cc}{\mathbb{C}}

\newcommand{\ga}{\gamma}
\newcommand{\Ga}{\Gamma}
\newcommand{\de}{\delta}
\newcommand{\De}{\Delta}

\newcommand{\io}{\iota}

\newcommand{\la}{\lambda}

\newcommand{\om}{\omega}

\newcommand{\si}{\sigma}
\newcommand{\Si}{\Sigma}
\newcommand{\te}{\theta}
\newcommand{\Te}{\Theta}
\newcommand{\ze}{\zeta}
\newcommand{\pa}{\partial}

\newcommand{\ov}{\overline}
\newcommand{\C}[1]{\mathcal{#1}}

\newcommand{\T}[1]{\textup{#1}}

\newcommand{\B}[1]{\mathbb{#1}}

\newcommand{\s}[1]{\mathscr{#1}}

\newcommand{\fork}[2]{\left\{ \begin{array}{#1} #2 \end{array} \right.} 

\newcommand{\ma}[2]{\left(\begin{array}{#1} #2 \end{array} \right)}

\newcommand{\su}{\subseteq}

\newcommand{\q}{\qquad}
\newcommand{\qq}{\qquad \qquad}
\newcommand{\qqq}{\qquad \qquad \qquad}

\newcommand{\inn}[1]{\langle #1 \rangle}
\newcommand{\binn}[1]{\big\langle #1 \big\rangle}

\newcommand{\Falg}{F_{\T{alg}}}

\newcommand{\Cc}{\mathbb{C}}

\makeatletter
\@namedef{subjclassname@2020}{%
  \textup{2020} Mathematics Subject Classification}
\makeatother

\begin{document}

 %, 47B35} %, 58B32} 30H20 55R25, 

\begin{abstract}
Motivated by the study of symmetries of $C^*$-algebras, as well as by multivariate operator theory, we  introduce the notion of an $SU(2)$-equivariant subproduct system of Hilbert spaces. We analyse the resulting Toeplitz and Cuntz--Pimsner algebras and provide results about their topological invariants through Kasparov's bivariant $K$-theory. In particular, starting from an irreducible representation of $SU(2)$, we show that the corresponding Toeplitz algebra is equivariantly $KK$-equivalent to the algebra of complex numbers. In this way, we obtain a six term exact sequence of $K$-groups containing a noncommutative analogue of the Euler class.
\end{abstract}
\maketitle

\tableofcontents
\parskip 1ex
\linespread{1.1}

Motivated by the study of symmetries of \Cs algebras, as well as by multivariate operator theory, in this paper we introduce the notion of an $SU(2)$-equivariant subproduct system of Hilbert spaces. Starting from a unitary representation of the Lie group $SU(2)$ on a finite-dimensional Hilbert space, we give an algorithm for constructing such an equivariant subproduct system and describe the associated Toeplitz--Pimsner and Cuntz--Pimsner algebras.

In the spirit of noncommutative topology, we compute topological invariants through Kasparov's bivariant $K$-theory \cite{Kas80}. In particular, we provide a partial answer to Open Question 3 in \cite[Section 6]{Vis12} concerning the computation of the $K$-theory groups of the Cuntz--Pimsner and Toeplitz--Pimsner algebras of a subproduct system. More precisely, our main result, Theorem \ref{t:KKequiv}, concerns $KK$-equivalence between the Toeplitz algebra of the subproduct system of an irreducible $SU(2)$-representation and the \Cs algebra of complex numbers. We further use this equivalence result to prove that the defining extension for the Cuntz--Pimsner algebra of a subproduct system induces an exact sequence in operator $K$-theory which contains a noncommutative Euler class and hence resembles a Gysin sequence. Using the exact sequence, we are able to compute the $K$-theory groups of the Cuntz--Pimsner algebra of our $SU(2)$-subproduct system. To our knowledge, this is the first time that the $K$-theory groups of \Cs algebras associated to subproduct systems are explicitly computed.

Our work fits into the framework of noncommutative topology, building on representation theoretic techniques, as well as Kasparov's bivariant $K$-theory.  One of our driving motivations lies in the noncommutative description of principal fibre bundles through Hopf--Galois extensions, a theory which works both algebraically and topologically \cite{BdCH}. This approach allows one to extend the scope to consider symmetries implemented by compact quantum groups. 

It is natural to try to extend this analogy to bundles with fibres other than quantum groups, as described in \cite{BrSz19}, where the authors initiated the development of an algebraic framework for noncommutative bundles with quantum homogeneous fibres. Here, however, we still focus on the group case and set the basis for an \emph{operator theoretic} approach to the study of sphere bundles with fibre the three-dimensional sphere. We are following the bottom-up approach offered by both the classical construction of the associated principal $G$-bundle to a fibre bundle with structure group $G$, and the construction of the sphere bundle of a Hermitian vector bundle.

We build upon the earlier work \cite{AKL16}, where we observed how the Cuntz--Pimsner algebra \cite{Pim97} of a noncommutative line bundle can be interpreted as the algebra of functions on a noncommutative circle bundle. This analogy also works at the level of topological invariants: Pimsner's construction naturally yields an exact sequence in $K$-theory, which mimics the classical Gysin sequence for circle bundles \cite{Gys,Ka78}.

The generalisation of this construction to structure groups different from $U(1)$ is not so straightforward and has, to our knowledge, escaped a satisfactory treatment. For instance, when applying Pimsner's construction to the module of sections of a complex $n$-dimensional vector bundle, possibly carrying the action of a compact group $G$, the resulting \Cs algebra has the structure of a bundle of algebras with fibres the Cuntz algebra $\B O_n$ \cite{Vas05}, a very different object from the algebra of functions of the associated principal $G$-bundle. Nevertheless, understanding the properties and symmetries of such \Cs algebras is an interesting question, which was recently addressed in \cite{Dea21}, where the author studied the Cuntz--Pimsner algebras constructed starting from the action of a compact group $G$ on a complex Hermitian vector bundle and their crossed products by $G$.

%Here we likewise start from a compact group, namely the special unitary group $SU(2)$, and construct Cuntz--Pimsner like algebra. In our bottom-up approach, the resulting \Cs algebra...
%associated principal SU(2)-bundle to a complex vector bundle with structure group SU(2).
Inspired by the representation theory of the group $SU(2)$, in particular by the Clebsch--Gordan theory, we adopt a novel approach, which relies on the theory of subproduct systems of \Cs correspondences. Subproduct systems were first described by Shalit and Solel in \cite{ShSo09}, inspired by the dilation theory of semigroups of completely positive maps, and independently by Bhat and Mukherjee \cite{BM11} in the Hilbert space setting, under the name of \emph{inclusion systems}. Motivated by examples in quantum electrodynamics, the related notion of \emph{interacting Fock spaces} was investigated in  \cite{AcSk08,ALV97}. The theory of subproduct systems was further developed by Viselter, who extended the notions of covariant representation and of Cuntz--Pimsner algebras of a \Cs correspondence to this more general framework \cite{Vis11,Vis12}. More recently, Dor-On and Markiewicz \cite{DorMa14, DorMa17} applied the theory of subproduct systems to the study of stochastic matrices. 

Another motivation for our work can be found in the question of understanding operator and \Cs algebras arising from zeros of polynomials in noncommutative variables. This relates to the programme of studying noncommutative domains initiated by Popescu \cite{Po06,Po07}. In \cite[Section 7]{ShSo09} Shalit and Solel established a noncommutative \emph{Nullstellensastz}: every homogeneous ideal $I$ in the algebra of noncommutative polynomials corresponds to a unique subproduct system, and vice versa. In our case, for every $n \in \nn$, we consider noncommutative varieties whose defining ideal in the free algebra $ \cc \langle X_0, \dotsc, X_n \rangle $ is generated by a single degree-two homogeneous polynomial arising from the \emph{determinant} of an $SU(2)$-representation. From a purely algebraic perspective, our setting is closely related to the one-relator quadratic regular Koszul algebras of global dimension two studied in \cite{Zh97,Zh98}.

Given the central role played by representation theory in our approach, this work also connects with the recent preprint
%by Andersson
%, who employed subproduct system of Hilbert spaces to recovers the \Cs algebra of continuous functions on the (quantum)-homogeneous space as a limit of matrix algebras. In 
\cite{And15}, where the author introduced $\mathbb{G}$ subproduct systems for $\mathbb{G}$ a compact quantum group and constructed the associated Cuntz--Pimsner algebras solely from representation-theoretic data.
% While the use of representation-theoretic information is a feature we have in common with \cite{And15}, the main difference should be highlighted: 
The main difference here lies in the fact that while Anderson obtained one subproduct system for every compact (quantum) group, here instead we describe a recipe for a \emph{family} of subproduct systems, one for every irreducible representation of the group $SU(2)$.

The outline of the paper is as follows. Section \ref{s:pre} is devoted to preliminaries on the theory of subproduct systems: we introduce the notion of $G$-equivariant subproduct system of \Cs correspondences, which we then specialise to the Hilbert space case. At the end of the section, we recall the one-to-one correspondence between subproduct systems of Hilbert spaces and ideals in the algebra of noncommutative polynomials. 

In Section \ref{s:subsu2} we show how, starting from a unitary representation of the Lie group $SU(2)$ on a finite-dimensional Hilbert space, one can construct an \emph{$SU(2)$-equivariant subproduct system} of Hilbert spaces over the semi-group $\mathbb{N}_0$.  An essential ingredient in our construction is what we call the \emph{determinant} of the representation. This determinant will resurface later in our computations in $KK$-theory as one of the summands in the Euler class of the representation. 

We proceed to studying the fusion rules of our equivariant $SU(2)$-subproduct system in Section \ref{s:fusion}. This section contains several lemmas containing explicit computations and showcasing interesting combinatorial properties, on which our later analysis relies. In particular, the structural properties of our subproduct systems naturally lead us to the commutation relations in the Toeplitz algebras, described in Section \ref{s:commreltoe}. 

Finally, we focus on $K$-theoretic invariants: Section \ref{s:KKequiv} is dedicated to the proof of $KK$-equivalence between the Toeplitz algebra of an irreducible $SU(2)$-representation and the algebra of complex numbers $\cc$. In Section \ref{s:gysin}, we present our main application: we establish a Gysin sequence in operator $K$-theory and employ it to compute the $K$-theory groups of the Cuntz--Pimsner algebra of the subproduct system. As mentioned earlier, this is, to our knowledge, the first time that the $K$-theory groups of \Cs algebras associated to subproduct systems are computed. In the final section we conclude the paper by mentioning a few open questions that we would like to address in the future.

\section{Preliminaries on subproduct systems}\label{s:pre}
In this section, we review the theory of subproduct systems of correspondences, specialising to the Hilbert space case. From the point of view of multivariate operator theory, subproduct systems of Hilbert space provide the natural framework for the study of tuples of operators subject to polynomial constraints. We shall elaborate on this analogy in the last part of the section.

For a pair of $C^*$-correspondences $X$ and $Y$ over the same $C^*$-algebra $B$, we let $X \hot_B Y$ denote their interior tensor product, which is again a $C^*$-correspondence over the $C^*$-algebra $B$ (see for instance \cite[Section 4]{Lan}). In the case where $G$ is a locally compact group and both $X$ and $Y$ are $G$-$C^*$-correspondences over the same $G$-$C^*$-algebra $B$ we turn $X \hot_B Y$ into a $G$-$C^*$-correspondence as well by equipping it with the diagonal action $g(\xi \ot \eta) := g(\xi) \ot g(\eta)$. We say that a $C^*$-correspondence $X$ over $B$ is \emph{faithful} when the left action $B \to \B L(X)$ is an injective $*$-homomorphism and \emph{essential} when $B \cd X$ is a norm-dense $B$-submodule of $X$.

\begin{dfn}[{\cite{ShSo09},\cite{Vis12}}]\label{d:sps}
Suppose that $\{E_m\}_{m\in \nn_0}$ is a sequence of essential and faithful $C^*$-correspondences over a $C^*$-algebra $B$ and that $\io_{k,m} : E_{k + m} \to E_k \hot_B E_m$ is a bounded adjointable isometry for every $k,m \in \nn_0$. We say that $(E, \io)$ is a \emph{subproduct system} over $B$ when the following holds for all $k,l,m \in \nn_0$: 
\begin{enumerate}
\item $E_0 = B$;
\item $\io_{0,m} : E_m \to E_0 \hot_B E_m$ and $\io_{m,0} : E_m \to E_m \hot_B E_0$ are the canonical identifications (so that the adjoints are induced by the bimodule structure on $E_m$) and;
\item the two bounded adjointable isometries $(1_k \ot \io_{l,m}) \ci \io_{k,l + m}$ and $(\io_{k,l} \ot 1_m) \ci \io_{k+l,m} : E_{k + l + m} \to E_k \hot_B E_l \hot_B E_m$ agree, where $1_k$ and $1_m$ denote the identity operators on $E_k$ and $E_m$, respectively.
\end{enumerate}
We refer to the bounded adjointable isometries $\io_{k,m} : E_{k + m} \to E_k \hot_B E_m$, $k,m \in \nn_0$, as the structure maps of our subproduct system. 
\end{dfn}

Note that for every $k,m \in \nn_0$, we have the orthogonal projections
\begin{equation}
\label{eq:proj_nm}
p_{k,m} = \io_{k,m} \io^*_{k,m} : E_k \hot_B E_m \to E_{k + m} \su E_k \hot_B E_m .
\end{equation}

\begin{dfn}
Let $G$ be a locally compact group and let $(E,\io)$ be a subproduct system over a $C^*$-algebra $B$. We say that $(E,\io)$ is a \emph{$G$-subproduct system} when $B$ is a $G$-$C^*$-algebra and $E_m$ is a $G$-$C^*$-correspondence for all $m \in \nn$, such that the structure maps $\io_{k,m} : E_{k + m} \to E_k \hot_B E_m$ are $G$-equivariant for all $k,m \in \nn_0$.  
\end{dfn} 

\begin{example}
If $(X, \phi)$ is an essential and faithful $C^*$-correspondence over a $C^*$-algebra $B$, then the sequence $\{X^{\hot_B m}\}_{m = 0}^\infty$ defines a subproduct system over $B$, where the structure maps are given by the canonical identifications $X^{\hot_B (m + k)} \cong X^{\hot_B m} \hot_B X^{\hot_B k}$.
\end{example}

\begin{dfn}
Given a subproduct system $(E,\io)$ over $\nn_0$, one defines its \emph{Fock correspondence} as the infinite Hilbert $C^*$-module direct sum $F := \op_{m = 0}^\infty E_m$. 
\end{dfn}

In the case where $G$ is a locally compact group and $(E,\io)$ is a $G$-subproduct system it holds that the Fock correspondence $F$ is a $G$-Hilbert $C^*$-module where the action of $G$ on $F$ is given by
\[
g( \{ \xi_m \}_{m = 0}^\infty) := \{ g(\xi_m) \}_{m = 0}^\infty
\]
for all $g \in G$ and $\{ \xi_m\}_{m = 0}^\infty \in F$.

For each $\xi \in E_k$, we define the \emph{creation operator} $T_\xi \in \B L (F)$ as
\[
T_\xi : F \to F \q T_\xi( \ze ) := \io^*_{k,m}(\xi \ot \ze) , \q \ze \in E_m \su F.
\]
\begin{dfn}
Let $(E,\io)$ be a subproduct system over $\nn_0$. We define the \emph{Toeplitz algebra} of the subproduct system $E$, denoted $\B T_E$, as the smallest unital $C^*$-subalgebra of $\B L( F)$ that contains all the creation operators, i.e.
\[
T_\xi \in \B T_E \q \T{for all } \xi \in E_k \, \, , \, \,\, k \in \nn_0 .
\]
\end{dfn}

\begin{lemma}\label{l:toepaction}
Let $G$ be a locally compact group and suppose that $(E,\io)$ is a $G$-subproduct system. Then the assignment $g(T_\xi) := T_{g(\xi)}$ defines a strongly continuous action of $G$ on the Toeplitz algebra $\B T_E$. %Moreover this $G$-action descends to a strongly continuous action of $G$ on the Cuntz--Pimsner algebra $\B O_E$.
\end{lemma}
\begin{proof}
Since the structure maps are $G$-equivariant it holds that 
\[
g T_\xi g^{-1}(\eta) = g \io^*_{k,m}( \xi \ot g^{-1} \eta) = \io^*_{k,m}(g \xi \ot \eta) = T_{g(\xi)}(\eta). 
\]
This proves that we have a well-defined action of $G$ on $\B T_E$. The strong continuity follows from the continuity properties of the action of $G$ on each $E_k$, $k \in \nn_0$. Indeed, since $E_k$ is a $G$-$C^*$-correspondence it holds for every $\xi \in E_k$ that the map $G \to E_k$ given by $g \mapsto g(\xi)$ is continuous. Moreover, since $\| T_\xi \| \leq \| \xi \|$ we obtain that the map $G \to \B T_E$ given by $g \mapsto T_{g(\xi)}$ is continuous for every $\xi \in E_k$. Finally, the strong continuity of our $G$-action follows since the Toeplitz algebra is generated by the creation operators $T_\xi$, $\xi \in E_k$. 
\end{proof}

Covariant representations of subproduct systems were studied in \cite{Vis11}. In the subsequent work \cite{Vis12}, the author described how one can associate a Cuntz--Pimsner algebra to every subproduct system over $\nn_0$. This algebra is constructed as a quotient of the Toeplitz algebra of the subproduct system by a suitable gauge invariant ideal (cf. \cite[Definition 2.6]{Vis12}). We recall the definition here:

For each $m \in \nn_0$ we let $Q_m : F \to F$ denote the orthogonal projection with image $E_m \su F$. 

\begin{dfn}
Let $(E,\io)$ be a subproduct system over $\nn_0$. The \emph{Cuntz--Pimsner algebra} of the subproduct system $(E,\io)$, denoted $\B O_E$, is the unital $C^*$-algebra obtained as the quotient of the Toeplitz algebra $\B T_E$ by the ideal
\[
\B I_E := \big\{ x \in \B T_E \mid \lim_{m \to \infty}\| Q_m x \| = 0 \big\} .
\]
Thus, $\B O_E := \B T_E / \B I_E$.
\end{dfn} 

In the case where $G$ is a locally compact group acting on a subproduct system $(E,\io)$, we obtain that our action of $G$ on the Toeplitz algebra $\B T_E$ descends to an action of $G$ on the Cuntz--Pimsner algebra. Indeed, this follows immediately since $g(\B I_E) \su \B I_E$ for all $g \in G$.

Viselter furthermore proved that, if $(E,\io)$ is a subproduct system of finite-dimensional Hilbert spaces, then the ideal $\B I_E$ is isomorphic to $\B K (F)$ (cf. \cite[Corollary 3.2]{Vis12}). Thus, in this case we have that $\B O_E = \B T_E / \B K(F)$.

\subsection{Subproduct systems and zeros of polynomials in noncommutative variables}\label{ss:poly}
We conclude this section by recalling how subproduct systems offer a framework for studying tuples of operators satisfying relations given by homogeneous polynomials. Our main reference is \cite[Section 7]{ShSo09}. In what follows we will restrict our attention to the finite-dimensional case.

Let $X := \lbrace x_0, \dots, x_n \rbrace$ be a finite set of $n + 1$ variables. We shall denote the free monoid generated by $X$ by $\langle X \rangle$, with unit the empty word, denoted by $1$. We denote by $X^m$ the set of all words of length $m$ in $\langle X \rangle$, so that the free monoid $\langle X \rangle$ is naturally graded by length.

Let $\cc \langle X \rangle := \cc \langle x_0, \dotsc, x_n \rangle$ denote the complex free associative unital algebra generated by $X$. Similarly to the free monoid, the free associative unital algebra $\cc \langle X \rangle$ is also graded by length. An element of $\cc \langle X \rangle$ is called a noncommutative polynomial. A noncommutative polynomial $f\in  \cc \langle X \rangle$ is \emph{homogeneous of degree $m$} if $f \in \cc X^{m}$.
By a homogeneous ideal in $\cc \langle X \rangle$ we mean a two-sided ideal which is the $\cc$-linear span of a set of homogeneous noncommutative polynomials belonging to the ideal. 

Let  $T = (T_0,T_1,\ldots,T_n)$ be an $(n+1)$-tuple of operators acting on a Hilbert space $H$. If $\alpha = (\alpha_1, \dotsc,\alpha_m) \in X^m$ is a word of length $m$, then we shall use the multi-index notation to indicate the product
\[
T^{\alpha} := T_{\alpha_1} \dotsc T_{\alpha_m},
\]
with the convention that $T^{1} = 1_{H}$.

If $p(x) = \sum c_{\alpha} x^{\alpha} \in \cc \langle X \rangle$ is a noncommutative polynomial, then $p(T)$ refers to the linear combination of operators $p(T) := \sum c_{\alpha} T^{\alpha}$. 

\begin{prop}[{\cite[Proposition 7.2]{ShSo09}}]\label{prop:1to1poly}
Let $H$ be an $(n+1)$-dimensional Hilbert space with orthonormal basis $\lbrace e_i \rbrace_{i=0}^n$. Then there is a bijective inclusion-reversing correspondence between proper homogeneous ideals $J  \subseteq \mathbb{C} \langle x_0,\ldots,x_n \rangle$ and standard subproduct systems $\lbrace E_m \rbrace_{m \in \mathbb{N}_0}$ with $E_1 \subseteq H$ (all structure maps are given by canonical inclusions).
\end{prop}

The correspondence works as follows: for a noncommutative polynomial $p=\sum c_{\alpha} x^{\alpha} \in \cc\inn{X}$, we write $p(e) = \sum c_{\alpha} e_{\alpha_1} \otimes \dotsc \otimes e_{\alpha_m}$. To any proper homogeneous ideal $J \subseteq \mathbb{C} \langle X \rangle $, we associate the standard subproduct system with fibres $E^J_m := H^{\otimes m} \ominus \lbrace p(e) \vert p \in J^{(m)} \rbrace$, for every $m \geq 0$, where $J^{(m)}$ denotes the degree $m$ component of the ideal $J$.

Conversely, given a standard subproduct system of Hilbert spaces $\lbrace E_m \rbrace_{m \in \mathbb{N}_0}$ with $E_1 \su H$, we associate to it the proper homogeneous ideal $J^{E} = \mathrm{span}_{\cc} \lbrace p \in \mathbb{C}\langle X \rangle \ \vert \ \exists m >0 : p(e) \in H^{\otimes m} \ominus E_m \rbrace$. %One can check that the ideal is indeed a {\red proper} homogeneous ideal.

The fact that the two maps are inverses to each other follows from the properties of the structure maps of a subproduct system outlined in Definition \ref{d:sps}.

Following \cite[Definition 7.3]{ShSo09}, we refer to $E^J$ and $J_E$ as the \emph{subproduct system associated to the ideal $J$}, and the \emph{ideal associated to the subproduct system $E$}, respectively.

Note that, while the subproduct system $E^J$ associated to a proper homogeneous ideal $J  \subseteq \cc \langle X \rangle$ depends on the choice of orthonormal basis for the Hilbert space $H$, different choices give rise to isomorphic subproduct systems (cf. \cite[Proposition 7.4]{ShSo09}).

In this work, we will be considering subproduct systems arising from a homogeneous ideal generated by a single degree two homogeneous polynomial. From an algebraic viewpoint, these ideals are  examples of the defining ideals for the one-relator quadratic regular Koszul algebras of global dimension two studied in \cite{Zh97,Zh98}. 

\section{Subproduct systems from $SU(2)$-actions}\label{s:subsu2}
Let $\tau : SU(2) \to U(H)$ be a strongly continuous unitary representation of the Lie group $SU(2)$ on a finite-dimensional Hilbert space $H$. 

We shall in this section see how every such representation $\tau : SU(2) \to U(H)$ gives rise to an $SU(2)$-subproduct system of finite-dimensional Hilbert spaces. These subproduct systems and their associated Cuntz--Pimsner algebras are the main focus of the present paper. To our knowledge these Cuntz--Pimsner algebras have so far only been studied in the particular case where the representation agrees with the fundamental representation of $SU(2)$ on $\cc^2$.

In that case, our procedure recovers the symmetric subproduct system on $\cc^2$ (cf. \cite[Example 1.3]{ShSo09} and \cite[Example 2.3]{Vis12}). In the language of \cite{And15}, this is an example of a $\mathbb{G}$-subproduct system, namely \emph{the} $SU(2)$-subproduct system. However, while \cite{And15} shows how to construct a canonical subproduct systems for every compact (quantum) group, here we focus on the $SU(2)$ case and obtain a \emph{family} of subproduct systems: one for every finite-dimensional representation.

\begin{dfn}
We define the \emph{determinant} of $H$ with respect to the representation $\tau$ as the subspace of invariant elements with respect to the diagonal action $\tau \otimes \tau$ on the tensor product $H \otimes H$: 
\[
\T{det}(\tau,H) = \lbrace \xi \in H \otimes H \mid 
\big( \tau(g) \ot \tau(g) \big)\xi = \xi  \quad \forall g \in SU(2) \rbrace.
\]
\end{dfn}

For each $m \in \{2,3,\ldots \}$ and each $i \in \{1,2,\ldots,m-1\}$ we define the strongly continuous unitary representation
\[
\De_m(i) : SU(2) \to U( H^{\ot m}) \q \De_m(i) := 1^{\ot (i-1)} \ot (\tau \ot \tau) \ot 1^{\ot (m - i-1)} .
\]
We then have the subspace $K_m(i) \su H^{\ot m}$ of invariant elements given by
\begin{equation}\label{eq:invariant}
K_m(i) := \lbrace \xi \in H^{\ot m} \mid \De_m(i)(g)(\xi) = \xi, \quad \forall g \in SU(2) \rbrace,
\end{equation}
and we consider the vector space span:
\begin{equation}\label{eq:invarsum}
K_m := \T{span}_{\Cc}\big\{ \xi \mid \xi \in K_m(i) \T{ for some } i \in \{1,2,\ldots,m-1\} \big\}
= \sum_{i = 1}^{m-1} K_m(i) \su H^{\ot m} .
\end{equation}
In particular, we remark that $K_2 = K_2(1) = \T{det}(\tau,H)$. 

Note that we have the following isomorphisms of vector spaces: 
\begin{equation}\label{eq:quad_quot}
K_m = K_2 \ot H^{\ot (m-2)} +H \ot K_2 \ot H^{\ot (m-3)} \plp H^{\ot (m-2)} \ot K_2 \su H^{\ot m} .
\end{equation}

For each $m \in \nn_0$ we put
\[
E_m(\tau, H) := \fork{ccc}{ K_m^\perp \su H^{\ot m} & \T{for} & m \geq 2 \\ H & \T{for} & m = 1 \\ \Cc & \T{for} & m = 0 } .
\]
When the representation $\tau : SU(2) \to U(H)$ is clear from the context we will suppress it from the notation and put $E_m := E_m(\tau,H)$.

We record the following:

\begin{lemma}\label{l:diarep}
Let $m \in \{2,3,\ldots \}$. The diagonal representation
\[
\tau^{\ot m} : SU(2) \to U( H^{\ot m})
\]
restricts to a strongly continuous unitary representation of $SU(2)$ on the subspace $E_m \su H^{\ot m}$.
\end{lemma}
\begin{proof}
Since $\tau^{\ot m}$ is a unitary representation, it suffices to show that each $K_m(i) \su H^{\ot m}$ is an invariant subspace for $\tau^{\ot m}$. Thus, let $\xi \in K_m(i)$ for some $i \in \{1,2,\ldots,m-1\}$ and let $g,h \in SU(2)$. We then have that
\[
\begin{split}
\De_m(i)(h) \tau(g)^{\ot m}(\xi)
& = \big( \tau(g)^{\ot (i-1)} \ot 1^{\ot 2} \ot \tau(g)^{\ot (m-i - 1)} \big) \De_m(i)(h) \De_m(i)(g)(\xi) \\
& = \big( \tau(g)^{\ot (i-1)} \ot 1^{\ot 2} \ot \tau(g)^{\ot (m-i - 1)} \big)(\xi) = \tau(g)^{\ot m}(\xi) .
\end{split}
\]
This proves the lemma.
\end{proof}

For each $m \geq 2$, we denote the strongly continuous representation of $SU(2)$ on $E_m$ by
\[
\tau_m : SU(2) \to U(E_m).
\]

Clearly, $SU(2)$ also acts on $E_1 = H$ (via the representation $\tau$) and on $\Cc$ (via the trivial representation).

We consider the sequence $E = \{ E_m\}_{m = 0}^\infty$ of finite-dimensional $SU(2)$-Hilbert spaces together with the structure maps $\io_{k,m} : E_{k  + m} \to E_k \ot E_m$, $k , m \in \nn_0$, induced by the canonical identification $H^{\ot (k+m)} \cong H^{\ot k} \ot H^{\ot m}$.

\begin{prop}
The pair $(E,\io)$ is an $SU(2)$-subproduct system.
\end{prop}
\begin{proof}
Consider $k,m \in \nn_0$, we need to verify that $E_{k  + m} \su E_k \ot E_m$. We assume that $k,m \geq 2$ and leave the remaining (easier) cases to the reader. We recall that $E_l = K_l^\perp$ for all $l \in \{2,3,\ldots \}$, so we need to show that
\[
K_{k + m}^\perp \su K_k^\perp \ot K_m^\perp,
\]
but this is equivalent to showing that
\[
K_k \ot H^{\ot m} + H^{\ot k} \ot K_m  = (K_k^\perp \ot K_m^\perp)^\perp \su K_{k + m} .
\]
The inclusion $K_k \ot H^{\ot m} + H^{\ot k} \ot K_m \su K_{k+ m}$ is an immediate consequence of the definition of $K_l$ for $l \in \{2,3,\ldots \}$, see \eqref{eq:invariant} and \eqref{eq:invarsum}. 

By definition of the involved $SU(2)$-actions we obtain that the inclusions $\io_{k,m} : E_{k + m} \to E_k \ot E_m$ are $SU(2)$-equivariant.
\end{proof}

\begin{rem}
Note that our subproduct system is by construction isomorphic to the \emph{maximal} subproduct system with prescribed fibres $E_1=H$ and $ E_2:=\det(\tau,H)^\perp$, as defined in \cite[Section 6.1]{ShSo09}. However, the context in \cite[Section 6.1]{ShSo09} does not in general yield the extra structure of an $SU(2)$-subproduct system. 
\end{rem}

We denote the Fock space of our $SU(2)$-equivariant subproduct system by
\[
F := F(\tau,H) := \op_{m = 0}^\infty E_m(\tau,H) = \op_{m = 0}^\infty E_m 
\]
and the associated strongly continuous action of $SU(2)$ on $F$ by
\begin{equation}\label{eq:tauinf}
\tau_\infty := \op_{m = 0}^\infty \tau_m :  SU(2) \to U(F) .
\end{equation}
For each $m \in \nn_0$, we recall that the orthogonal projection onto $E_m \su F$ is denoted by $Q_m : F \to F$. 

We apply the notation 
\[
\B T:= \B T(\tau,H) \su \B L(F) \q \T{and} \q \B O := \B O(\tau,H) := \B T/ \B K(F)  .
\]
for the associated Toeplitz algebra and Cuntz--Pimsner algebra. By the observations carried out in Section \ref{s:pre} we see that both the Toeplitz algebra and the Cuntz--Pimsner algebra carry a gauge action of $SU(2)$.

We let $\Falg \su F$ denote the algebraic direct sum of the subspaces $E_m \su F$:
\[
\Falg := F_{\T{alg}}(\tau,H) := \T{span}\{ \xi \in F \mid \xi \in E_m \T{ for some } m \in \nn_0 \big\} .
\]

We also define $F_+ \su F$ as the Hilbert space direct sum
\[
F_+ := \op_{m = 1}^\infty E_m
\]
and denote the vacuum vector by $\om := 1 \in E_0 = \cc \su F$, so that $F_+$ identifies with the orthogonal complement $(\cc \om)^\perp \su F$. In particular, we have that
\[ 
F_{+} = (1-Q_0)F \quad \T{and} \quad \cc \om = Q_0 F.
\]

\begin{rem}
Since the Hilbert space $H$ is finite-dimensional, it follows from the definition of the determinant as a subspace of $H \otimes H$, that the correspondence from Proposition \ref{prop:1to1poly} maps the generators of $\T{det}(\tau,H)$ to a finite number of quadratic polynomials. Therefore, our subproduct system corresponds to an ideal generated by a finite collection of quadratic polynomials, and this ideal in turn corresponds to a quadratic algebra (through the correspondence described, for instance, in \cite[Chapter 4]{Ma18}).

It is therefore not surprising that we make use of the identity \eqref{eq:quad_quot} when inductively constructing our subproduct system: the same formula is used in algebra for realising any given quadratic algebra as a quotient of the tensor algebra. 
\end{rem}

\subsection{Example: the case of the fundamental representation}\label{ss:funda}
We are now going to describe the subproduct system over $\nn_0$ coming from the fundamental representation $\rho : SU(2) \to U(\Cc^2)$. We let $\{f_0,f_1\}$ denote the standard basis for $\Cc^2$. 

We have that
\[
\T{det}(\rho, \Cc^2) = \Cc \cd ( f_0 \ot f_1 - f_1 \ot f_0 ) \su \Cc^2 \ot \Cc^2
\] 
and thus that
\[
K_m(i) = (\Cc^2)^{\ot (i-1)} \ot \Cc \cd (f_0 \ot f_1 - f_1 \ot f_0) \ot (\Cc^2)^{\ot m - i - 1}
\]
for all $m \in \{2,3,\ldots \}$ and all $i \in \{1,2,\ldots, m-1\}$.  Remark in particular that $\T{det}(\rho,\Cc^2)$ agrees with the usual determinant of $\Cc^2$ namely the wedge-product $\Cc^2 \we \Cc^2 \su \Cc^2 \ot \Cc^2$.

Let now $m \in \nn$. We recall that the $m$-fold symmetric tensor product of a finite-dimensional Hilbert space $H$ may be defined as the invariant subspace
\[
H^{\ot_S m} := \big\{ \xi \in H^{\ot m} \mid \si(\xi) = \xi \quad \forall \si \in S_m \big\},
\]
where the symmetric group $S_m$ acts unitarily on $H^{\ot m}$ via the rule 
\[
\Phi_{\si}( \xi_1 \ot \xi_2 \olo \xi_m) := \xi_{\si^{-1}(1)} \ot \xi_{\si^{-1}(2)} \olo \xi_{\si^{-1}(m)} .
\]
In particular, we have the identity of vector spaces
\[
(\Cc^2)^{\ot_S m} = E_m(\rho,\Cc^2) .
\]
This is a consequence of the Clebsch--Gordan theory for the representations of $SU(2)$ (cf. \cite[{Appendix C}]{Ha15}) and of the properties of the symmetric subproduct system \cite[Examples 1.3, 6.4]{ShSo09}.

For each $m \in \nn$, we define the vectors
\[ 
f_0^k f_1^{m-k} := p_m(f_0^{\ot k} \ot f_1^{\ot (m-k)}), \quad k=0, \dots, m,
\]
where $p_m : (\cc^2)^{\ot m} \to (\cc^2)^{\ot m}$ denotes the orthogonal projection onto the symmetric tensor product $(\cc^2)^{\ot_S m} \su (\cc^2)^{\ot m}$.
The vectors $\lbrace f_0^k f_1^{m-k}, \ k=0, \dots, m \rbrace$ form an orthogonal vector space basis for $E_m(\rho, \Cc^2)$ and their norm is given by the combinatorial expression
\begin{equation}\label{eq:norm_basis_fund_rep}
\| f_0^k f_1^{m-k} \|^2 = \frac{k! (m-k)!}{m!}
\end{equation} 
as described in \cite[Lemma 3.8]{Arv98}.

Thanks to the identification between symmetric tensors and homogeneous polynomials, we obtain a unitary isomorphism between the resulting Fock space $F(\rho,\Cc^2)$ and the Drury--Arveson space $H^2_2$, see \cite{Dru,Arv98,Sh14}. %\todo{I deleted the definition and discussion of the Drury--Arveson space.}

%{\color{purple} The latter is defined as a graded completion of the space of complex polynomials in two commuting variables $\mathbb{C}[z_0, z_1]$. For $z = (z_0, z_1)$ and every multi-index $\alpha= (\alpha_0, \alpha_1) \in \mathbb{N}_0^2$ we write 
%\[ z^{\alpha} = z_0^{\alpha_0} z_{1}^{\alpha_{1}}.\]
%The Drury--Arveson space $H^2_d$ is then the graded completion of the polynomials $\mathbb{C}[z_0,\dots,z_{d-1}]$, with respect to the topology induced by the inner product
%\[ \langle z^\alpha, z^\beta \rangle = \delta_{\alpha,\beta} \frac{\alpha!}{\vert \alpha ! \vert},\]
%where we are following the usual conventions on the multi-index factorial and norm:  $\alpha! := \alpha_0!  \alpha_{1}!$ and $\vert \alpha \vert = \alpha_0 + \alpha_{1}$.
%
%Geometrically, the Drury--Arveson space $H^2_2$ can be identified with the space of holomorphic
%functions on the unit ball $\varphi : \mathbb{B}^2 \subseteq \mathbb{C}^2\to \mathbb{C}$ which have a power series $\varphi (z) = \sum_{\alpha} c_{\alpha} z^{\alpha}$ satisfying \[
%\| \varphi \|^2_2 :=
% \sum_{\alpha} \vert c_{\alpha} \vert  \frac{\alpha!}{\vert \alpha ! \vert} < \infty.\] 
 % 
 % The assignment $z^{\alpha} \mapsto f_0^{\alpha_0} f_1^{\alpha_1}$ implements the unitary isomorphism between $H^2_d$ and our Fock space $F(\rho,\Cc^2)$ (cf. \cite[Page 14]{Sh14}).}
%{\color{violet} Note \todo{maybe put this in a footnote?} that the construction works for any integer $d$, as well as for $d=\infty$. The resulting Drury--Arveson space $H^2_d$ is isomorphic to the symmetric Fock space in $d$ variables.}

On our Fock space we introduce the unbounded selfadjoint operator $N : \T{Dom}(N) \to F(\rho,\Cc^2)$ defined by $N(\xi) = m \cd \xi$ for every homogeneous $\xi \in E_m$. The domain of $N$ is given explicitly by 
\[
\T{Dom}(N) := \big\{ \{ \xi_m\}_{m = 0}^\infty \in F \mid \{ m \cd \xi_m \}_{m = 0}^\infty \in F \big\} .
\]
The unbounded selfadjoint operator $N$ is referred to as the \emph{number operator}.

\begin{thm}[{cf. \cite[Proposition 5.3]{Arv98}, \cite[Example 6.4]{ShSo09}}]
\label{t:toe3sph}
The Toeplitz algebra $\B T(\rho,\Cc^2)$ associated to the fundamental representation is the $C^*$-subalgebra of $\B L(F)$ generated by the two operators $T_0 := T_{f_0}$ and $T_1:=T_{f_1}$. These satisfy the commutation relations
\begin{align}
& \label{eq:Toe_fund1} T_0T_1 =T_1T_0, \\
& \label{eq:Toe_fund2}T_0^* T_0 + T_1^* T_1 = (2+N)(1+N)^{-1},\\
& \label{eq:Toe_fund3} T_i^*T_j - T_jT_i^* = (1+N)^{-1} (\delta_{ij}1-T_jT_i^*).
\end{align}
\end{thm}
In other words, the pair of operators $(T_0,T_1)$ is a commuting, essentially normal row contraction. We remark that the two operators also satisfy $T_0T_0^* + T_1T_1^* = 1-Q_0$, i.e., the contraction is pure.

\begin{thm}[{\cite[Thm. 5.7]{Arv98}}]
The Toeplitz algebra $\B T(\rho,\Cc^2)$ contains the algebra of compact operators on the Drury--Arveson space $H^2_2$, and we have an exact sequence of $C^*$-algebras
\begin{equation}
\xymatrix{ 0 \ar[r] & \B K(H^2_2) \ar[r] & \B T(\rho,\Cc^2)\ar[r] & C(S^3) \ar[r] & 0},
\end{equation}
where $C(S^3)$ is the commutative $C^*$-algebra of continuous functions on the $3$-sphere $S^3 \su \B C^2$. In particular, we have that the Cuntz--Pimsner algebra $\B O(\rho,\Cc^2)$ is isomorphic to $C(S^3)$.
\end{thm}

The above Toeplitz extension is well-studied and understood. Moreover, the Toeplitz algebra is known to be $KK$-equivalent to the complex numbers. We are going to prove that this is a general feature of the Toeplitz algebras of the $SU(2)$-subproduct systems constructed from irreducible $SU(2)$-representations.

\subsection{Computation of determinants}
We now provide a computation of the subspace $\det(\tau,H) \su H \ot H$, starting with the case where the representation $\tau : SU(2) \to U(H)$ is irreducible. We disregard the case where $\tau$ is (unitarily equivalent to) the trivial representation on $\cc$.

We put $n := \T{dim}(H) - 1 \in \nn$ and we let $L_n = (\cc^2)^{\ot_S n}$ denote the $n$-fold symmetric tensor product $\cc^2$. We let $\rho_n : SU(2) \to U(L_n)$ denote the irreducible representation obtained by restriction of the $n$-fold tensor product of the fundamental representation. For each $k \in \{0,1,\ldots,n\}$ we define the unit vector
\begin{equation}\label{eq:orthobasis}
e_k := \sqrt{ \frac{n!}{k! (n-k)!} } \cd {f_0^k f_1^{n-k}} \in L_n 
\end{equation}
so that $\{e_k\}_{k = 0}^n$ is an orthonormal basis for $L_n$, see Subsection \ref{ss:funda}.

\begin{prop}\label{p:irrdet}
Suppose that $\tau : SU(2) \to U(H)$ is irreducible and let $V : L_n \to H$ be a unitary operator intertwining $\tau$ with $\rho_n$.  
Then the determinant $\det(\tau,H) \su H \ot H$ is a one-dimensional vector space spanned by the vector
\[
(V \ot V)\big( (n+1)^{-1/2} \sum_{k = 0}^n (-1)^{n-k} e_k \ot e_{n-k} \big) .
\]
\end{prop}
\begin{proof}
Using the representation theory for $SU(2)$ we know that we may find a unitary operator $W$ from $\op_{m = 0}^n L_{2m}$ to $L_n \ot L_n$ intertwining the representations $\op_{m = 0}^n \rho_{2m}$ and $\rho_n \ot \rho_n$. The structure of this unitary operator is determined by the Clebsch--Gordan coefficients and on $L_0 = \cc$ it is given by 
\[
W(1) = (n+1)^{-1/2} \sum_{k = 0}^n (-1)^{n-k} e_k \ot e_{n-k} = \sum_{k,l = 0}^n C^{0,0}_{n/2,k-n/2,n/2,l-n/2} \cd e_k \ot e_l . \qedhere
\]
\end{proof}

\begin{rem}%Definition \cite[Definition 7.3]{ShSo09}, 
Going back to the correspondence described in Subsection \ref{ss:poly}, the homogeneous ideal associated to the subproduct system of the irreducible representation $\rho_n : SU(2) \to U(L_n)$ is the proper homogeneous ideal in the free algebra on $(n+1)$ generators $\cc \langle x_0, \dotsc, x_n \rangle$ generated by the single degree two homogeneous polynomial $p(x_0, \dotsc, x_n) = \sum_{i=0}^n (-1)^i x_i x_{n-i}$.
\end{rem} 

In the more general case where $\tau : SU(2) \to U(H)$ need not be irreducible, we choose a unitary operator $V : \bop_{m = 0}^\infty L_m^{\op k_m} \to H$ intertwining the representations $\bop_{m = 0}^\infty \rho_m^{\op k_m}$ and $\tau$, where $k_m \in \nn_0$ for all $m \in \nn_0$ and we identify $L_n^{\op 0}$ with $\{0\}$. Of course, since $H$ is finite-dimensional, there exists an $M \in \nn_0$ such that $k_m = 0$ for all $m \geq M$.

\begin{prop}\label{p:compdetredu}
The determinant $\det(\tau,H) \su H \ot H$ has dimension $\sum_{m = 0}^\infty k_m^2$ and is unitarily isomorphic to the Hilbert space
\[
\bop_{m = 0}^\infty \det(\rho_m,L_m)^{\op k_m^2} \su \bop_{m =0}^\infty (L_m \ot L_m)^{\op k_m^2}
\]
via the isometry 
\[
\begin{CD}
\bop_{m =0}^\infty (L_m \ot L_m)^{\op k_m^2} \cong \bop_{m = 0}^\infty ( L_m^{\op k_m} \ot L_m^{\op k_m}) 
@>{\io}>> H \ot H,
\end{CD}
\]
where $\io$ is defined in degree $m$ by $\io(\xi_m \ot \eta_m) := V( \xi_m \de_m) \ot V(\eta_m \de_m)$.
\end{prop}
\begin{proof}
Using the unitary operator $V : \op_{m = 0}^\infty L_m^{\op k_m} \to H$ we identify $H \ot H$ with 
\[
\big( \op_{m = 0}^\infty L_m^{\op k_m} \big) \ot \big( \op_{l = 0}^\infty L_l^{\op k_l} \big)
\cong \op_{m,l = 0}^\infty (L_m \ot L_l)^{\op k_m \cd k_l} .
\]
Under this unitary isomorphism the representation $\tau \ot \tau$ identifies with the representation $\op_{m,l = 0}^\infty (\rho_m \ot \rho_l)^{\op k_m \cd k_l}$. Since the tensor product of representations $\rho_m \ot \rho_l$ contains no copy of the trivial representation for $m \neq l$, the determinant in question identifies with $\op_{m = 0}^\infty \det(\rho_m,L_m)^{\op k_m^2}$. The claim concerning the dimension of the determinant now follows immediately from Proposition \ref{p:irrdet}.
\end{proof}

\section{Fusion rules for an $SU(2)$-equivariant subproduct system}
\label{s:fusion}
From now on, we fix a strictly positive integer $n \in \nn$ and consider the irreducible representation $\rho_n : SU(2) \to U(L_n)$. We write $\{ e_k \}_{k = 0}^n$ for the orthonormal basis for the Hilbert space $L_n = (\cc^2)^{\ot_S n}$ introduced in \eqref{eq:orthobasis}. We put
\[
D := \det(\rho_n,L_n) \su L_n \ot L_n 
\]
so that $D$ is a one-dimensional vector space spanned by the unit vector 
\begin{equation}\label{eq:delta_vect}
\de := \frac{1}{\sqrt{n+1}} \cd \sum_{k = 0}^n (-1)^k e_k \ot e_{n-k} \in D ,
\end{equation}
as shown in Proposition \ref{p:irrdet}.

We have an associated sequence of finite-dimensional Hilbert spaces $\{ E_m \}_{m = 0}^\infty := \{ E_m(\rho_n,L_n) \}_{m = 0}^\infty$ defined as in Section \ref{s:subsu2}. Each of these Hilbert spaces carries a strongly continuous unitary representation of $SU(2)$ which in degree $m \in \nn_0$ is induced by the tensor product $\rho_n^{\ot m} : L_n^{\ot m} \to L_n^{\ot m}$. We emphasise that these representations are in general not irreducible (unless $n = 1$).

The main result of this section is the following orthogonal decomposition of the tensor products:

\begin{thm}\label{t:fusionintro}
For each $k,l \in \nn_0$ there exists an explicit $SU(2)$-equivariant unitary isomorphism
\[
{ E_k \ot E_l \cong E_{k + l} \op E_{k + l -2} \opo E_{|k-l|} .}
\]
\end{thm}

We view Theorem \ref{t:fusionintro} as an expression of the \emph{fusion rules} for our $SU(2)$-equivariant subproduct system. Moreover, for $n > 1$ one may interpret Theorem \ref{t:fusionintro} as a non-irreducible solution to the fusion rules of $SU(2)$. For $n = 1$ we exactly recover the usual (irreducible) fusion rules of $SU(2)$ (see for instance \cite{CFT}). {The fusion rules presented in Theorem \ref{t:fusionintro} play a key role in our later computation of the $K$-theory of the Toeplitz algebra $\B T(\rho_n, L_n)$.}

For every $k,m \in \nn_0$, we remind the reader of the notation  
\[
\io_{k,m} : E_{k + m} \to E_k \ot E_m \q \T{and} \q p_{k,m} := \io_{k,m} \io^*_{k,m} : E_k \ot E_m \to E_k \ot E_m
\]
for the inclusion and the associated orthogonal projection.
%
%For each $k,m \in \nn_0$ we apply the notation 
%\[
%\io_{k,m} : E_{k + m} \to E_k \ot E_m \q \T{and} \q p_{k,m} := \io_{k,m} \io^*_{k,m} : E_k \ot E_m \to E_k \ot E_m
%\]
%for the inclusion and the associated orthogonal projection.

\subsection{Preliminaries on integer sequences}
We consider the sequence of strictly positive integers $\{ d_m \}_{m = 0}^\infty$ defined recursively by the formula:
\begin{equation}\label{eq:d_m_rec}
d_0 := 1 \, \, , \, \, \, d_1 := n+1 \, \, , \, \, \,
d_m := d_1 \cd d_{m-1} - d_{m-2} \, \, , \, \, \, m \geq 2 .
\end{equation}
We furthermore put $d_{-1} :=0$. These sequences are well-studied and understood and we refer the reader to the Online Encyclopaedia of Integer Sequences where examples are given, \cite{seq:d2,seq:d3,seq:d4,seq:d5}. 

Later on, in Lemma \ref{l:decomp}, we shall see that $d_m = \T{dim}(E_m)$ for all $m \in \nn_0$. Towards this goal, we start out by summarising various identities involving the numbers $d_m \in \nn$, $m \in \nn_0$. 

\begin{lemma}\label{l:prop_d_con}
Let $m,k,l \in \nn_0$. We have the identities
\[
d_m^2 - d_{m-1} d_{m+1} = 1 \q \mbox{and} \q \sum_{i = 0}^l d_{k+m + 2i} =  d_{k + l} d_{m + l} - d_{k-1} d_{m-1} .
%\label{eq:did2} &  d_{m+1} = \frac{d_1}{d_m} \cd \sum_{j = 0}^m (-1)^{j + m} d_j^2 \\
\]
\end{lemma}
\begin{proof}
For the convenience of the reader, we provide a proof of the second of the two identities. The proof runs by induction on $l \in \nn_0$ but the only tricky part is the induction start. So suppose that $l = 0$. We shall prove by induction on $m \in \nn_0$ that
\begin{equation}\label{eq:did4}
d_{k + m} = d_k d_m - d_{k-1} d_{m-1}
\end{equation}
whenever $k \in \nn_0$ is fixed. For $m = 0,1$ there is nothing to prove, so supposing that the identity in \eqref{eq:did4} is verified for all $m \in \{0,1,\ldots,m_0\}$ for some $m_0 \in \nn$, we compute that
\[
\begin{split}
d_{k + m_0 + 1} & = d_{k + m_0} d_1 - d_{k + m_0 -1} = (d_k d_{m_0} - d_{k-1} d_{m_0-1} ) d_1 - d_k d_{m_0-1} + d_{k - 1} d_{m_0-2} \\
& = d_k (d_{m_0} d_1 - d_{m_0-1}) - d_{k-1} (d_{m_0-1} d_1 - d_{m_0-2})
= d_k d_{m_0+1} - d_{k-1} d_{m_0} .
\end{split}
\]
This proves the lemma.
\end{proof}

We remind the reader that $n \in \nn$, i.e., we are excluding the case of the trivial representation. This is essential for our results, which do not hold for $n=0$.
%We recall that $n \in \nn$ and hence in particular that $n \neq 0$. Otherwise, the result here below would be incorrect.

\begin{lemma}\label{l:quocon}
The sequence of quotients $\left\lbrace {d_{m-1}}/{d_m} \right\rbrace_{m = 0}^\infty$ is strictly increasing and converges to the limit $\ga_n = ({n+1 - \sqrt{ (n+1)^2 - 4}})/2 \in (0,1]$.
\end{lemma}
\begin{proof}
We first remark that $d_{m+1}> d_{m}$ for all $m \in \nn_0$, and hence that $d_{m+1} \geq m + 1$ (because $d_0 = 1$). Indeed, assuming that $d_m > d_{m-1}$ for some $m \in \nn$ we obtain that
\[
d_{m + 1} -d_m = d_m \cd n  - d_{m-1}
> d_{m-1} \cd (n - 1) \geq 0,
\]
since $n \in \nn$ by our standing assumptions. The claimed result now follows by induction (remark that the assumption $n \in \nn$ translates into the strict inequality $d_1 > d_0$.

We also observe that Lemma \ref{l:prop_d_con} implies that
\[ \frac{d_{m-1}}{ d_m} = \sum_{j = 1}^m \left( \frac{d_{j-1}}{d_j}
-\frac{d_{j-2}}{d_{j-1}} \right)
= \sum_{j = 1}^m \frac{1}{d_{j-1}d_j} .\]
This shows that our sequence is strictly increasing and moreover, our lower bound on the dimensions imply that the infinite sum $\sum_{j = 1}^\infty \frac{1}{d_{j-1}d_j}$ converges.

In order to compute the limit $\ga_n$ we apply \eqref{eq:d_m_rec} to see that
\[
\frac{d_{m-1}}{d_m} = \frac{d_m+d_{m-2}}{d_m d_1} =\frac{1}{d_1} +  \frac{d_{m-2}}{d_1 d_m}
= \frac{1}{n+1} + \frac{1}{n+1} \cd \frac{d_{m-2}}{d_m},
\]
for all $m \in \nn$, implying by taking limits that
\[
\ga_n = \frac{1}{n+1} + \frac{1}{n+1} \cd \ga_n^2 .
\]
The above quadratic equation has only one solution in the interval $(0,1]$, which yields
\begin{equation}
\label{eq:gamma}
\ga_n = \frac{n+1 - \sqrt{ (n+1)^2 - 4}}{2} .
\end{equation}
This proves the claim.
\end{proof}

\begin{rem}
%The recurrence relations satisfied by the $E_{m}$ are well known and understood. In particular, the particular form of the determinant given in Proposition \ref{p:irrdet} implies that 
Note that $d_{m}$ agrees with the number of length $m$ words in the alphabet $\lbrace 0, 1, 2, \dots, n \rbrace $ that do not contain the string $(0,n)$ (cf. \cite[Corollary 37]{Jan15}). In particular, our sequences are an example of cardinality sequences of word systems: thanks to \cite[Proposition 3.2]{GS14}, for every finite-dimensional subproduct systems of Hilbert spaces $\lbrace H_m \rbrace_{m\in \nn_0}$, there exists a word system $\lbrace X_m \rbrace_{m\in \nn_0}$ such that $\dim(H_m) = \vert X_m \vert$ for all $m \in \nn_0$ (see also \cite[Lemma 1.1]{An82} for a noncommutative algebraic version of this claim). However, the subproduct system associated to the word system described above is, in general, not isomorphic to the original one.

For $n \geq 2$, the constant $\gamma_n$ in \eqref{eq:gamma} equals the Perron--Frobenius eigenvalue of the $(n+1)\times (n+1)$-matrix with all entries equal to $1$ and except for a single $0$ in position $(1,n+1)$. See for instance \cite[Observation 1.4.2]{Kit}. For $n=1$ we cannot use the Perron--Frobenius theory because the matrix associated to the set of words in the alphabet is not an \emph{irreducible} one. Still the above ratio converges to the highest eigenvalue of said $2 \ti 2$-matrix. 
\end{rem}

 To end this subsection, we define the strictly positive integers
\begin{equation}\label{eq:mucon}
\mu_m := \frac{d_m d_{m-1}}{d_1} \, \, , \, \,\, m \in \nn .
\end{equation}
%and we put $\mu_0 := 1$. \todo{do we need this?}

Using the recursive definition \eqref{eq:d_m_rec}, it can be verified that the sequences $\{ \mu_m \}_{m = 1}^\infty$ and $\{ d_m \}_{m = 0}^\infty$ are connected via the identity
\begin{equation}\label{eq:mudee}
d_m^2 = \mu_m + \mu_{m+1} \, \, , \, \,\, m \in \nn .
\end{equation}
This can be used to prove that the sequence $\{ \mu_m \}_{m = 1}^\infty$ can also be obtained using the recurrence relation
\begin{equation}
\mu_{m+1} = ((n+1)^2-2) \mu_m - \mu_{m-1} +1 , \quad \mu_1=1, \quad \mu_2= (n+1)^2 -1 .
\end{equation}
It is easy to see that by definition, for $n=1$ we obtain the triangular numbers \cite{seq:triangle}, as \eqref{eq:mucon} reduces to a binomial coefficient. For small $n > 1$, we also recover some well known combinatorial sequences: for $n=2, 3$ we obtain the sequences \cite{seq:mun2} and \cite{seq:mun3}, respectively. For $n=4$, the sequence $\{ 5\mu_m \}_{m = 1}^\infty$ agrees with \cite{seq:5mu4}.
 
Since at the moment of writing this paper the sequences $\{ \mu_m\}_{m = 1}^\infty$ for $n \geq 4$ were not listed in the OEIS, we are currently in the process of updating the database. We have done so starting from the sequence for $n=4$ \cite{seq:mu4}. 

%\todo[inline]{Update: I now have an account there, and I will start the submission.}

\subsection{Decomposing tensor products by $E_1$ from the right}
We start out by proving the decomposition result in Theorem \ref{t:fusionintro} in the case where the second representation space is just $E_1$. Thus, for every $m \in \nn$, we are going to show that $E_m \ot E_1 \cong E_{m+1} \op E_{m-1}$ via an $SU(2)$-equivariant unitary.

We recall that $K_2 = \cc \cd \delta$ and for every $m \geq 2$ we have that
\[
K_m = \sum_{i = 0}^{m-2} L_n^{\ot i} \ot K_2 \ot L_n^{\ot (m-2-i)} .
\]
We also put $K_1 = K_0 := \{0\}$ and define $E_m = K_m^\perp \su L_n^{\ot m}$ for all $m \in \nn_0$. As in Definition \ref{d:sps}, we denote the identity operator on the Hilbert space $E_m$, with the symbol $1_m$.

%and that $E_m = K_m^\perp \su L_n^{\ot m}$ for all $m \geq 2$. Moreover, $E_1 = L_n$ and $E_0 = \cc$, and  $E_m \su E_{m-1} \ot E_1$ for all $m \geq 2$.
 %as can be seen from the inclusion $K_{m-1} \ot E_1 \su K_m$.

We recursively define a linear map $G_m : E_{m-1} \to K_{m+1}$ for each $m \in \nn$:
\begin{equation}\label{eq:Gm_rec}
G_1(1) := \de \, \, , \, \, \, 
G_m := G_{m-1} \ot 1 + (-1)^{(n+1)(m-1)} d_{m-1} \cd 1_{m-1} \ot G_1 \q \T{for } m\geq 2,
\end{equation}
where we are suppressing the inclusion $\io_{m-2,1} : E_{m-1} \to E_{m-2} \ot E_1$ and the obvious identification $\iota_{m-1,0} : E_{m-1} \stackrel{\cong}{\longrightarrow} E_{m-1} \ot E_0$.

\begin{lemma}\label{l:rightequiv}
Let $m \in \nn$. The linear map $G_m : E_{m-1} \to K_{m+1}$ is equivariant meaning that 
\[
\rho_n^{\ot (m+1)}(g) G_m = G_m \rho_n^{\ot (m-1)}(g) \q \mbox{for all } g \in SU(2) .
\]
\end{lemma}
\begin{proof}
The proof runs by induction on $m \in \nn$. The case where $m = 1$ holds since $\rho_n(g)^{\ot 2}(\de) = \de = G_1(1)$. Suppose now that the equivariance condition holds for some $m \in \nn$. For $\xi \in E_m$, the recursive definition of the maps $G_m$ in \eqref{eq:Gm_rec} implies that
\[
\begin{split}
\rho_n^{\ot (m+2)}(g) G_{m+1}(\xi) 
& = \rho_n^{\ot (m+2)}(g) (G_m \ot 1)(\xi) +(-1)^{(n+1)m} {d_m} \cd  \rho_n^{\ot (m+2)}(g) (\xi \ot \delta) \\
& = (G_m \ot 1)\rho_n^{\ot m}(g) (\xi )+(-1)^{(n+1)m} {d_m}   \cd \rho_n^{\ot m}(g) (\xi) \ot \de  \\
& = G_{m+1} \rho_n^{\ot m}(g)(\xi) .
\end{split}
\]
This proves the lemma.
\end{proof}

\begin{lemma}\label{l:leftfus}
Let $m \in \nn$. It holds that
\begin{enumerate}
\item $\binn{(G_m \ot 1)(\xi), \eta \ot \de} = (-1)^{(n+1)m+1}\frac{d_{m-1}}{d_1} \cd \inn{\xi,\eta}$ for all $\xi \in E_{m-1} \ot E_1$, $\eta \in E_m$; 
\item $\binn{ G_m(\xi), G_m(\eta) } = \mu_m \cd \inn{\xi,\eta}$ for all $\xi,\eta \in E_{m-1}$;
\item $\binn{ (G_m \ot 1)(\xi), G_{m+1}(\eta)} = 0$ for all $\xi \in E_{m-1} \ot E_1, \eta \in E_m$.
%\item $\binn{ (G_{m-1} \ot 1)({\red i_{m-2}(\xi)}), G_m(\eta)} = 0$ for all $\xi \in E_{m-2} \ot E_1, \eta \in E_{m-1}$.
\end{enumerate}
\end{lemma}
\begin{proof}
$(1)$: We focus on the case where $m \geq 2$. Let $\xi = \sum_{j = 0}^n \xi_j \ot e_j \in E_{m-1} \ot E_1$ and $\eta \in E_m$ be given. We compute that
\[
\begin{split}
& \binn{(G_m \ot 1)(\xi), \eta \ot \de} 
= \sum_{j = 0}^n \binn{ G_m(\xi_j) \ot e_j, \eta \ot \de} \\
& \q = \sum_{j = 0}^n \binn{(G_{m-1} \ot 1)(\xi_j) \ot e_j, \eta \ot \de}
+ (-1)^{(n+1)(m-1)} d_{m-1}  \cd \sum_{j = 0}^n \inn{\xi_j \ot \de \ot e_j, \eta \ot \de} \\
& \q = (-1)^{(n+1)(m-1)} d_{m-1} \cd \sum_{j = 0}^n \frac{(-1)^n}{n+1} \cd \inn{\xi_j \ot e_j \ot e_{n - j} \ot e_j, \eta \ot e_{n-j} \ot e_j}\\
& \q = (-1)^{(n+1)m+1}\frac{d_{m-1}}{d_1} \cd \inn{\xi,\eta} ,
\end{split}
\]
where the third identity follows from the structure of the vector $\de = \frac{1}{\sqrt{n+1}} \cd \sum_{j = 0}^n (-1)^j e_j \ot e_{n-j}$ and from the inclusion $\T{Im}(G_{m-1}) \su K_m = E_m^\perp$.

$(2)$: The proof runs by induction on $m \in \nn$. For $m = 1$, the result follows since $\inn{\de,\de} = 1$. Next, given $m \geq 1$ we assume that $(2)$ holds and for $\xi,\eta \in E_m$ we then compute that
\[
\begin{split}
\binn{ G_{m+1}(\xi), G_{m+1}(\eta) } & = \binn{ (G_m \ot 1)(\xi), (G_m \ot 1)(\eta)}
+ d_m^2 \cd \inn{\xi \ot \de, \eta \ot \de}  \\ 
& \q +(-1)^{(n+1)m} d_{m} \cd  \left( \binn{ (G_m \ot 1)(\xi), \eta \ot \de}
+ \binn{ \xi \ot \de, (G_m \ot 1)(\eta)}  \right) \\
& = \mu_m \cd \inn{\xi,\eta} + d_m^2 \cd \inn{\xi,\eta} + (-1)^{(n+1)m} d_{m}  \cd (-1)^{(n+1)m+1}\frac{d_{m-1}}{d_1} \cd 2 \inn{\xi,\eta} \\
& = \mu_m \cd \inn{\xi,\eta} + d_m^2 \cd \inn{\xi,\eta} - 2 \frac{d_m d_{m-1}}{d_1}  \cd \inn{\xi,\eta} \\
& = (d_m^2 - \mu_m) \cd \inn{\xi,\eta} = \mu_{m+1} \cd \inn{\xi,\eta} ,
\end{split}
\]
where the second identity follows from the induction hypothesis and $(1)$ and the fifth identity follows from \eqref{eq:mudee}.

$(3)$: Let $\xi \in E_{m-1} \ot E_1$ and $\eta \in E_m$ be given. Using $(1)$ and $(2)$ we compute that
\[
\begin{split}
\binn{ (G_m \ot 1)(\xi), G_{m+1}(\eta)} & = \binn{ (G_m \ot 1)(\xi), (G_m \ot 1)(\eta)} \\
& \q + (-1)^{(n+1)m} d_m \cd \binn{ (G_m \ot 1)(\xi), \eta \ot \de} \\
& = \mu_m \cd \inn{\xi,\eta} - \frac{d_m d_{m-1}}{d_1} \cd \inn{\xi,\eta}
= 0 .
\end{split}
\]
This proves the lemma.
\end{proof}

\begin{lemma}\label{l:decomp}
The vector space sum yields a unitary isomorphism of Hilbert spaces
\[
(K_m \ot E_1) \op G_m(E_{m-1}) \cong K_{m+1} ,
\]
for all $m \geq 1$.
\end{lemma}
\begin{proof}
For $m = 1$, the vector space decomposition follows immediately from the identities $G_1(E_0) = \cc \cd \delta =  K_2$ and $K_1 = \{0\}$.

Suppose thus that $m \geq 2$ and let $\xi \in K_{m+1}$ be given. Remark that 
\[
\begin{split}
K_{m+1} & = K_m \ot E_1 + E_1^{\ot (m-1)} \ot { K_2} = K_m \ot E_1 + K_{m-1} \ot { K_2} + E_{m-1} \ot { K_2} \\ 
& = K_m \ot E_1 + E_{m-1} \ot { K_2} .
\end{split}
\]
We may therefore choose $\eta \in K_m \ot E_1$ and $\ze \in E_{m-1}$ such that $\xi = \eta + \ze \ot \de$. Using \eqref{eq:Gm_rec} we then obtain that
\[
\begin{split}
\xi & = \eta + \frac{ (-1)^{(n+1)(m-1)}}{d_{m-1}} \cd \big( G_m(\ze) - (G_{m-1} \ot 1)(\ze) \big)
\end{split}
\]
Since $\T{Im}(G_{m-1}) \su K_m$ this proves the surjectivity claim.
 
To prove that the Hilbert space direct sum in question is isometrically isomorphic to $K_{m+1}$, we apply induction on $m \geq 1$. The case $m = 1$ has already been discussed, so suppose that the vector space sum yields an isometry for some $m \geq 1$ and let $\eta \in K_{m+1} \ot E_1$ and $\ze \in E_m$ be given. We need to show that $\binn{\eta, G_{m+1}(\ze)} = 0$. By the surjectivity part we may find $\xi \in K_m \ot E_1 \ot E_1$ and $\rho \in E_{m-1} \ot E_1$ such that $\eta = \xi + (G_m \ot 1)(\rho)$. By Lemma \ref{l:leftfus} part $(3)$, the induction hypothesis, and the fact that $K_m = E_m^\perp$, we then have the identities
\[
\begin{split}
\binn{ \eta, G_{m+1}(\ze)} 
& = \binn{\xi, G_{m+1}(\ze)} + \binn{ (G_m \ot 1)(\rho), G_{m+1}(\ze)}
= \binn{\xi, G_{m+1}(\ze)} \\
&  = \binn{\xi, (G_m \ot 1)(\ze)} + {(-1)^{(n+1)m} d_m}  \cd \binn{\xi,\ze \ot \de}
= 0 .
\end{split}
\]
This proves the lemma. \end{proof}

\begin{lemma}\label{l:dimension}
It holds that $\T{dim}(E_m)=d_m$ for all $m \in \nn_0$.
\end{lemma}
\begin{proof}
This is a consequence of Lemma \ref{l:decomp}, yielding the following identities of dimensions:
\[
\begin{split}
\T{dim}(E_{m+1}) & = (n+1)^{m+1} - \T{dim}(K_{m+1}) = (n+1)^{m+1} - (n+1) \cd \T{dim}(K_m) - \T{dim}(E_{m-1}) \\
& = (n+1) \cd \T{dim}(E_m) - \T{dim}(E_{m-1}).
\end{split}
\]
Since $d_0 = \T{dim}(E_0)$ and $d_1 = \T{dim}(E_1)$ and since the sequences $\{ d_m\}_{m = 0}^\infty$ and $\{ \T{dim}(E_m) \}_{m = 0}^\infty$ satisfy the same recursion formula, they must necessarily agree.
\end{proof}

\begin{rem}
Note that a subproduct system of Hilbert spaces $\lbrace E_m \rbrace_{m \in \nn_0}$ is called \emph{commutative} if the corresponding Fock space is a subspace of the symmetric Fock space on $E_1$ or, equivalently, if $E_m \subseteq E_1^{\otimes_S m}$ for all $m \in \nn_0$. It follows from Lemma \ref{l:dimension} that our subproduct systems are noncommutative for every $n>1$, as we have
$\dim(E_2) = (n+1)^2-1 > \binom{n+2}{2} = \dim((\cc^{n+1})^{\otimes_S 2})$.
\end{rem}

Lemma \ref{l:decomp} has the important consequence that the image of $G_m : E_{m - 1} \to K_{m + 1}$ is in fact equal to the intersection $K_{m + 1} \cap ( E_m \ot E_1)$. Moreover, Lemma \ref{l:leftfus} implies that the induced $SU(2)$-equivariant linear map
\begin{equation}\label{eq:V_m}
V_m :=  \frac{(-1)^{(n+1)(m-1)}}{\sqrt{\mu_m}} \cd G_m : E_{m-1} \to E_m \ot E_1
\end{equation}
is an isometry for all $m \geq 1$. We have therefore established the announced main result of this subsection:

\begin{prop}\label{p:decompright}
Let $m \in \nn$. The linear map
\[
\ma{cc}{\io_{m,1} & V_m} : E_{m + 1} \op E_{m - 1} \to E_m \ot E_1
\]
is an $SU(2)$-equivariant unitary isomorphism.
\end{prop}

\subsection{Decomposing tensor products by $E_1$ from the left}
The result of Proposition \ref{p:decompright} provides us with an $SU(2)$-equivariant unitary isomorphism $E_{m + 1} \op E_{m-1} \to E_1 \ot E_m$, for every $m \in \nn_0$, obtained by composing $\ma{cc}{\io_{m,1} & V_m}$ with the flip map $E_m \ot E_1 \to E_1 \ot E_m$. In this subsection we shall provide an alternative $SU(2)$-equivariant unitary isomorphism $E_{m+1} \op E_{m-1} \to E_1 \ot E_m$, where the relevant isometry $E_{m-1} \to E_1 \ot E_m$ is given by a recursive formula which is similar to \eqref{eq:Gm_rec}. This alternative $SU(2)$-equivariant unitary isomorphism will play an essential role in the rest of our work, as one of the building blocks for our proof of the $KK$-equivalence between the Toeplitz algebra and the complex numbers.

We define the linear maps $G_m' : E_{m-1} \to K_{m+1}$, $m \in \nn_0$, recursively by the formulae
\begin{equation}
\label{eq:Gprime_rec}
G_1'(1) := \de \, \, , \, \, \, 
G_m' := 1 \ot G_{m-1}' + (-1)^{(n+1)(m-1)} d_{m-1} \cd G_1' \ot 1_{m-1} \, \, , \, \,\,  m \geq 2 ,
\end{equation}
where the vector $\delta \in K_2$ and the constant $d_{m-1}$ are defined in \eqref{eq:delta_vect} and \eqref{eq:d_m_rec}.

Again, notice that we are suppressing the inclusion $\io_{1,m-2} : E_{m-1} \to E_1 \ot E_{m-2}$ (for $m \geq 2$) and the  obvious identification $\iota_{0,m-1} : E_{m-1} \stackrel{\cong}{\longrightarrow} E_{0} \ot E_{m-1}$. 

\begin{lemma}\label{l:leftequiv}
Let $m \in \nn$. The linear map $G_m' : E_{m-1} \to K_{m+1}$ is equivariant meaning that 
\[
\rho_n^{\ot (m+1)}(g) G_m' = G_m' \rho_n^{\ot (m-1)}(g) \q \mbox{for all } g \in SU(2) .
\]
\end{lemma}
\begin{proof}
The proof runs by induction on $m \in \nn$, using the same argument as in the proof of Lemma \ref{l:rightequiv}.
\end{proof}

\begin{lemma}\label{l:rightfus}
Let $m \in \nn$. We have the identities
\begin{enumerate}
\item $\binn{(1 \ot G_m')(\xi), \de \ot \eta} =  (-1)^{(n+1)m+1}\frac{d_{m-1}}{d_1}  \cd \inn{\xi,\eta}$ for all $\xi \in E_1 \ot E_{m-1}$, $\eta \in E_m$;
\item $\binn{ G_m'(\xi), G_m'(\eta) } = \mu_m \cd \inn{\xi,\eta}$ for all $\xi,\eta \in E_{m-1}$;
\item $\binn{ (1 \ot G_m')(\xi), G_{m+1}'(\eta)} = 0$ for all $\xi \in E_1 \ot E_{m-1}$, $\eta \in E_m$.
\end{enumerate}
\end{lemma}
\begin{proof}
The proof follows the proof of Lemma \ref{l:leftfus} \emph{verbatim}.
\end{proof}

\begin{lemma}\label{l:leftcomp}
For each $m \in \nn$, the vector space sum yields a unitary isomorphism of Hilbert spaces
\[
(E_1 \ot K_m) \op G_m'(E_{m-1}) \cong K_{m+1} .
\]
\end{lemma}
\begin{proof}
The proof is \emph{mutatis mutandis} the same as the proof of Lemma \ref{l:decomp}.
\end{proof}

In analogy with the previous subsection, we obtain from Lemma \ref{l:leftcomp} that the image of $G_m'$ agrees with the intersection $K_{m+1} \cap (E_1 \ot E_m)$ and, moreover, we see from Lemma \ref{l:rightfus} that the induced $SU(2)$-equivariant linear map
\begin{equation}\label{eq:V_mprime}
V_m' :=  \frac{(-1)^{(n+1)(m-1)}}{\sqrt{\mu_m}} \cd G_m' : E_{m-1} \to E_1 \ot E_m
\end{equation}
is an isometry for all $m \geq 1$. We announce the following:

\begin{prop}\label{p:decompleft}
Let $m \in \nn$. The linear map
\[
\ma{cc}{\io_{1,m} & V_m'} : E_{m + 1} \op E_{m - 1} \to E_1 \ot E_m
\]
is an $SU(2)$-equivariant unitary isomorphism.
\end{prop}

\subsection{Orthogonal decomposition of tensor products of representations}
As we saw in Lemma \ref{l:decomp} and Lemma \ref{l:leftcomp}, we may change the codomains of the linear maps defined in \eqref{eq:Gm_rec} and \eqref{eq:Gprime_rec} and instead consider the $SU(2)$-equivariant linear maps
\[
G_m : E_{m-1} \to E_m \ot E_1 \q \T{and} \q G_m' : E_{m-1} \to E_1 \ot E_m
\]
for all $m \in \nn$. These linear maps then satisfy the recursive relations
\begin{equation}\label{eq:recursiota}
\begin{split}
& (\io_{m-1,1} \ot 1) \ci G_m = (G_{m-1} \ot 1) \ci \io_{m-2,1} + (-1)^{(n+1) (m-1)} d_{m-1} \cd 1_{m-1} \ot G_1 \q \T{and} \\
& (1 \ot \io_{1,m-1}) \ci G_m' = (1 \ot G_{m-1}') \ci \io_{1,m-2}  + (-1)^{(n+1) (m-1)} d_{m-1} \cd G_1' \ot 1_{m-1} 
\end{split}
\end{equation}
for all $m \geq 2$. We recall that $G_1'(1) = G_1(1) = \de$, where the unit vector $\de \in K_2$ was introduced in \eqref{eq:delta_vect}. 

For every $k,m \in \nn_0$ we introduce the $SU(2)$-equivariant linear map
\begin{equation}\label{eq:defsigma}
\si_{k,m} : E_k \ot E_m \to E_{k+1} \ot E_{m+1} \q \si_{k,m} := (1_{k+1} \ot \io^*_{1,m})(G_{k+1} \ot 1_m) .
\end{equation}
For $k = -1$ or $m = -1$ we put $\si_{k,m} := 0 : \{0\} \to E_{k+1} \ot E_{m+1}$. These linear maps are going to play a key role in establishing the main result of this section, namely the fusion rules for our $SU(2)$-equivariant subproduct system as announced in Theorem \ref{t:fusionintro}. Before we can study these maps in more detail we need a few preliminary lemmas.

\begin{lemma}\label{l:iotagee}
Let $m \in \nn$. It holds that
\[
\begin{split}
G_m & = (-1)^{(n+1) (m-1)} d_{m-1} \cd (\io^*_{m-1,1} \ot 1)(1_{m-1} \ot G_1) \q \mbox{and} \\ 
G_m' & = (-1)^{(n+1) (m-1)} d_{m-1} \cd (1 \ot \io^*_{1,m-1})(G_1 \ot 1_{m-1}) .
\end{split}
\]
\end{lemma}
\begin{proof}
We focus on proving the claim for $G_m : E_{m-1} \to E_m \ot E_1$. To this end, we compute that
\[
\begin{split}
G_m & = (\io^*_{m-1,1} \io_{m-1,1} \ot 1) G_m \\ 
& = (\io^*_{m-1,1} \ot 1)(G_{m-1} \ot 1) \io_{m-2,1} + (-1)^{(n+1)(m-1)} d_{m-1} \cd (\io^*_{m-1,1} \ot 1)(1_{m-1} \ot G_1) \\
& = (-1)^{(n+1)(m-1)} d_{m-1} \cd (\io^*_{m-1,1} \ot 1)(1_{m-1} \ot G_1) ,
\end{split}
\]
where the last identity follows since $\T{Im}(G_{m-1}) = K_m \cap (E_{m-1} \ot E_1)$ and since $\io_{m-1,1} \io^*_{m-1,1} : E_{m-1} \ot E_1 \to E_{m-1} \ot E_1$ is the orthogonal projection onto the subspace $E_m = K_m^\perp$. 
\end{proof}
\begin{lemma}\label{lem:iota*}
Let $m \in \nn$. It holds that
\[
\begin{split}
\iota_{m-1,1}^* & =  (-1)^{(n+1)m +1}\frac{d_1}{d_{m-1}} \cd  (1_m \otimes G_1^* ) (G_m \otimes 1) : E_{m-1} \otimes E_1 \to E_m \\
\iota_{1,m-1}^* & = (-1)^{(n+1)m +1}\frac{d_1}{d_{m-1}} \cd ( (G_1')^* \otimes 1_m) (1 \otimes G_m'): E_1 \otimes E_{m-1} \to E_m .
\end{split}
\]
\end{lemma}
\begin{proof}
We focus on proving the claim for $\io_{m-1,1}^* : E_{m-1} \ot E_1 \to E_m$. Using Lemma \ref{l:rightfus} $(1)$ and Lemma \ref{l:iotagee} we obtain that
\[
\begin{split}
& (-1)^{(n+1)m +1}\frac{d_1}{d_{m-1}} \cd  (1_m \otimes G_1^* ) (G_m \otimes 1) \\
& \q = (-1)^n d_1 \cd (1_m \otimes G_1^* )(\io^*_{m-1,1} \ot 1 \ot 1)(1_{m-1} \ot G_1 \ot 1) \\
& \q = (-1)^n d_1 \cd \io^*_{m-1,1} ( 1_{m-1} \ot 1 \ot G_1^*)(1_{m-1} \ot G_1 \ot 1) 
= \io^*_{m-1,1} . \qedhere
\end{split}
\] 
\end{proof}

%As a consequence, a similar result can be stated for the projections $p_{m-1,1}$ and $p_{1,m-1}$.

\begin{lemma}\label{l:PP*}
%For every $m\geq 1$, let  $p_{m-1,1}: E_{m-1} \ot E_1 \to E_m$ and $p_{1,m-1}: E_1 \ot E_{m-1} \to E_m$ be defined as in \eqref{eq:proj_nm}. Then we have
Let $m \in \nn$. It holds that
\[
\begin{split}
p_{m-1,1} & = 1_{m-1} \ot 1 + (-1)^{(n+1)m +1}\frac{d_1}{d_{m-1}} \cd (G_{m-1} \otimes G_1^* )(\iota_{m-2,1} \otimes 1) \\
& \q : E_{m-1} \otimes E_1 \to E_{m-1} \ot E_1 \q \mbox{and} \\
p_{1,m-1} & = 1 \ot 1_{m-1} + (-1)^{(n+1)m +1}\frac{d_1}{d_{m-1}} ( G_1^* \otimes G_{m-1}')(1 \otimes \iota_{1,m-2})  \\
& \q : E_1 \otimes E_{m-1} \to E_1 \otimes E_{m-1} .
\end{split}
\]
\end{lemma}
\begin{proof}
We focus on the orthogonal projection $p_{m-1,1} : E_{m-1} \ot E_1 \to E_{m-1} \ot E_1$. Using Lemma \ref{l:rightfus} (1), Lemma \ref{lem:iota*}, and the recursive relation from \eqref{eq:recursiota} we compute that
\[
\begin{split}
p_{m-1,1} 
& = \io_{m-1,1} \io^*_{m-1,1} = (-1)^{(n+1)m + 1} \frac{d_1}{d_{m-1}} \cd (1_{m-1} \ot 1 \ot G_1^*)(\io_{m-1,1} \ot 1 \ot 1)(G_m \ot 1) \\
& = (-1)^{(n+1)m + 1} \frac{d_1}{d_{m-1}} \cd ( G_{m-1} \ot G_1^*)( \io_{m-2,1} \ot 1) \\
& \q + (-1)^{(n+1)m + 1} \frac{d_1}{d_{m-1}} \cd (-1)^{(n+1)(m-1)} d_{m-1} \cd (1_{m-1} \ot 1 \ot G_1^*)(1_{m-1} \ot G_1 \ot 1) \\
& = (-1)^{(n+1)m + 1} \frac{d_1}{d_{m-1}} \cd (G_{m-1} \ot G_1^*)( \io_{m-2,1} \ot 1) + 1_{m-1} \ot 1 . \qedhere
\end{split}
\]
\end{proof}
%\begin{lemma}
%\label{l:iotaGGprimev1}
%Let $\xi \in E_{k+l}$. Then we have 
%\[(1_{E_{k+1}} \otimes \iota_{1,l}^*)(G_{k+1}\otimes 1_{E_l})(\xi) = \frac{\lambda_k}{\lambda_l} (1_{E_{k+1}}\otimes G_{l}') (\xi) + (G_{k+1}\otimes 1_{E_l})(\xi).\] 
%\end{lemma}
%\begin{proof}
%\end{proof}
%\todo[inline]{How about writing 
%\[ (1_{E_{k+1}} \otimes \iota_{1,l}^*  - 1)(G_{k+1}\otimes 1_{E_l})(\xi) = \frac{\lambda_k}{\lambda_l} (1_{E_{k+1}}\otimes G_{l}') (\xi) \]
%for all $\xi \in (E_k \otimes E_l) \cap (E_{k+1} \ot E_{l-1})$ ?}

\begin{prop}\label{p:sigma*sigma}
Let $k,m \in \nn_0$. We have the identity
\begin{equation}\label{eq:almostnormal}
\begin{split}
\si_{k,m}^* \si_{k,m} 
& = \frac{d_k d_{k+m+1}}{d_1 d_m} \cd 1_k \ot 1_m 
+ \frac{d_k d_{m-1}}{d_{k-1} d_m} \cd \si_{k-1,m-1} \si_{k-1,m-1}^* \\
& \q : E_k \ot E_m \to E_k \ot E_m .
\end{split}
\end{equation}
\end{prop}
\begin{proof}
We focus on the case where $k,m \in \nn$. Using Lemma \ref{l:leftfus} and Lemma \ref{l:PP*} we see that
\[
\begin{split}
\si_{k,m}^* \si_{k,m} 
& = (G_{k+1} \ot 1_m)^* (1_{k+1} \ot p_{1,m})(G_{k+1} \ot 1_m) \\
& = \mu_{k+1} \cd 1_k \ot 1_m \\ 
& \q + (-1)^{(n+1)m + n} \frac{d_1}{d_m} \cd (G_{k+1} \ot 1_m)^*(1_{k+1} \ot G_1^* \ot G_m')(G_{k+1} \ot \io_{1,m-1}) .
\end{split}
\]
We continue by analysing the second term in this sum by applying Lemma \ref{l:iotagee} and the recursive relation from \eqref{eq:recursiota}:
\begin{equation}\label{eq:almostI}
\begin{split}
& (-1)^{(n+1)m + n} \frac{d_1}{d_m} \cd (G_{k+1} \ot 1_m)^*(1_{k+1} \ot G_1^* \ot G_m')(G_{k+1} \ot \io_{1,m-1}) \\
& \q = (-1)^{(n+1)(m + k) + n} \frac{d_1 d_k}{d_m} \cd (1_k \ot G_1^* \ot 1_m)(1_k \ot 1 \ot G_1^* \ot G_m') \\
& \qqq \qqq \qq \ci (\io_{k,1} \ot 1 \ot \io_{1,m-1})(G_{k+1} \ot 1_m) \\
& \q = (-1)^{(n+1)(m + k) + n} \frac{d_1 d_k}{d_m} \cd (1_k \ot G_1^* \ot 1_m)(G_k \ot G_1^* \ot G_m')(\io_{k-1,1}\ot \io_{1,m-1}) \\
& \qq + (-1)^{(n+1)m + n} \frac{d_1 d_k^2}{d_m} \cd (1_k \ot G_1^* \ot 1_m)(1_k \ot 1 \ot G_1^* \ot G_m')(1_k \ot G_1 \ot \io_{1,m-1}) .
\end{split}
\end{equation}
Using Lemma \ref{l:iotagee} and Lemma \ref{lem:iota*} we then obtain that the first term in the above sum is given by
\[
\begin{split}
& (-1)^{(n+1)(m + k) + n} \frac{d_1 d_k}{d_m} \cd (1_k \ot G_1^* \ot 1_m)(G_k \ot G_1^* \ot G_m')(\io_{k-1,1}\ot \io_{1,m-1}) \\
& \q = (-1)^{(n+1)(k+1)} \frac{d_k d_{m-1}}{d_m} \cd (1_k \ot \io_{1,m-1}^*)(G_k \ot G_1^* \ot 1_{m-1})(\io_{k-1,1}\ot \io_{1,m-1}) \\
& \q = \frac{d_k d_{m-1}}{d_m d_{k-1}} \cd \si_{k-1,m-1} \si_{k-1,m-1}^* ,
\end{split}
\]
corresponding to the second term in \eqref{eq:almostnormal} (in the case where $k,m \in \nn$). We continue with the remaining term in \eqref{eq:almostI} and apply Lemma \ref{l:rightfus}, Lemma \ref{lem:iota*}: 
\[
\begin{split}
& (-1)^{(n+1)m + n} \frac{d_1 d_k^2}{d_m} \cd (1_k \ot G_1^* \ot 1_m)(1_k \ot 1 \ot G_1^* \ot G_m')(1_k \ot G_1 \ot \io_{1,m-1}) \\
& \q = (-1)^{(n+1)m} \frac{d_k^2}{d_m} \cd (1_k \ot G_1^* \ot 1_m)(1_k \ot 1 \ot G_m')(1_k \ot \io_{1,m-1}) \\
& = - \frac{d_k^2 d_{m-1} }{d_1 d_m} \cd 1_k \ot 1_m .
\end{split}
\]
The result of the proposition now follows by an application of Lemma \ref{l:prop_d_con} in the case where $l = 0$, yielding that 
\[
\mu_{k+1} - \frac{ d_k^2 d_{m-1} }{d_1 d_m} = \frac{d_k d_{k+m+1}}{d_1 d_m} . \qedhere
\]
\end{proof}

The following lemmas contain further properties of the operators $\sigma_{k,m} : E_k \ot E_m \to E_{k+1} \ot E_{m+1}$, $k,m \in \nn_0$. For ease of notation, we omit the subscripts. 

\begin{lemma}\label{l:sigma_propI}
Let $k,m \in \nn_0$ and $j \in \nn$. We have the identity:
\[
\begin{split}
 \sigma^* \sigma^j 
& = \mu_{k + j}\cd \big(1 - \frac{d_k d_{m-1} }{d_{k+j} d_{m + j-1}} \big) \cd \sigma^{j-1} 
+ \frac{d_{m-1} d_{k+j-1}}{d_{k-1} d_{m+j-1}} \cd \sigma^j \sigma^* \\
& \q : E_k \ot E_m \to E_{k + j-1} \ot E_{m + j - 1}
\end{split}
\]
\end{lemma}
\begin{proof}
Applying Proposition \ref{p:sigma*sigma} we obtain by induction on $j \in \nn$ that
\[
\begin{split}
\si^* \si^j 
& = \frac{d_{k+j-1} ( d_{k + m + 2j -1} + d_{k + m + 2j - 3} \plp d_{k + m + 1})}{d_1 d_{m + j-1}} \si^{j-1}  \\
& \q + \frac{d_{k + j-1} d_{m - 1}}{d_{k -1} d_{m + j - 1}} \si^j \si^* .
\end{split}
\]
The result of the present lemma then follows by an application of Lemma \ref{l:prop_d_con}:
\[
\begin{split}
& \frac{d_{k+j-1} ( d_{k + m + 2j -1} + d_{k + m + 2j - 3} \plp d_{k + m + 1})}{d_1 d_{m + j-1}}  \\
& \q = \frac{d_{k+j-1} ( d_{k + j} d_{m + j-1} - d_k d_{m-1})}{d_1 d_{m + j-1}} = \mu_{k + j} \cd \big( 1 - \frac{d_k d_{m-1}}{d_{k+j} d_{m + j-1}} \big) . \qedhere
\end{split}
\]
\end{proof}

\begin{lemma}\label{l:sigma_propII}
Let $k,m \in \nn_0$ and $j \in \nn$. We have the identities: 
\[
\begin{split}
\sigma^* \io_{k,m} & = 0 : E_{k + m} \to E_{k-1} \ot E_{m-1} \q \mbox{and} \\ 
(\sigma^*)^j \sigma^j \io_{k,m} & = \prod_{i = 1}^j \mu_{k + i} \big(1 - \frac{d_k d_{m-1}}{d_{k+i} d_{m + i -1}} \big) \cd  \io_{k,m} 
: E_{k + m} \to E_k \ot E_m .
\end{split}
\]
\end{lemma}
\begin{proof}
By Lemma \ref{l:sigma_propI} it suffices to show that $\si_{k-1,m-1}^* \io_{k,m} = 0$. This is a triviality for $k = 0$ or $m = 0$ and for $k,m \in \nn$ we have that $\si_{k-1,m-1}^* \io_{k,m} = (G_k^* \ot 1_{m-1})(1_k \ot \io_{1,m-1})\io_{k,m} : E_{k+m} \to E_{k-1} \ot E_{m-1}$. However, by Lemma \ref{l:iotagee} this linear map is a scalar multiple of the inclusion $E_{k + m} \to E_{k-1} \ot E_1 \ot E_1 \ot E_{m-1}$ composed with $1_{k-1} \ot \inn{\de,\cd} \ot 1_{m-1}$. Since $E_{k-1} \ot D \ot E_{m-1}$ lies in the orthogonal complement of $E_{k + m} \su E_{k-1} \ot E_1 \ot E_1 \ot E_{m-1}$ we have proved the lemma.
\end{proof}

Our computations culminate in the following important result concerning the decomposition of the tensor product of two elements of our subproduct system of Hilbert spaces.
\begin{thm}\label{t:iso_fusion} 
Let $k,m \in \nn_0$ and put $l := \min\{k,m\}$. We have an $SU(2)$-equivariant unitary isomorphism
\[
W_{k,m} = \ma{cccc}{W_{k,m}^0 & W_{k,m}^1 & \ldots & W_{k,m}^l}: \bop_{j = 0}^l E_{k + m - 2j} \to E_k \ot E_m 
\]
defined component-wise by
\[
\begin{split}
W_{k,m}^j 
& = \prod_{i=1}^j \frac{1}{\sqrt{\mu_{ k-j+i}}}  \left( 1 - \frac{d_{ k-j} d_{ m-j-1}}{d_{ k-j+i} d_{ m-j+i-1}} \right)^{-1/2} \cd \si^j \io_{k-j,m-j} \\
& \q : E_{k + m - 2j} \to E_k \ot E_m 
\end{split}
\]
for all $j \in \big\{1,\ldots,l \big\}$ and $W_{k,m}^0 := \io_{k,m} : E_{k + m} \to E_k \ot E_m$.
\end{thm}
\begin{proof}
By Lemma \ref{l:sigma_propII} we have that $W_{k,m}^j : E_{k + m - 2j} \to E_k \ot E_m$ is an isometry for all $j \in \{0,1,\ldots,l\}$. Moreover, it follows from Lemma \ref{l:sigma_propII} that $(W_{k,m}^i)^* W_{k,m}^j = 0 : E_{k+m-2j} \to E_{k + m - 2i}$ whenever $0 \leq j < i \leq l$. These two observations establish that $W_{k,m} : \bop_{j = 0}^l E_{k + m - 2j} \to E_k \ot E_m$ is an isometry. The fact that $W_{k,m}$ is surjective now follows by dimension considerations since Lemma \ref{l:prop_d_con} implies that $d_k d_m = \sum_{j = 0}^l d_{k + m - 2l + 2j}$. The $SU(2)$-equivariance of $W_{k,m}$ is a consequence of the $SU(2)$-equivariance of the structure maps of our subproduct system  and the definition in \eqref{eq:defsigma} together with Lemma \ref{l:rightequiv}. 
\end{proof}

\section{Commutation relations for the Toeplitz algebra}\label{s:commreltoe}
Throughout this section we fix an $n \in \nn$ and consider the Toeplitz algebra coming from the irreducible representation $\rho_n : SU(2) \to U(L_n)$. We let $\{e_j \}_{j = 0}^n$ denote the orthonormal basis for $L_n$ introduced in \eqref{eq:orthobasis}. In particular, we have the associated Toeplitz operators
\[
T_j := T_{e_j} : F \to F \q j \in \{0,1,\ldots,n\} . 
\]
For each $j \in \{0,1,2,\ldots,n\}$ we also introduce the bounded operator $T'_j : F \to F$ defined by
\[
T_j'(\xi) := \io^*_{m,1}(\xi \ot e_j) \q \T{for all } \xi \in E_m . 
\]
In other words, $T'_j$ is the \emph{right} creation operator associated to the vector $e_j \in E_1 = L_n$.

We define the $SU(2)$-equivariant bounded operators $\io_L : F \to E_1 \ot F$ and $\io_R : F \to F \ot E_1$ by $\io_L(\xi) := \io_{1,m-1}(\xi)$ and $\io_R(\xi) := \io_{m-1,1}(\xi)$ for homogeneous elements $\xi \in E_m$ with $m \geq 1$ and for $\xi \in E_0$ we put $\io_L(\xi) = 0$ and $\io_R(\xi) = 0$. 

\begin{lemma}\label{lem:iota*TT*}
We have the identities
\[ 
\begin{split}
\io^*_L & = \sum_{j=0}^n { \inn{e_j,\cd} } \ot T_j : E_1 \ot F \to F \q \mbox{and} \\
\io^*_R & = \sum_{j=0}^n T_j' \ot { \inn{e_j,\cd}} : F \ot E_1 \to F .
\end{split}
\]
\end{lemma}
\begin{proof}
Let $\xi \in E_m$ and $i \in \{0,1,\ldots,n\}$ be given. We compute that
\[
\io^*_L(e_i \ot \xi)= \io^*_{1,m}(e_i \otimes \xi) = T_i(\xi) = \sum_{j = 0}^n ({ \inn{e_j,\cd}} \ot T_j)(e_i \ot \xi) . 
\]
The identity involving $\io_R^* : F \ot E_1 \to F$ is proved in the same way.
%which concludes the proof.
\end{proof}

We are now going to further analyse the structural properties of the $SU(2)$-equivariant isometries $V_m : E_{m-1} \to E_m \ot E_1$ and $V^\prime_m : E_{m-1} \to E_1 \ot E_m$ defined in \eqref{eq:V_m} and \eqref{eq:V_mprime}.

\begin{lemma}\label{l:leftright_Toe}
Let $m \in \nn$. For every $\xi \in E_{m-1}$, we have the identities
\[
\begin{split}
V_m'(\xi) & = \sqrt{ d_{m-1}/ d_m} \cd \sum_{j = 0}^n {(-1)^j } \cd e_j \ot T_{n-j}(\xi) \q \mbox{and} \\
V_m(\xi) & = \sqrt{ d_{m-1} / d_m } \cd \sum_{j = 0}^n {(-1)^{n-j} } \cd T'_{n-j}(\xi) \ot e_j .
\end{split}
\]
\end{lemma}
\begin{proof}
By definition of $V_m' : E_{m-1} \to E_1 \ot E_m$ and by Lemma \ref{l:iotagee} it holds that
\[
\begin{split}
V_m'(\xi) & = \frac{(-1)^{(n+1)(m-1)}}{\sqrt{\mu_m}} \cd G_m'(\xi) = \frac{d_{m-1}}{\sqrt{\mu_m}}(1 \ot \io^*_{1,m-1})(\de \ot \xi) \\
& = \frac{d_{m-1}}{\sqrt{\mu_m \cd (n+1)}}  \cd \sum_{j = 0}^n (-1)^j \cd e_j \ot T_{n-j}(\xi) 
=  \sqrt{\frac{d_{m-1}}{d_m}} \cd \sum_{j = 0}^n (-1)^j \cd e_j \ot T_{n-j}(\xi)  ,
\end{split}
\]
where the last equality follows from the definition of the constant $\mu_m$ in \eqref{eq:mucon}.

The proof of the second identity follows \emph{mutatis mutandis} the proof of the first one. 
\end{proof}

\subsection{The dimension operator}
Recall that $\Falg \su F$ denotes the algebraic Fock space defined as the vector space direct sum of the vector spaces $E_m$, $m \in \nn_0$. 

\begin{dfn}\label{d:dim_op}
We define the \emph{dimension operator} $D : \T{Dom}(D) \to F$ as the closure of the unbounded operator $\mathcal{D}: \Falg \to F$, given by $\C {D}(\xi) = d_m \cd \xi$ for $\xi \in E_m$. 
\end{dfn}

Observe that the dimension operator is positive and invertible and that the inverse $D^{-1} : F \to F$ is an $SU(2)$-equivariant compact operator. In particular $D^{-1} \in \B T$.

In the special case of the fundamental representation, the operator $D$ equals $N+1$, where $N$ is the number operator.

We further define the $SU(2)$-equivariant bounded positive invertible operator 
\begin{equation}\label{eq:defphi}
\Phi: F \to F \q \Phi \xi= \frac{d_m}{d_{m+1}} \xi \, \, \, \T{ for all } \xi \in E_m .
\end{equation}

%{\color{purple} Note that for $n=1$ we have that 
%\[\Phi = - \frac{1}{3} (C(N) +3),\]
%where $C(N)$ is the Cayley transform. }

\begin{lemma}\label{l:phitoe}
The bounded invertible operator $\Phi : F \to F$ belongs to the Toeplitz algebra $\B T$.
\end{lemma}
\begin{proof}
Let $\ga_n \in (0,1]$ be the constant defined in Lemma \ref{l:quocon}. Since $\Phi -\gamma_n \cd 1_F$ is a compact operator on $F$ and $\B K(F)\subseteq \B T$, we obtain the result of the lemma. 
\end{proof}

%{\color{purple}
%Finally, define the $SU(2)$-equivariant bounded operator $\Phi': F \to F$ by the assignment $\Phi' (\xi) = \frac{d_{m-1}}{d_m} \xi$ for all $\xi \in E_m$ with $m \in \nn$ and $\Phi'(\xi) = 0$ for $\xi \in E_0$. Then the bounded operator $\Phi' (1-Q_0) + Q_0$ is positive and invertible and belongs to the Toeplitz algebra $\B T$ as well.}
%

We define the $SU(2)$-equivariant isometries $V_R : F \to F \ot E_1$ and $V_L : F \to E_1 \ot F$ by
\[
V_R(\xi) = V_m(\xi) \q \T{and} \q V_L(\xi) = V_m'(\xi)
\]
for all $\xi \in E_{m-1} \su F$. We may then restate the result of Lemma \ref{l:leftright_Toe} as follows:

\begin{prop}\label{p:isometries}
For every $\xi \in F$, we have the identities
\[
\begin{split}
V_L(\xi) & = \sum_{j = 0}^n (-1)^j \cd e_j \ot T_{n-j} \Phi^{1/2} (\xi) \q \mbox{and} \\
V_R(\xi) & = \sum_{j = 0}^n (-1)^{n-j} \cd T_{n-j}' \Phi^{1/2} (\xi) \ot e_j .
\end{split}
\]
\end{prop}

\subsection{Commutation relations}
We now present the commutation relations for our Toeplitz algebras in the general case of an irreducible representation $\rho_n : SU(2) \to U(L_n)$ for $n \geq 1$. These commutation relations can be used to recover the commutation relations in Theorem \ref{t:toe3sph} in the case of the fundamental representation. For the time being we do not know whether there are any further relations in the Toeplitz algebra $\B T(\rho_n,L_n)$.

We start out by remarking that 
\begin{equation}\label{eq:sumcomp}
\sum_{i = 0}^n T_i T_i^* = \io_L^* \io_L = 1_F - Q_0 . 
\end{equation} 

\begin{thm}
Let $n \in \nn$, and consider the irreducible representation $\rho_n : SU(2) \to U(L_n)$. Then the Toeplitz operators $T_i$, with $i=0,\dotsc,n$ satisfy the following commutation relations:
\begin{align}
 \label{eq:non_self} & \sum_{i = 0}^n (-1)^i T_i T_{n-i} =0, \\
  \label{eq:Toe_adj} & 
  T_i^*T_j = \delta_{ij} { \cd 1_F} +  (-1)^{i+j+1} \big((n+1) { \cd 1_F} - \Phi^{-1} \big) T_{n-i} T_{n-j}^* \\
  \label{eq:Toe_sph2} & \sum_{i=0}^n T_i^*T_i = \Phi^{-1}.
\end{align}
\end{thm}
\begin{proof} 
The relation in \eqref{eq:non_self} follows from our computation of the determinant in Proposition \ref{p:irrdet} (cf. \cite[\S 10]{ShSo09}).

We now move on to establishing the relation in \eqref{eq:Toe_adj}. Consider $i, j \in \{0,1,\ldots,n\}$. By Proposition \ref{p:decompleft} we have that $\io_L \io_L^* +  V_L V_L^* = 1_F \ot 1 : F \ot E_1 \to F \ot E_1$ and hence that
\[
T_i^* T_j = (\inn{e_i , \cd } \ot 1_F ) \io_L \io_L^* (e_j \ot 1_F) 
= \de_{ij} \cd 1_F - (\inn{e_i , \cd } \ot 1_F ) V_L V_L^* (e_j \ot 1_F) .
\]
Then, upon using Proposition \ref{p:isometries} we obtain that $( \inn{e_i,\cd} \ot 1_F) V_L = (-1)^i T_{n-i} \Phi^{1/2}$ and hence that
\[
T_i^* T_j = \de_{ij} \cd 1_F + (-1)^{i+j+1} T_{n - i} \Phi T_{n-j}^* .
\]
The relation in \eqref{eq:Toe_adj} now follows by the definition of $\Phi : F \to F$ from \eqref{eq:defphi} upon noting that $T_{n-i}(E_m) \su E_{m+1}$ and $d_1 - d_{m+2} / d_{m+1} = d_m / d_{m+1}$ for all $m \in \nn_0$.

We are now left with proving the relation in \eqref{eq:Toe_sph2}. From the identities in \eqref{eq:sumcomp} and \eqref{eq:Toe_adj} we obtain that
\[
\begin{split}
\sum_{i=0}^n T_i^*T_i  
& =  (n+1) { \cd 1_F} -   \big( (n+1) { \cd 1_F} - \Phi^{-1} \big) \sum_{i=0}^n T_{n-i} T_{n-i}^* \\
& { =  (n+1) { \cd 1_F} -   \big( (n+1) { \cd 1_F} - \Phi^{-1} \big) (1_F - Q_0)
= \Phi^{-1} .}
\end{split}
\]
This ends the proof of the theorem.
\end{proof}

%\todo{I deleted the examples and just commented on the fundamental representation in the beginning of the subsection instead. I didn't find the example for the coadjoint representation so illuminating since it is just a special case of the theorem.}

%\todo{I also commented out the final remark on the relations in the Cuntz--Pimsner algebra since it was rather incomplete.}
%We denote by $S_0, \dotsc, S_n$ the  images of the operators $T_0,\dotsc,T_n$ in the
%quotient. They satisfy the sphere relation \begin{equation}
%\sum_{i=0}^n S_i S_i^* = 1
%\end{equation}
%\todo[inline]{What else happens in the quotient?What does the fact that $\varphi-\gamma_n \cdot 1_F$ is compact imply?}

\section{A quasi-homomorphism from the Toeplitz algebra to the complex numbers}\label{s:quasi}
Let $n \in \nn$ be given and consider the irreducible representation $\rho_n : SU(2) \to U(L_n)$. We denote the corresponding Toeplitz algebra by $\B T \su \B L(F)$, where $F = \bop_{m = 0}^\infty E_m$ denotes the Fock space. In this section we start relating the $K$-theory of the Toeplitz algebra to the $K$-theory of the complex numbers by constructing an $SU(2)$-equivariant quasi-homomorphism $(\psi_+,\psi_-)$ from $\B T$ to $\B C$. 

Both of the $*$-homomorphisms $\psi_+$ and $\psi_-$ act on the Hilbert space direct sum $F \op F$ and we define $\psi_+ : \B T \to \B L(F \op F)$ by $\psi_+(x) := x \op x$ for all $x \in \B T$. The construction of $\psi_- : \B T \to \B L(F \op F)$ uses the representation theoretic considerations from Section \ref{s:fusion}. 

Recall that $V_R : F \to F \ot E_1$ denotes the $SU(2)$-equivariant isometry defined by
\[
V_R(\xi) := V_{m+1}(\xi) = {\frac{(-1)^{(n+1)m}}{\sqrt{\mu_{m+1}}}} \cd G_{m+1}(\xi) \in { E_{m+1}} \ot E_1 \su F \ot E_1
\] 
for every homogeneous $\xi \in E_m \su F$, $m \in \nn_0$. Moreover, we have the $SU(2)$-equivariant linear map $\io_R : F \to F \ot E_1$ defined by
\[
\io_R(\xi) := \io_{m-1,1}(\xi) \in E_{m-1} \ot E_1 \su F \ot E_1
\]
for every homogeneous $\xi \in E_m \su F$, $m \in \nn$ and $\io_R(\xi) = 0$ for $\xi \in E_0$. It follows from Proposition \ref{p:decompright} that the $SU(2)$-equivariant linear map 
\begin{equation}\label{eq:WR}
W_R : F \ot E_1 \to F \op F \q W_R = { \ma{c}{\io_R^* \\ V_R^*} }
\end{equation}
is an isometry and that the image agrees with the subspace $F_+ \op F \su F \op F$. We may thus define the $*$-homomorphism
\[
\psi_- : \B T \to \B L( F \op F) \q \psi_-(x) := W_R (x \ot 1) W_R^* .
\]

We also recall that we have the $SU(2)$-equivariant linear map $\io_L : F \to E_1 \ot F$  defined by the formula
\[
\io_L(\xi) := \io_{1,m-1}(\xi) \in E_1 \ot E_{m-1} \su E_1 \ot F
\]
for homogeneous elements $\xi \in E_m \su F$ with $m \in \nn$ ad $\io_L(\xi) = 0$ for $\xi \in E_0$. 

We announce the following result:

\begin{prop}\label{p:quasi}
The pair of $*$-homomorphisms $(\psi_+,\psi_-)$ defines an $SU(2)$-equivariant quasi-homomorphism from $\B T$ to $\cc$ and hence a class $[\psi_+,\psi_-] \in KK_0^{SU(2)}(\B T,\cc)$.
\end{prop}
\begin{proof} The $SU(2)$-equivariance of the two $*$-homomorphisms follows from the $SU(2)$-equivariance of $W_R : F \ot E_1 \to F \op F$ together with the observation that the action of $SU(2)$ on the Toeplitz algebra is obtained via conjugation with the corresponding action on the Fock space $F$, see Lemma \ref{l:toepaction}.

For each $x \in \B T$ we have to show that the difference $\psi_+(x) - \psi_-(x) = x \op x - W_R(x \ot 1) W_R^*$ is a compact operator on $F \op F$. Since $\B T$ is generated as a $C^*$-algebra by the operators $T_j^* : F \to F$, $j \in \{0,1,\ldots,n\}$ together with the unit $1_F : F \to F$, it suffices to prove compactness when $x \in \B T$ agrees with one of these operators. For the case of the unit $1_F : F \to F$ we have that $1_F \op 1_F - W_R W_R^*$ agrees with the orthogonal projection onto the one-dimensional subspace $(F_+ \op F)^\perp \cong {\B C}$ so we focus on the operator $T_j^* : F \to F$ for a fixed $j \in \{0,1,\ldots,n\}$. We compute that
\[
W_R(T_j^* \ot 1) W_R^* = \ma{cc}{ \io_R^* (T_j^* \ot 1) \io_R & \io_R^*(T_j^* \ot 1)V_R \\ V_R^*(T_j^* \ot 1) \io_R & V_R^* (T_j^* \ot 1) V_R} .
\]
Applying the identities $(T_j^* \ot 1) \io_R = \io_R T_j^*$, $V_R^* \io_R = 0$ (see Proposition \ref{p:decompright}), and using the fact that $\io_R^* \io_R$ is the orthogonal projection onto $F_+ \su F$, we obtain that
\[
W_R(T_j^* \ot 1) W_R^* \sim \ma{cc}{ T_j^* & \io_R^*(T_j^* \ot 1)V_R \\ 0 & V_R^* (T_j^* \ot 1) V_R}
\]
modulo compact operators. Now, by Proposition \ref{p:comdif} here below in Subsection \ref{ss:norm} we have that the operator $(T_j^* \ot 1) V_R$ agrees with $V_R T_j^*$ modulo compact operators. But this implies the result of this proposition, using that $V_R^* V_R = 1_F$ and $\io_R^* V_R = 0$.
\end{proof}

We are eventually going to show that the Toeplitz algebra $\B T$ is $KK$-equivalent to $\B C$ and the class $[\psi_+,\psi_-] \in KK_0^{SU(2)}(\B T,\B C)$ provides us with one of the two relevant morphisms. The other morphism is given by the unital inclusion $i : \cc \to \B T$, which defines a class $[i] \in KK_0^{SU(2)}(\cc,\B T)$.
% and thus induces maps $i^*: KK(\B T, C) \to KK(\cc, C)$ and $i_*: KK(A, \B T) \to KK( A, \cc)$, for fixed $A$ and $C$.}

\begin{prop}\label{p:easyprod}
The interior Kasparov product $ [i] \hot_{\B T} [\psi_+,\psi_-]$ agrees with the unit ${\bf 1}_\cc \in KK_0^{SU(2)}(\cc,\cc)$.
\end{prop}
\begin{proof}
The interior Kasparov product $ [i] \hot_{\B T} [\psi_+,\psi_-]$ is represented by the $SU(2)$-equivariant quasi-homomorphism $(\psi_+ \ci i, \psi_- \ci i)$. The $*$-homomorphism $\psi_+ \ci i : \cc \to \B L(F \op F)$ is unital whereas $(\psi_- \ci i)(1) = W_RW^*_R : F \op F \to F \op F$. Since $1_{F \op F} - W_RW^*_R : F \op F \to F \op F$ is the orthogonal projection onto the one-dimensional subspace $\cc  \omega \op \{0\} \su F \op F$, this proves the proposition.
\end{proof}

\subsection{Compactness of commutators}\label{ss:norm}
In this subsection we are providing the remaining ingredient for the proof of Proposition \ref{p:quasi}. More precisely, we shall see in Proposition \ref{p:comdif} that the difference $V_R T_j^* - (T_j^* \ot 1) V_R : F \to F \ot E_1$ is indeed a compact operator.

\begin{lemma}\label{l:diffvee}
For each $m \geq 2$ we have the identity
\[
\io^*_{1,m-2} (1 \ot V_{m-1})^* (\io_{1,m-1} \ot 1) V_m =  \left( 1 - \frac{1}{d_{m-1}^2} \right)^{1/2} \cd 1_{m-1} .
\]
\end{lemma}
\begin{proof}
Using Lemma \ref{l:iotagee} and \eqref{eq:V_m} we see that 
\begin{equation}\label{eq:veestar}
(1 \ot V_{m-1})^* 
= \frac{ (-1)^{(n+1)m}}{\sqrt{\mu_{m-1}}} \cd (1 \ot G_{m-1}^*) 
= \frac{d_{m-2}}{\sqrt{\mu_{m-1}}} \cd (1 \ot 1_{m-2} \ot G_1^*)(1 \ot \io_{m-2,1} \ot 1) .
\end{equation}
Next, we have the structural identity 
\[
(1 \ot \io_{m-2,1})\io_{1,m-1} = (\io_{1,m-2} \ot 1)\io_{m-1,1} : E_m \to E_1 \ot E_{m-2} \ot E_1 ,
\]
which combined with \eqref{eq:veestar} yields that
\[
\begin{split}
& \io^*_{1,m-2} (1 \ot V_{m-1})^* (\io_{1,m-1} \ot 1) V_m \\
& \q = \frac{d_{m-2}}{\sqrt{\mu_{m-1}}} \cd \io^*_{1,m-2} (1 \ot 1_{m-2} \ot G_1^*)(\io_{1,m-2} \ot 1 \ot 1)(\io_{m-1,1} \ot 1) V_m \\
& \q = \frac{d_{m-2}}{\sqrt{\mu_{m-1}}} \cd (1_{m-1} \ot G_1^*)(\io_{m-1,1} \ot 1) V_m .
\end{split}
\]
Using Lemma \ref{l:prop_d_con} and \eqref{eq:veestar} one more time together with the fact that $V_m : E_{m-1} \to E_m \ot E_1$ is an isometry we then get that
\[
\begin{split}
& \frac{d_{m-2}}{\sqrt{\mu_{m-1}}} \cd (1_{m-1} \ot G_1^*)(\io_{m-1,1} \ot 1) V_m \\
& \q = \frac{d_{m-2} \cd \sqrt{\mu_m}}{\sqrt{\mu_{m-1}} \cd d_{m-1}} \cd V_m^* V_m 
= \frac{\sqrt{d_{m-2} d_m}}{d_{m-1}} \cd 1_{m-1} = \left( 1 - \frac{1}{d_{m-1}^2} \right)^{1/2} \cd 1_{m-1}. \qedhere
\end{split}
\]
\end{proof}

\begin{prop}\label{p:comdif}
The difference
\[
(T_j^* \ot 1) V_R - V_R T_j^* : F \to F \ot E_1
\]
is a compact operator.
\end{prop}
\begin{proof}
Since $T_j^* = (\inn{e_j, \cd} \ot 1_F)\io_L : F \to F$, it is enough to show that the difference
\[
( \io_L \ot 1 ) V_R - (1 \ot V_R ) \io_L : F \to E_1 \ot F \ot E_1
\]
is a compact operator.

Notice first that the Hilbert space $E_{m-1} \su F$ is finite-dimensional for each $m \in \nn$ and that both $(\io_L \ot 1 ) V_R$ and $(1 \ot V_R) \io_L$ map $E_{m-1}$ into $E_1 \ot E_{m-1} \ot E_1$. The corresponding restrictions are given by $(\io_{1,m-1} \ot 1) V_m$ and $(1 \ot V_{m-1}) \io_{1,m-2} : E_{m-1} \to E_1 \ot E_{m-1} \ot E_1$. It therefore suffices to show that the sequence of operator norms
\[
\big\{ \| {(\io_{1,m-1} \ot 1) V_m - (1 \ot V_{m-1}) \io_{1,m-2}} \| \big\}_{m = 1}^\infty
\]
converges to zero.

Let $m \geq 2$. Using Lemma \ref{l:diffvee} together with the fact that $(\io_{1,m-1} \ot 1) V_m$ and $(1 \ot V_{m-1}) \io_{1,m-2}$ are isometries, we obtain that
\[
\begin{split}
& \big( (\io_{1,m-1} \ot 1)V_m  - (1 \ot V_{m-1})\io_{1,m-2} \big)^* \big( (\io_{1,m-1} \ot 1)V_m - (1 \ot V_{m-1})\io_{1,m-2} \big) \\
& \q = 2 \left( 1 - { \left( 1 - \frac{1}{d_{m-1}^2} \right)^{1/2}} \right) \cd 1_{m-1} .
\end{split}
\] 
which implies that 
\begin{equation}\label{eq:normdiff}
\| (\io_{1,m-1} \ot 1) V_m - (1 \ot V_{m-1}) \io_{1,m-2} \| 
= \sqrt{2} \cd \Big( 1 - { \left( 1 - \frac{1}{d_{m-1}^2} \right)^{1/2}} \Big)^{1/2} .
\end{equation}

The result of the lemma now follows since the sequence $\{ 1/d_{m-1}^2 \}_{m = 1}^\infty$ converges to zero (using again the global assumption that $n \geq 1$). 
\end{proof}

In fact, we can do slightly better than the above proposition:

\begin{lemma}\label{l:commdim}
Let $p \in [0,1]$. It holds that
\[
(D^p \ot 1)\big( (T_j^* \ot 1) V_R - V_R T_j^*\big) D^{1-p} : \Falg \to F \ot E_1
\]
extends to a bounded operator.
\end{lemma}
\begin{proof}
We first remark that the unbounded operator $(D^p \ot 1)\big( (T_j^* \ot 1) V_R - V_R T_j^*\big) D^{1-p} : \Falg \to F \ot E_1$ maps the subspace $E_m$ into $E_m \ot E_1$ for each $m \in \nn_0$. It therefore suffices to show that the supremum over $m \in \nn_0$ of the corresponding operator norms is finite.

Let $m \in \nn$ be given. We compute that
\[
\begin{split}
& (D^p \ot 1)\big( (T_j^* \ot 1) V_R - V_R T_j^* \big) D^{1-p} \vert_{E_m} \\
& \q = d_m \cd \big( (\inn{e_j,\cd} \ot 1_m \ot 1)(\io_{1,m} \ot 1)V_{m+1} - (\inn{e_j, \cd} \ot V_m)\io_{1,m-1} \big) \\
& \q = d_m \cd (\inn{e_j,\cd} \ot 1_m \ot 1) \big( (\io_{1,m} \ot 1)V_{m+1} - (1 \ot V_m)\io_{1,m-1} \big) .
\end{split}
\]
The result of the present lemma then follows from \eqref{eq:normdiff} by noting that
\[
\begin{split}
& d_m^2 \cd \| (\io_{1,m} \ot 1)V_{m+1} - (1 \ot V_m)\io_{1,m-1} \|^2 \\
& \q = 2 d_m^2 \cd (1 - \sqrt{1 - 1/d_m^2}) \leq 2 . \qedhere
\end{split}
\]
\end{proof}

\section{The $K$-theory of the Toeplitz algebra}\label{s:KKequiv}
Recall from Section \ref{s:quasi} that we have an $SU(2)$-equivariant isometry $W_R : F \ot E_1 \to F \op F$ (cf. \eqref{eq:WR}), which we use 
to define the $*$-homomorphism
\[
\psi_- : \B T \to \B L( F \op F) \q \psi_-(x) := W_R (x \ot 1) W^*_R .
\]
We clearly also have the $*$-homomorphism $\psi_+ : \B T \to \B L(F \op F)$, $\psi_+(x) := x \op x$.

We saw in Proposition \ref{p:quasi} that the pair { $(\psi_+,\psi_-)$} is an $SU(2)$-equivariant quasi-homomorphism form $\B T$ to $\B C$ and we therefore have a class $[\psi_+,\psi_-] \in KK_0^{SU(2)}(\B T,\B C)$. We moreover saw in Proposition \ref{p:easyprod} that the interior Kasparov product $[i] \hot_{\B T} [\psi_+,\psi_-] \in KK_0^{SU(2)}(\B C,\B C)$ agrees with the unit ${\bf 1}_{\B C}$, where we recall that $[i] \in KK_0^{SU(2)}(\B C,\B T)$ is the class associated with the unital inclusion $i : \B C \to \B T$.

In this section we are going to prove the following main result:

\begin{thm}\label{t:KKequiv}
The interior Kasparov product $[\psi_+,\psi_-] { \hot_{\B C}} [i]$ agrees with the unit ${\bf 1}_{\B T} \in KK_0^{SU(2)}(\B T,\B T)$. In particular, we have that $\B T$ and $\B C$ are $KK$-equivalent in an $SU(2)$-equivariant way.
\end{thm}

We let $F \hot \B T$ denote the standard module over $\B T$, defined as the exterior tensor product of the Fock space $F$ and the Toeplitz algebra $\B T$ viewed as a right Hilbert $C^*$-module over itself. The standard module becomes an $SU(2)$-Hilbert-$C^*$-module via the diagonal representation of $SU(2)$ on $F \hot \B T$ given explicitly by
\[
g ( \xi \ot T_\eta ) := g(\xi) \ot T_{g(\eta)}
\]
for every $g \in SU(2)$, $\xi \in F$ and $\eta \in E_k$. 

We remark that the interior Kasparov product $[\psi_+,\psi_-] { \hot_{\B C}} [i]$ is represented by the $SU(2)$-equivariant quasi-homomorphism
$( \psi_+ \ot 1_{\B T} , \psi_- \ot 1_{\B T})$, where $\psi_+ \ot 1_{\B T} : \B T \to \B L\big( (F \op F) \hot \B T \big)$ and $\psi_- \ot 1_{\B T} : \B T \to \B L\big( (F\op F) \hot \B T\big)$ are $SU(2)$-equivariant $*$-homomorphisms.

We let $M_{\B T} : \B T \to \B L(\B T)$ denote the $SU(2)$-equivariant $*$-homomorphism obtained by letting the Toeplitz algebra act as bounded adjointable operators on itself via left-multiplication. Recall moreover that $Q_0 : F \to F$ is the orthogonal projection onto the vacuum subspace $E_0 \su F$. %We lift $Q_0$ to an orthogonal projection on $F \op F$ by putting $Q_0^T := Q_0 \op 0 : F \op F \to F \op F$.

Our proof of Theorem \ref{t:KKequiv} amounts to showing that the $SU(2)$-equivariant quasi-homomorphism $(\psi_+ \ot 1_{\B T}, \psi_- \ot 1_{\B T})$ is homotopic to the $SU(2)$-equivariant quasi-homomorphism $(\psi_- \ot 1_{\B T} + { (Q_0  \op 0) \ot M_{\B T} }, \psi_- \ot 1_{\B T})$. Indeed, we would then obtain the following identities inside $KK_0^{SU(2)}(\B T,\B T)$: 
\[
[\psi_+,\psi_-] { \hot_{\B C}} [i] = [\psi_+ \ot 1_{\B T}, \psi_- \ot 1_{\B T}] 
= [\psi_- \ot 1_{\B T} + { (Q_0 \op 0) } \ot M_{\B T}, \psi_- \ot 1_{\B T}]
= {\bf 1}_{\B T} .
\]

The proof of the $SU(2)$-equivariant homotopy
\[
(\psi_+ \ot 1_{\B T}, \psi_- \ot 1_{\B T}) \sim_h (\psi_- \ot 1_{\B T} + { (Q_0 \op 0) }\ot M_{\B T}, \psi_- \ot 1_{\B T})
\]
is divided into three steps and occupies the remainder of this section.

It will sometimes be convenient to view the standard module $F \hot \B T$ as a closed subspace of bounded operators from $F$ to the Hilbert space tensor product $F \hot F$. Indeed, for every $\xi \in F$ and $x \in \B T$ we have the bounded operator
\[
\xi \ot x : F \to F \hot F \q (\xi \ot x)(\eta) := \xi \ot x(\eta) 
\]
and $F \hot \B T$ does in fact agree with the smallest closed subspace of $\B L(F, F \hot F)$ containing the bounded operators of the form $\xi \ot x$ for all $\xi \in F$ and $x \in \B T$. The inner product on $F \hot \B T$ then agrees with the operation
\[
\inn{\xi,\eta} := \xi^* \cd \eta \q \xi, \eta \in F \hot \B T
\]
using only products and adjoints of bounded operators. Moreover, the right action of $\B T$ on $F \hot \B T$ is simply induced by the composition of bounded operators $\B L(F,F \hot F)$ and $\B L(F)$. Any bounded operator $T : F \hot F \to F \hot F$ acts on the operator space $\B L(F,F\hot F)$ via the composition of bounded operators in $\B L(F \hot F)$ and $\B L(F,F \hot F)$. In this fashion, the unital $C^*$-algebra of bounded adjointable operators on $F \hot \B T$ identifies with the unital $C^*$-subalgebra of $\B L(F \hot F)$ consisting of those bounded operators $T : F \hot F \to F \hot F$ with the property that both $T$ and $T^*$ preserves the closed subspace $F \hot \B T \su \B L(F,F \hot F)$. To wit,
\[
\B L(F \hot \B T) \cong \big\{ T \in \B L(F \hot F) \mid T \cd ( F \hot \B T) \,  , \, \, T^* \cd (F \hot \B T) \su F \hot \B T \big\} . 
\]

\subsection{Intertwining representations of the Toeplitz algebra}
Before we can construct our homotopy we need some preliminaries, explaining better the relationship between the $SU(2)$-equivariant $*$-homomorphisms 
$\psi_+ \ot 1_{\B T}$ and $\psi_- \ot 1_{\B T} + Q_0^T \ot M_{\B T} : \B T \to \B L\big( (F \op F)  \hot \B T\big)$.

We are in this respect particularly interested in the $SU(2)$-equivariant bounded operator
\[
W: { (F \hot F)^{\op 2} \to (F \hot F)^{\op 2}}
\]
defined as the composition 
\[
\xymatrix{ { (F \hot F)^{\op 2}} \ar[rrr]_-{\ma{c}{ \io_R^* \ot 1_F \\ V_R^* \ot 1_F}^* } &&& (F \ot E_1) \hot F \cong F \hot (E_1 \ot F) 
\ar[rr]_-{\ma{c}{ 1_F \ot \iota^*_L \\ 1_F \ot V_L^* }  }  && { (F \hot F)^{\op 2}} .} 
\]
 We express this bounded operator in the following matrix form:
\begin{equation}\label{eq:douumatrix}
W = \ma{cc}{v^{TT} & v^{TB} \\
v^{BT} & v^{BB} } = 
\ma{cc}{ (1 \ot \io^*_L)(\io_R \ot 1) & (1 \ot \io^*_L)(V_R \ot 1) \\ 
(1 \ot V_L^*)(\io_R \ot 1) & (1 \ot V_L^*)(V_R \ot 1) } ,
\end{equation}
where all the entries belong to $\B L(F \hot F)$.

%\todo[inline]{I deleted the top-bottom notation. When reading the paper again I found it more confusing than convenient.}

We moreover let $\Si : F \hot F \to F \hot F$ denote the flip map $\Si(\xi \ot \eta) = \eta \ot \xi$ and remark that $\Si$ is an $SU(2)$-equivariant unitary operator.

Using Proposition \ref{p:decompright} and Proposition \ref{p:decompleft} we see that the $SU(2)$-equivariant operators
\[
\begin{split}
& W_R:= \ma{c}{ \io_R^* \\
V_R^* } : F \ot E_1 \to F \oplus F \q \T{and} \\
& W_L:= \ma{c}{ \io_L^*\\
V_L^* } : E_1 \ot F \to F \oplus F .
\end{split}
\]
are isometric with $W_R W_R^*$ and $W_L W_L^*$ both being the orthogonal projection onto $F_+ \op F$. It moreover holds that
\[
W = (1_F \ot W_L) (W_R^* \ot 1_F) \in \B L\big( { (F \hot F) \op (F \hot F)}  \big) .
\]

\begin{lemma}\label{l:conjdouu}
The $SU(2)$-equivariant operator $W$ is a partial isometry with 
\[
1 - W W^* = \ma{cc}{1_F \ot Q_0 & 0 \\ 0 & 0} \quad \mbox{and} \quad
1 - W^*W = \ma{cc}{Q_0 \ot 1_F & 0 \\ 0 & 0 }
\]
Moreover, it holds that
\[
W^* (\psi_+(x) \ot 1_F) W = \psi_-(x) \ot 1_F
\]
for all $x \in \B T$.
\end{lemma}
\begin{proof}
The first claim follows immediately from the above remarks and the computations
\[
\begin{split}
W W^* & = (1_F \ot W_L) (W_R^* \ot 1_F)(W_R \ot 1_F)(1_F \ot W_L^*) = 1_F \ot W_L W_L^* \q \T{and} \\
W^* W & = (W_R \ot 1_F)(1_F \ot W_L^*)(1_F \ot W_L)(W_R^* \ot 1_F) = W_R W_R^* \ot 1_F .
\end{split}
\]
Let now $x \in \B T$ be given. The second claim follows from the computations
\[
\begin{split}
W^* (\psi_+(x) \ot 1_F) W 
& = (W_R \ot 1_F)(1_F \ot W_L^*)(x \ot 1_{F \op F})(1_F \ot W_L)(W_R^* \ot 1_F) \\
& = (W_R \ot 1_F)(x \ot 1 \ot 1_F)(W_R^* \ot 1_F) = \psi_-(x) \ot 1_F  ,
\end{split}
\]
using that $W_L : F \ot E_1 \to F \op F$ is an isometry.
\end{proof}

\begin{lemma}\label{l:intertwine}
The operator 
\[
H_0 := -W + {\ma{cc}{ \Si (Q_0 \ot 1_F) & 0 \\ 0 & 0 } }
\in \B L\big( { (F \hot F) \op (F \hot F)}\big)
\]
is an $SU(2)$-equivariant unitary operator and we have the identity
\[
H_0^* (\psi_+(x) \ot 1_F) H_0 
= \psi_-(x) \ot 1_F + { \ma{cc}{Q_0 \ot x & 0 \\ 0 & 0} } \in \B L\big( { (F \hot F) \op (F \hot F)} \big)
\]
for all $x \in \B T$.
\end{lemma}
\begin{proof}
The fact that $H_0$ is a unitary operator follows by noting that both $W$ and $\ma{cc}{\Si(Q_0 \ot 1_F) & 0 \\ 0 & 0}$ are partial isometries satisfying that
\[
W W^* + {\ma{cc}{\Si (Q_0 \ot 1_F) (Q_0 \ot 1_F) \Si & 0 \\ 0 & 0 } } = 1 
= W^* W + { \ma{cc}{(Q_0 \ot 1_F)  \Si \Si (Q_0 \ot 1_F) & 0 \\ 0 & 0}}.
\]
Since all the involved operators are $SU(2)$-equivariant it holds that $H_0$ is $SU(2)$-equivariant as well.

Let now $x \in \B T$ be given. Using that $W_R : F \ot E_1 \to F \op F$ is an isometry together with the definitions of the involved operators we compute that
\[
\begin{split}
( \psi_+(x) \ot 1_F ) H_0 
& = - ( x \ot 1_{F \op F} ) W + \ma{cc}{ (x \ot 1_F) \Si (Q_0 \ot 1_F) & 0 \\ 0 & 0 } \\
& = - ( x \ot 1_{F \op F} ) (1_F \ot W_L)(W_R^* \ot 1_F) + \ma{cc}{ \Si (Q_0 \ot x) & 0 \\ 0 & 0 } \\
& = - W (\psi_-(x) \ot 1_F) + { \ma{cc}{ \Si (Q_0 \ot x) & 0 \\ 0 & 0 }} .
\end{split}
\]
This computation and the first part of the present proof imply the intertwining identity stated in the lemma.
\end{proof}

Let us apply the notation $j: \B T \to \B L(F)$ for the inclusion $\B T \su \B L(F)$ so that $j$ becomes a unital $*$-homomorphism. The above lemma then shows that the two $SU(2)$-equivariant $*$-homomorphisms 
\[
\psi_+ \ot 1_F \, \, \T{and} \, \, \, \psi_- \ot 1_F + { (Q_0 \op 0) \ot j}
: \B T \to \B L\big( (F \hot F) \op (F \hot F) \big)
\]
are unitarily equivalent via the $SU(2)$-equivariant unitary operator $H_0 \in \B L\big( (F \hot F) \op (F \hot F) \big)$. We emphasise that $H_0$ does \emph{not} define a bounded adjointable operator on $(F \hot \B T) \op (F \hot \B T)$ (because of the part containing the flip map). The two $*$-homomorphisms
\[
\psi_+ \ot 1_{\B T} \, \, \T{and} \, \, \, { \psi_- \ot 1_{\B T} + (Q_0 \op 0 ) \ot M_{\B T} } 
: \B T \to \B L\big( (F \hot \B T) \op (F \hot \B T) \big)
\]
are therefore most likely \emph{not} unitarily equivalent. 

In any case, we now start analysing the unitary operator $H_0 \in \B L\big( (F \hot F) \op (F \hot F) \big)$ in more details, paying particular attention to the partial isometry $W \in \B L\big( (F \hot F) \op (F \hot F) \big)$.

Recall that the invertible element $\Phi \in \B T$ was introduced in \eqref{eq:defphi}.

\begin{lemma}\label{l:WinM2}
The partial isometry $W$ defines a bounded adjointable operator on $(F \hot \B T) \op (F \hot \B T)$. In fact, we explicitly have that 
\[
W = \ma{cc}{v^{TT} & v^{TB} \\
v^{BT} & v^{BB} } = 
\sum_{j = 0}^n \ma{cc}{ (T_j')^* \ot T_j & (-1)^{n-j} T'_{n-j} \Phi^{1/2} \ot T_j \\ 
(-1)^j (T_j')^* \ot \Phi^{1/2} T^*_{n-j} & (-1)^n T'_{n - j} \Phi^{1/2} \ot \Phi^{1/2} T_{n-j}^*} .
\]
\end{lemma}
\begin{proof}
This follows from Lemma \ref{l:phitoe} and the matrix description of $W$ from \eqref{eq:douumatrix} together with the formulae provided in Lemma \ref{lem:iota*TT*} and Proposition \ref{p:isometries}. 
\end{proof}

Remark that it follows from Lemma \ref{l:WinM2} that
\begin{equation}\label{eq:adjointTB}
v^{BT} = (\Phi^{-1/2} \ot \Phi^{1/2}) \cd (v^{TB})^* \q \T{and} \q
v^{BB} = (-1)^n (1_F \ot \Phi^{1/2}) \cd (v^{TT})^* \cd (\Phi^{1/2} \ot 1_{\B T}) .
\end{equation}

For later use, we now relate the bounded operator $v^{TB} : F \hot F \to F \hot F$ to the bounded operators $\si_{k,m} : E_k \ot E_m \to E_{k+1} \ot E_{m+1}$ introduced in \eqref{eq:defsigma} for $k,m \in \nn_0$.

\begin{lemma}\label{l:sigteebot}
We have the identity
\[
v^{TB}(\xi) = \frac{(-1)^{(n+1)k}}{\sqrt{\mu_{k+1}}} \cd \si_{k,m}(\xi) = \frac{(-1)^{(n+1) k} \sqrt{ n+1}}{ \sqrt{d_k d_{k+1}}} \cd \si_{k,m}(\xi) .
\]
for all $\xi \in E_k \ot E_m$.
\end{lemma}
\begin{proof}
This follows immediately from the definition of the involved operators, see \eqref{eq:V_m}, \eqref{eq:defsigma} and \eqref{eq:douumatrix}. Recall also from \eqref{eq:mucon} that $\mu_{k+1} = (d_k d_{k+1})/d_1$ for all $k \in \nn_0$. 
\end{proof}

\begin{prop}\label{p:compactdouu}
For every $x \in \B T$ we have that
\[
[\psi_+(x) \ot 1_{\B T}, W ] \in M_2( \B K \hot \B T) .
\]
\end{prop}
\begin{proof}
Let $x \in \B T$ be given. We know from Proposition \ref{p:quasi} that the difference
\[
\psi_-(x) - \psi_+(x) : F \op F \to F \op F
\]
is a compact operator. Notice also that it follows from Lemma \ref{l:conjdouu} that $W W^* ( \psi_+(x) \ot 1_{\B T} ) = ( \psi_+(x) \ot 1_{\B T} ) W W^*$. Using these facts together with one more application of Lemma \ref{l:conjdouu} and Lemma \ref{l:WinM2} we may compute modulo compact operators in the following way: 
\[
\begin{split}
[\psi_+(x) \ot 1_{\B T}, W ] 
& \sim (\psi_+(x) \ot 1_{\B T}) W  - W (\psi_-(x) \ot 1_{\B T}) \\
& = (\psi_+(x) \ot 1_{\B T}) W  - W W^* (\psi_+(x) \ot 1_{\B T}) W
= 0 .
\end{split}
\]
This proves the present proposition.
\end{proof}

We now present a more refined estimate on the commutator between the generator $T_j^* : F \to F$ and the intertwining partial isometry $W \in M_2( \B L(F \hot \B T))$.

\begin{prop}\label{p:DpcommIII}
Let $p \in [0,1]$ and $j \in \{ 0, 1, \dotsc, n  \}$. The unbounded operators
\[ 
\begin{split}
& (D^p \ot 1_{\cc^2 \ot \B T}) [ \psi_+(T_j^*) \ot 1_{\B T},W] (D^{1-p} \ot 1_{\cc^2 \ot \B T}) \q \mbox{and} \\
& (D^p \ot 1_{\cc^2 \ot \B T}) [ \psi_+(T_j^*) \ot 1_{\B T},W^*] (D^{1-p} \ot 1_{\cc^2 \ot \B T}) 
: (\Falg \ot \cc^2 \ot \B T) \to (F \ot \cc^2) \hot \B T 
\end{split}
\]
both extend to elements in $M_2\big(\B L (F \hot \B T) \big)$.
\end{prop}
\begin{proof}
We start with the claim regarding the commutator with $W : (F \op F) \hot \B T \to (F \op F) \hot \B T$. By the identity in \eqref{eq:douumatrix} and the fact that $(T_j^* \ot 1)\io_R = \io_R T_j^*$ we have that
\begin{equation}\label{eq:commtj*}
\begin{split}
& [ \psi_+(T_j^*) \ot 1_{\B T},W] = 
\ma{cc}{ 0 & (1_F \ot \io_L^*)\Big( \big( (T_j^* \ot 1)V_R - V_R T_j^*\big) \ot 1_{\B T}\Big) \\
0 & (1_F \ot V_L^*)\Big( \big( (T_j^* \ot 1)V_R - V_R T_j^*\big) \ot 1_{\B T}\Big). }
\end{split}
\end{equation}
Now, from Lemma \ref{lem:iota*TT*} and Proposition \ref{p:isometries} we obtain that the bounded operators
\[
1_F \ot \io_L^* \, \, \T{and} \, \, \, 1_F \ot V_L^* 
: F \hot (E_1 \ot F) \to F \hot  F
\]
both define elements in $\B L\big( (F \ot E_1) \hot \B T, F \hot \B T\big)$. It therefore suffices to show that
\[
(D^p \ot 1) ( (T_j^* \ot 1)V_R - V_R T_j^* ) D^{1-p} : \Falg \to  F \ot E_1
\]
extends to a bounded operator. But this was already proved in Lemma \ref{l:commdim}. 

We continue with the claim regarding the commutator with $W^* : (F \op F) \hot \B T \to (F \op F) \hot \B T$. We are going to suppress the extra ``$\ot 1_{\cc^2 \ot \B T}$'' from the notation, e.g. writing $D^p$ instead of $D^p \ot 1_{\cc^2 \ot \B T}$. Notice first that the unbounded operator
\[
D^r W^* D^{-r} : (\Falg \ot \cc^2 \ot \B T) \to (F \ot \cc^2) \hot \B T
\]
extends to a bounded adjointable operators on $(F \op F) \hot \B T$ for all $r \in \rr$. To see this, we remark that 
\[
\begin{split}
D^r \io_R^* (D^{-r} \ot 1)(\xi) & = \io_R^* (\Phi^{-r} \ot 1)(\xi) \q \T{and} \\
D^r V_R^* (D^{-r} \ot 1)(\xi) & = \Phi^r V_R^*(\xi)
\end{split}
\]
for all $\xi \in \Falg \ot E_1$ and hence, upon using \eqref{eq:douumatrix}, Lemma \ref{lem:iota*TT*} and Proposition \ref{p:isometries}, we obtain that $D^r W^* D^{-r}$ extends to the bounded adjointable operator
\[
\ma{cc}{ (v^{TT})^* (\Phi^{-r} \ot 1_{\B T}) & (v^{BT})^* (\Phi^{-r} \ot 1_{\B T}) \\
(\Phi^r \ot 1_{\B T}) (v^{TB})^* & (\Phi^r \ot 1_{\B T}) (v^{BB})^* } \in \B L\big( (F \op F) \hot \B T \big) .
\] 
Next, remark that $T_j^* W W^* = WW^* T_j^*$ since $1 - WW^* =  (1_F \ot Q_0) \op 0$. Then, for every $\xi \in \Falg \ot \cc^2 \ot \B T$ we have that
\[
\begin{split}
D^p [ T_j^*,W^*] D^{1-p}(\xi) 
& = (1 - W^* W) D^p T_j^* W^* D^{1-p}(\xi)
+ D^p W^* W T_j^* W^* D^{1-p} - D^p W^* T_j^* D^{1-p}(\xi) \\
& = (1 - W^* W) D^p T_j^* W^* D^{1-p}(\xi) \\
& \q + D^p W^* D^{-p} \cd ( D^p [W,T_j^*] D^{1-p}) \cd  D^{p-1} W^* D^{1-p}(\xi) .
\end{split}
\]
Each of the terms in this sum extends to a bounded adjointable operator on $(F \op F) \hot \B T$. For the first term this follows since $1 - W^* W =  (Q_0 \ot 1_{\B T}) \op 0$, and for the second term this follows from the argument carried out earlier in this proof. 
\end{proof}

%%%%%%%%%%%%%%%%%%%%%%%%%%%%%%%%%%%%%%%%%%%%%%%%%%%%%%%%%%%%%%%%%%8<%%%%%%%%%%%%%%%%%%%%%%%%%%%%%%%%%%%%%%%%%%%%%%%%%%%%%%%%%%%%%%

%%%%%%%%%%%%%%%%%%%%%%%%%%%%%%%%%%%%%%%%%%%%%%%%%%%%%%%%%%%%%%%%%%8<%%%%%%%%%%%%%%%%%%%%%%%%%%%%%%%%%%%%%%%%%%%%%%%%%%%%%%%%%%%%%%

%%%%%%%%%%%%%%%%%%%%%%%%%%%%%%%%%%%%%%%%%%%%%%%%%%%%%%%%%%%%%%%%%%8<%%%%%%%%%%%%%%%%%%%%%%%%%%%%%%%%%%%%%%%%%%%%%%%%%%%%%%%%%%%%%%

%%%%%%%%%%%%%%%%%%%%%%%%%%%%%%%%%%%%%%%%%%%%%%%%%%%%%%%%%%%%%%%%%%8<%%%%%%%%%%%%%%%%%%%%%%%%%%%%%%%%%%%%%%%%%%%%%%%%%%%%%%%%%%%%%%

\subsection{Decomposition of the standard module}
We define the Hilbert space $G \su F \hot F$ as the closure of the subspace
\begin{equation}\label{eq:classsub}
\T{span}\big\{ \io_{k,m}(\xi) \mid k,m \in \nn_0 \, , \, \, \xi \in E_{k + m} \big\} \su F \hot F .
\end{equation}
Our strategy for constructing our homotopy is to work separately on the closed subspace 
\[
( G \op \{ 0 \} ) \su (F \hot F) \op (F \hot F)
\]
and the orthogonal complement $G^\perp \op (F \hot F)$. In fact, it turns out that our homotopy behaves very much like the classical $U(1)$-case (cf. \cite[Section 4]{Pim97}) on the closed subspace $G \op \{ 0 \}$ whereas the remaining part (taking place on $G^\perp \op (F \hot F)$) requires a separate argument. We therefore need to understand the orthogonal projection $\Pi : F \hot F \to F \hot F$ onto the orthogonal complement $G^\perp \su F \hot F$. We show here below that $\Pi$ defines a bounded adjointable operator on $F \hot \B T$ and that the commutator $[ x \ot 1_{\B T}, \Pi ]$ is a compact operator for every $x \in \B T$.  

It turns out that the orthogonal projection $\Pi : F \hot F \to F \hot F$ is related to the bounded operator $v^{TB} : F \hot F \to F \hot F$ and a proper description of this relationship requires a better understanding of the polar decomposition of $v^{TB} : F \hot F \to F \hot F$. 

We are going to apply Proposition \ref{p:polar} with $X := F \hot \B T$ and $y := v^{TB} : F \hot \B T \to F \hot \B T$. The relevant dense submodule is the algebraic tensor product $\s X := \Falg \ot \B T$. We fix $j \in \{0,1,\ldots,n\}$ and put $x_j := T_j^* \ot 1_{\B T} : X \to X$. We immediately remark that
\[
x_j(\s X) \, , \, \, x^*_j(\s X) \, , \, \, y^*(\s X) \su \s X,
\]
where the last inclusion follows from Lemma \ref{l:WinM2}.

We now compute the bounded adjointable operator $y^* y = (v^{TB})^* v^{TB} : F \hot \B T \to F \hot \B T$. To this end, we apply Theorem \ref{t:iso_fusion} and define positive invertible operators
\[
\Ga_{k,m} : E_k \ot E_m \to E_k \ot E_m  \q k,m \in \nn_0
\]
using the prescription
\begin{equation}\label{eq:gammadef}
\Ga_{k,m} (\sigma^j \io_{k-j,m-j} \xi) := \left( 1- \frac{d_{k-j}d_{m-j-1}}{d_{k+1}d_m}\right)( \sigma^j \io_{k-j,m-j} \xi) ,
\end{equation}
for all $\xi \in E_{k+m-2j}$ and $0\leq j \leq k,m$. A quick computation shows that 
\[
\| \Ga_{k,m} \| = 1 - \frac{d_{k-l} d_{m-l-1}}{d_{k+1} d_m} \leq 1
\]
where $l = \min\{k,m\}$ and we therefore obtain a positive bounded operator
\[
\Ga : F \hot F \to F \hot F \q \Ga|_{E_k \ot E_m} := \Ga_{k,m} 
\]
with dense image. We are here applying our standing convention that $n \in \nn$ so that the irreducible representation $\rho_n : SU(2) \to U(L_n)$ is non-trivial.

\begin{lemma}\label{l:gamma}
We have the identity
\[
(v^{TB})^* v^{TB} = \Ga : F \hot F \to F \hot F .
\]
\end{lemma}
\begin{proof}
Let $k,m \in \nn_0$, let $j \in \{0,1,\ldots, \min\{k,m\} \}$ and let $\xi \in E_{k+m-2j}$ be given. Using Theorem \ref{t:iso_fusion} it suffices to show that
\[
(v^{TB})^* v^{TB}( \si^j \io_{k -j, m -j} \xi) = \Ga ( \si^j \io_{k - j,m-j} \xi) .
\]
However, by Lemma \ref{l:sigteebot} we have that
\[
(v^{TB})^* v^{TB}(\eta) = \frac{1}{\mu_{k+1}} \si_{k,m}^* \si_{k,m}(\eta) 
\]
for every $\eta \in E_k \ot E_m$. Hence we see from Lemma \ref{l:sigma_propI} and Lemma \ref{l:sigma_propII} that
\[
\begin{split}
(v^{TB})^* v^{TB}( \si^j \io_{k -j, m -j} \xi)
& = \frac{1}{\mu_{k+1}} \si^* \si^{j+1} \io_{k-j,m-j} \xi
= \big( 1 - \frac{d_{k-j} d_{m-j-1}}{d_{k+1} d_m } \big) \cd \si^j \io_{k-j,m-j} \xi \\
& = \Ga_{k,m}(\si^j \io_{k-j,m-j} \xi ) .
\end{split}
\]
This proves the present lemma.
\end{proof}

It follows from Lemma \ref{l:WinM2} and Lemma \ref{l:gamma} that the positive bounded operator $\Ga : F \hot F \to F \hot F$ defines a positive bounded adjointable operator $\Ga : F \hot \B T \to F \hot \B T$.

\begin{lemma}\label{l:gammadense}
The image of the positive bounded adjointable operator $\Ga : F \hot \B T \to F \hot \B T$ contains the dense submodule $\s X = \Falg \ot \B T \su F \hot \B T$. 
\end{lemma}
\begin{proof}
Let us fix a $k \in \nn_0$ and show that $E_k \ot \B T \su \T{Im}(\Ga)$. We recall that $Q_k : F \to F$ denotes the orthogonal projection with image $E_k \su F$. It then follows from the definition of $\Ga : F \hot F \to F \hot F$ that the bounded operator
\[
\Ga (Q_k \ot 1_F) + (1_F - Q_k) \ot 1_F = (Q_k \ot 1_F) \Ga (Q_k \ot 1_F) + (1_F - Q_k) \ot 1_F : F \hot F \to F \hot F 
\]
has a bounded inverse. Indeed, for every $m \in \nn_0$ it holds that $\Ga_{k,m} : E_k \ot E_m \to E_k \ot E_m$ is invertible with $\| \Ga_{k,m}^{-1} \| \leq ( 1 - \frac{d_k }{d_{k + 1}} )^{-1}$. Now, since the invertible  bounded operator $\Ga (Q_k \ot 1_F) + (1_F - Q_k) \ot 1_F \in \B L(F \hot F)$ belongs to the unital $C^*$-subalgebra $\B L(F \hot \B T) \su \B L(F \hot F)$ we obtain that the bounded adjointable operator $\Ga (Q_k \ot 1_{\B T}) + (1_F - Q_k) \ot 1_{\B T} : F \hot \B T \to F \hot \B T$ is invertible as well. But this shows that 
\[
E_k \ot \B T = \T{Im}( \Ga (Q_k \ot 1_{\B T}) ) \su \T{Im}(\Ga) . \qedhere
\] 
\end{proof}

As a consequence of Lemma \ref{l:gammadense} we obtain that $\Ga^{-1} : \T{Im}(\Ga) \to F \hot \B T$ is an unbounded positive and regular operator on the Hilbert $C^*$-module $F \hot \B T$. Moreover, we see from the proof of Lemma \ref{l:gammadense} that the domain of $\Ga^{-1}$ contains the algebraic tensor product $\s X = \Falg \ot \B T$.

%We consider the densely defined unbounded operator $v^{TB}\Ga^{-1/2} : \T{Im}(\Ga^{1/2}) \to F \hot \B T$. Using Lemma \ref{l:WinM2} we see that $(v^{TB})^* : F \hot \B T \to F \hot \B T$ preserves the algebraic tensor product $\Falg \ot \B T$ and hence that $v^{TB} \Ga^{-1/2} : \T{Im}(\Ga^{1/2}) \to F \hot \B T$ has a densely defined adjoint. Indeed, we have the inclusion
%\[
%\Ga^{-1/2}(v^{TB})^* \su (v^{TB} \Ga^{-1/2} )^* ,
%\]
%where the domain of the left hand side is dense in $F \hot \B T$ since it contains the algebraic tensor product $\Falg \ot \B T$.

\begin{lemma}\label{l:image}
The closure of $v^{TB} \Ga^{-1/2} : \T{Im}(\Ga^{1/2}) \to F \hot \B T$ is a bounded adjointable isometry $\Te : F \hot \B T \to F \hot \B T$ and the associated orthogonal projection $\Te \Te^* \in \B L(F \hot \B T)$ agrees with $\Pi \in \B L(F \hot F)$ (upon suppressing the inclusion $\B L(F \hot \B T) \su \B L(F \hot F)$).
\end{lemma}
\begin{proof}
Since $\Ga = (v^{TB})^* v^{TB}$ and the domains of both $v^{TB} \Ga^{-1/2}$ and $(v^{TB} \Ga^{-1/2})^*$ contain the dense submodule $\Falg \ot \B T$ we obtain that $\Te : F \hot \B T \to F \hot \B T$ is a well-defined bounded adjointable isometry. We now compute the image of $\Te$ considered as a bounded operator on $F \hot F$. This image clearly agrees with the closure of the image of $v^{TB}$ restricted to the algebraic tensor product $\Falg \ot \Falg$. For each $k,m \in \nn_0$ we know that the image of $v^{TB}|_{E_k \ot E_m} : E_k \ot E_m \to E_{k+1} \ot E_{m+1}$ agrees with the image of $\si_{k,m} : E_k \ot E_m \to E_{k+1} \ot E_{m+1}$. However, from Theorem \ref{t:iso_fusion} we see that the image of $\si_{k,m} : E_k \ot E_m \to E_{k+1} \ot E_{m+1}$ agrees with the orthogonal complement of $\io_{k+1,m+1}( E_{k+m + 2}) \su E_{k+1} \ot E_{m+1}$. These observations entail that the image of $\Te : F \hot F \to F \hot F$ agrees with 
\[
\T{span}\{ \io_{k,m}( \xi ) \mid k,m \in \nn_0 \, , \, \, \xi \in E_{k+m} \}^\perp \su F \hot F .
\]
In other words, we have that $\T{Im}(\Te : F \hot F \to F \hot F) = G^\perp = \T{Im}(\Pi)$. This proves the present lemma.
\end{proof}

Let us introduce the compact operator
\[
K := D^{-1} \ot 1_{\B T} : F \hot \B T \to F \hot \B T ,
\]
recalling that the dimension operator $D : \T{Dom}(D) \to F$ was introduced in Definition \ref{d:dim_op}.

Recall that $x_j := T_j^* \ot 1_{\B T}$ and $y := v^{TB} : F \hot \B T \to F \hot \B T$.

\begin{lemma}\label{l:decomptheta}
There exist bounded adjointable operators $L, \ov{L}, M, \ov{M} : F \hot \B T \to F \hot \B T$ such that
\[
K^{1/2} L K^{1/2} = [x_j,y] = M K \q \mbox{and} \q  K^{1/2} \ov{L} K^{1/2} = [x_j,y^*] = K \ov{M} . 
\]
\end{lemma}
\begin{proof}
This follows immediately from Proposition \ref{p:DpcommIII}. Firstly, $L$ and $\ov{L}$ are the bounded adjointable extensions of $(D^{1/2} \ot 1_{\B T})[T_j^* \ot 1_{\B T},v^{TB}](D^{1/2} \ot 1_{\B T})$ and $(D^{1/2} \ot 1_{\B T})[T_j^* \ot 1_{\B T},(v^{TB})^*](D^{1/2} \ot 1_{\B T})$, respectively. Secondly, $M$ and $\ov{M}$ are the bounded adjointable extensions of $[T_j^* \ot 1_{\B T},v^{TB}](D \ot 1_{\B T})$ and $(D \ot 1_{\B T})[T_j^* \ot 1_{\B T},(v^{TB})^*]$, respectively. It is here understood that all the involved unbounded operators are defined on the algebraic tensor product $\Falg \ot \B T$ even though this is not properly reflected in the notation.
\end{proof}

In order to apply Proposition \ref{p:polar} we still have to control the growth of the resolvent $R_\la := (\la + (v^{TB})^* v^{TB})^{-1}$ as the parameter $\la > 0$ approaches zero.

\begin{lemma}\label{l:DGamma_inv_sqrt}
It holds that $(D^{-1} \ot 1_{\B T}) (v^{TB})^* v^{TB}  = (v^{TB})^* v^{TB} (D^{-1} \ot 1_{\B T})$. Moreover, there exists a constant $C > 0$ such that
\[
\| (D^{-1} \ot 1_{\B T}) R_\la \| \leq C \q \mbox{and} \q \| (D^{-1/2} \ot 1_{\B T}) v^{TB} R_\la \| \leq C
\]
for all $\la > 0$.
\end{lemma}
\begin{proof}
It follows from the definitions of $\Ga = (v^{TB})^* v^{TB}$ and $D^{-1} \ot 1_{\B T} : F \hot \B T \to F \hot \B T$ that these two operators commute. Moreover, similarly to the proof of Proposition \ref{p:DpcommIII}, we obtain that $(D^{-1/2} \ot 1_{\B T}) v^{TB} (D^{1/2} \ot 1_{\B T}) : \Falg \ot \B T \to F \hot \B T$ extends to the bounded adjointable operator $v^{TB}(\Phi^{1/2} \ot 1_{\B T})$. This implies that
\[
(D^{-1/2} \ot 1_{\B T}) v^{TB} R_\la = v^{TB} (\Phi^{1/2}D^{-1/2} \ot 1_{\B T}) R_\la
= v^{TB} R_\la (D^{-1/2} \Phi^{1/2} \ot 1_{\B T}) . 
\]
It therefore suffices to estimate the quantity $\| (D^{-1} \ot 1_{\B T}) R_\la \|$ for all $\la > 0$. 

Let $\la > 0$ and $k,m \in \nn_0$ be given. We remark that $E_k \ot E_m$ is an invariant subspace for the selfadjoint operator $(D^{-1} \ot 1_F) R_\la : F \hot F \to F \hot F$. The restriction to this subspace is given by
\[
d_k^{-1} (\la + \Ga_{k,m})^{-1} : E_k \ot E_m \to E_k \ot E_m .
\]
Using the description of $\Ga_{k,m} : E_k \ot E_m \to E_k \ot E_m$ from \eqref{eq:gammadef} we then obtain that
\[ 
\begin{split}
\| d_k^{-1} (\la + \Ga_{k,m})^{-1} \| & \leq \|  d_k^{-1}\Gamma^{-1}_{k,m} \| 
 = d_k^{-1}  \cd \left( 1 - \frac{d_k d_{m-1}}{d_{k+1} d_m} \right)^{-1} \\
& \leq d_k^{-1}  \cd \left( 1 - \frac{d_k }{d_{k+1}} \right)^{-1} 
= \frac{d_{k+1}}{d_k} \cd (d_{k+1} - d_k)^{-1} \leq n + 1 .
\end{split}
\]
Remark that we are here applying the recursive definition of the sequence $\{d_l\}_{l = 0}^\infty$ from \eqref{eq:d_m_rec} together with Lemma \ref{l:quocon} which ensures that $d_{k+1} - d_k \geq 1$ for all $k \in \nn_0$. 
%
%We have now proved that $\| (D^{-1} \ot 1_{\B T}) R_\la \| \leq n + 1$ for all $\la > 0$.
\end{proof}

We are now ready to establish the main result of this subsection:

\begin{prop}\label{p:orthinadj}
The unbounded operator $v^{TB} |v^{TB}|^{-1} : \T{Im}( |v^{TB}|) \to F \hot \B T$ extends to a bounded adjointable isometry $\Te : F \hot \B T \to F \hot \B T$ satisfying that 
\begin{enumerate}
\item the commutator $[ \Te, x \ot 1_{\B T} ] : F \hot \B T \to F \hot \B T$ is a compact operator for all $x \in \B T$;
\item the composition $\Te \Te^*$ agrees with the orthogonal projection $\Pi : F \hot \B T \to F \hot \B T$.
\end{enumerate}
In particular, we obtain that $[x \ot 1_{\B T}, \Pi] \in \B K( F \hot \B T)$ for all $x \in \B T$.
\end{prop}
\begin{proof}
The claim in $(2)$ was already verified in Lemma \ref{l:image}. The claim regarding the commutator with $\Pi$ follows immediately from $(1)$ and $(2)$ and the fact that $\Te$ is a bounded adjointable operator. So we focus on the claim in $(1)$. It suffices to establish this claim for the generators $T_j^*$ and $T_j$, $j \in \{0,1,\ldots,n\}$. But this is a consequence of Proposition \ref{p:polar} upon applying Lemma \ref{l:gamma}, Lemma \ref{l:gammadense}, Lemma \ref{l:decomptheta} and Lemma \ref{l:DGamma_inv_sqrt}.
\end{proof}

\begin{rem}
For $n > 1$ it can be proved that $\Ga : F \hot F \to F \hot F$ has a bounded inverse. It then follows from Lemma \ref{l:WinM2} and Lemma \ref{l:gamma} that $\Ga^{-1} \in \B L(F \hot F)$ defines a positive bounded adjointable operator on the standard module $F \hot \B T$. We therefore immediately obtain that the isometry $\Te = v^{TB} \Ga^{-1/2}$ lies in $\B L(F \hot \B T)$ as well. Remark now that the set of bounded adjointable operators on $F \hot \B T$ which commutes up to compact operators with all operators of the form $x \ot 1_{\B T}$ for $x \in \B T$ form a unital $C^*$-subalgebra of $\B L(F \hot \B T)$. This observation together with Lemma \ref{l:gamma} and Proposition \ref{p:compactdouu} then allow us to conclude that $\Te \in \B L(F \hot \B T)$ has this property as well. The situation is more complicated for $n=1$ since the inverse of $\Ga : F \hot F \to F \hot F$ is in fact unbounded and this is the reason for carrying out the more detailed analysis presented in this subsection. 
\end{rem}

%%%%%%%%%%%%%%%%%%%%%%%%%%%%%%%%%%%%%%%%%%%%%%%%%%%%%%%%%%%%%%%%%%8<%%%%%%%%%%%%%%%%%%%%%%%%%%%%%%%%%%%%%%%%%%%%%%%%%%%%%%%%%%%%%%

%%%%%%%%%%%%%%%%%%%%%%%%%%%%%%%%%%%%%%%%%%%%%%%%%%%%%%%%%%%%%%%%%%8<%%%%%%%%%%%%%%%%%%%%%%%%%%%%%%%%%%%%%%%%%%%%%%%%%%%%%%%%%%%%%%

%%%%%%%%%%%%%%%%%%%%%%%%%%%%%%%%%%%%%%%%%%%%%%%%%%%%%%%%%%%%%%%%%%8<%%%%%%%%%%%%%%%%%%%%%%%%%%%%%%%%%%%%%%%%%%%%%%%%%%%%%%%%%%%%%%

%%%%%%%%%%%%%%%%%%%%%%%%%%%%%%%%%%%%%%%%%%%%%%%%%%%%%%%%%%%%%%%%%%8<%%%%%%%%%%%%%%%%%%%%%%%%%%%%%%%%%%%%%%%%%%%%%%%%%%%%%%%%%%%%%%

\subsection{First step: the classical part}
In the first step of our homotopy between the two quasi-homomorphisms $(\psi_- \ot 1_{\B T} + { (Q_0 \op 0)} \ot M_{\B T}, \psi_- \ot 1_{\B T})$ and $(\psi_+ \ot 1_{\B T}, \psi_- \ot 1_{\B T})$ we create a homotopy between the two homomorphism 
\[
(Q_0 \op 0) \ot M_{\B T} \, \, \T{and} \, \, \, (\T{inc} \op 0) \ot Q_0 : \B T \to \B L\big( (F \op F) \hot \B T \big) .
\]
This part of the homotopy behaves very much like the classical $U(1)$-case corresponding to Cuntz--Pimsner algebras associated with $C^*$-correspondences, see for instance \cite[Theorem 4.4]{Pim97}. However, since we are working with an $SU(2)$-gauge action instead of a $U(1)$-gauge action it is unreasonable to expect that the $U(1)$-argument would entirely carry over to our situation. Therefore, after this initial step there is still a quite complicated homotopy argument left and this is mainly carried out in Subsection \ref{ss:everything}. 

We recall the definition of the closed subspace $G \su F \hot F$ from \eqref{eq:classsub} and we apply the notation
\[
P := \Pi \op 1_{F \hot F} \in \B L\big( (F \hot F) \op (F \hot F) \big)
\]
for the orthogonal projection onto the closed subspace $G^\perp \op (F \hot F)$. We emphasise that it follows from the definition of the closed subspace $G \su F \hot F$ that the orthogonal projection $\Pi$ onto $G^\perp \su F \hot F$ is $SU(2)$-equivariant.

\begin{lemma}\label{l:WMT}
It holds that  $[W, P] = 0$ and the restriction $W \vert_{\T{Im}(P)} : \T{Im}(P) \to \T{Im}(P)$ is a unitary operator. In fact, we have the identities
\begin{equation}\label{eq:douupi}
v^{TT} (1 - \Pi) = (1 - \Pi) v^{TT} \, \q \T{and} \q \, v^{BT}(1 - \Pi) = 0 = (1 - \Pi) v^{TB}
\end{equation}
among bounded operators on $F \hot F$.
\end{lemma}
\begin{proof}
Let $k,m \in \nn_0$ and $\xi \in E_{k + m}$ be given and consider the vector $\io_{k,m}(\xi) \in { \T{Im}(1 - P)}$. Remark that this kind of vectors span a dense subspace of $\T{Im}(1-P)$. Using the properties of the structure maps for our subproduct system we obtain that
\[
\begin{split}
& (\io_R \ot 1_F) \io_{k,m}(\xi) = (1_F \ot \io_L)\io_{k-1,m+1}(\xi) \q \T{and} \\
& (1_F \ot \io_L) \io_{k,m}(\xi) = (\io_R \ot 1_F)\io_{k+1,m-1}(\xi) ,
\end{split}
\] 
where we apply the convention $\io_{l,-1} = 0 = \io_{l,-1}$ for all $l \in \nn_0$. Since $V_L^* \io_L = 0 = V_R^* \io_R$ and 
$\io_R^* \io_R = 1_F  - Q_0 = \io_L^* \io_L$ we then obtain that
\begin{equation}\label{eq:douushift}
\begin{split}
W\ma{c}{\io_{k,m}(\xi) \\ 0} & = \ma{c}{\io_{k-1,m+1}(\xi) \\ 0} \in \T{Im}(1 - P) \, \, \T{and} \\
W^*\ma{c}{\io_{k,m}(\xi) \\ 0} & = \ma{c}{\io_{k+1,m-1}(\xi) \\ 0} \in \T{Im}(1 - P) ,
\end{split}
\end{equation}
proving the first claim of the lemma together with the identities in \eqref{eq:douupi}. The fact that the restriction $W \vert_{\T{Im}(P)} : \T{Im}(P) \to \T{Im}(P)$ is a unitary operator now follows since both $1 - W^* W = (Q_0 \ot 1_F) \op 0$ and $1 - W W^* = (1_F \ot Q_0) \op 0$ restrict to the zero operator on $\T{Im}(P) \su (F \hot F) \op (F \hot F)$.
\end{proof}

For ease of notation, we put
\[
p_R := 1 - W^* W = \ma{cc}{ Q_0 \ot 1_{\B T} & 0 \\ 0 & 0} \, \, \T{and} \, \, \,
p_L := 1 - W W^* = \ma{cc}{ 1_F \ot Q_0 & 0 \\ 0 & 0} .
\]
For each $t \in (0,\pi/2]$ we then define the $SU(2)$-equivariant bounded adjointable operator
\begin{equation}\label{eq:defuuutee}
\begin{split}
U_t & := - \cos(t) W  + (p_L + \sin(t) W W^*)(1 - \cos(t) W^* )^{-1}(p_R +  \sin(t) W^* W)  \\
& \qq \in M_2\big(\B L(F \hot \B T) \big)  .
\end{split}
\end{equation}
Note that $U_{\pi/2} = 1$. Moreover, we define the $SU(2)$-equivariant bounded adjointable operator
\[
H_t := U_t { (1 - P)}  - W  P  \in M_2\big(\B L(F \hot \B T) \big) \su M_2\big( \B L(F \hot F)\big).
\]
For $t = 0$ we recall from Lemma \ref{l:intertwine} that
\[
H_0 =  -W  + \ma{cc}{ \Si (Q_0 \ot 1_F) & 0 \\ 0 & 0} \in M_2\big( \B L(F \hot F)\big) .
\]

\begin{lemma}\label{l:heiuni}
The $SU(2)$-equivariant bounded operator $H_t \in M_2( \B L(F \hot F) )$ is unitary for all $t \in [0,\pi/2]$.
\end{lemma}
\begin{proof}
For $t = 0$ this was already proved in Lemma \ref{l:intertwine}. Thus, let $t \in (0,\pi/2]$ be given. We start by noting that $U_t \in M_2\big( \B L(F \hot \B T)\big)$ is a unitary operator. In fact, a unitary operator like $U_t$ can be constructed from an arbitrary partial isometry $W$ in a unital $C^*$-algebra. It is in this respect crucial that $t \neq 0$ since $(1 - \cos(t) W^*)^{-1}$ would otherwise not be a well-defined bounded operator. Using Lemma \ref{l:WMT}, we then see that
\[
H_t^* H_t = U_t^* U_t (1 - P) + W^* W P = 1 = U_t U_t^* (1 - P) + W W^* P = H_t H_t^* . \qedhere
\]
\end{proof}

\begin{prop}\label{p:cont_op_norm}
Let $j \in \{0,1,\ldots,n\}$. For each $t \in [0,\pi/2]$ we have that
\begin{equation}\label{eq:cont_op_norm}
\begin{split}
& H_t^* \cd ( \psi_+(T_j^*) \ot 1_F) \cdot H_t (1 - P) \\
& \q = \big(W^* W + p_R \cdot \sin(t) \big) \cd \big(\psi_+(T_j^*) \ot 1_F\big) \cd (1 - P) 
+ \cos(t) \cd (1_{F \op F} \ot T_j^*) \cd p_R .
\end{split}
\end{equation}
In particular, it holds that the map
\[
t \mapsto  H_t^* \cd ( \psi_+(T_j^*) \ot 1_F) \cd H_t (1 - P) 
\]
is continuous in operator norm on the interval $[0,\pi/2]$.
\end{prop}
\begin{proof}
We start by remarking that
\begin{equation}\label{eq:actiontee*}
(T_j^* \ot 1_F) \io_{k,m}(\xi) = \io_{k-1,m}(T_j^* \xi)
\end{equation}
for all $k,m \in \nn_0$ and all $\xi \in E_{k+m}$. 

For the rest of this proof we sometimes use the shorthand notation $T_j^*$ for $\psi_+(T_j^*) \ot 1_F$. It follows from \eqref{eq:actiontee*} that $T_j^*(1 - P) = (1 - P) T_j^*(1 - P)$ and hence we obtain from Lemma \ref{l:WMT} and \eqref{eq:defuuutee} that
\[
H_t^* \cd T_j^* \cdot H_t (1 - P) = U_t^* T_j^* U_t (1 - P)
\]
for all $t \in (0,\pi/2]$.

Using the identities in \eqref{eq:actiontee*} and \eqref{eq:douushift} we moreover see that
\begin{equation}\label{eq:tee*W*}
\begin{split}
T_j^* W \cd (1 - P) & = W T_j^* \cd (1 - P) \q \T{and} \\
T_j^* W^*\cd (1 - P) & = W^*T_j^*  \cd (1 - P) + (1 \ot T_j^*) \cd p_R .
\end{split}
\end{equation}

For $t = 0$ we then know from Lemma \ref{l:conjdouu} and Lemma \ref{l:intertwine} that
\[
\begin{split}
H_0^* \cd T_j^* \cdot H_0 (1 - P) 
& = (\psi_-(T_j^*) \ot 1_F) \cd (1 - P) + (1 \ot T_j^*) \cd p_R \\
& = W^* T_j^* W \cd (1 - P) + (1 \ot T_j^*) \cd p_R \\
& = W^*W T_j^* \cd (1 - P) + (1 \ot T_j^*) \cd p_R .
\end{split}
\]
This proves the identity in \eqref{eq:cont_op_norm} for $t = 0$.

For $t \in (0,\pi/2]$ we record that
\[
\begin{split}
& T_j^* WW^* = WW^* T_j^*  \q \T{and} \\
& T_j^* (1 - \cos(t) W^*)^{-1} \cd (1 - P) \\
& \q = (1 - \cos(t) W^*)^{-1}\cd \big( T_j^* \cd (1 - P)
+ \cos(t) (1 \ot T_j^*) \cd p_R \big) ,
\end{split}
\]
where the first identity relies on Lemma \ref{l:conjdouu} and the second identity uses \eqref{eq:tee*W*} together with the fact that $p_R W^* = 0$. We also remark that
\[
T_j^* \cd (p_R + \sin(t) W^* W) = \sin(t) \cd T_j^* = (p_R + \sin(t) W^* W) \cd (\sin(t) p_R + W^* W) T_j^*.
\]
For $t \in (0,\pi/2]$ the identity in \eqref{eq:cont_op_norm} then follows from the computation
\[
\begin{split}
& T_j^* U_t \cd (1 - P) \\
& \q = -\cos(t) W \cd T_j^* \cd (1 - P) \\
& \qq + ( p_L + \sin(t) WW^* ) (1 - \cos(t) W^*)^{-1} \\
& \qqq  \q { \cd \big(T_j^*(1 - P) + \cos(t)(1 \ot T_j^*) p_R\big)(p_R + \sin(t) W^* W)} \\
& { \q = -\cos(t) W \cd T_j^* \cd (1 - P) + U_t \cd \cos(t)(1 \ot T_j^*) p_R } \\
& \qq { + ( p_L + \sin(t) WW^* ) (1 - \cos(t) W^*)^{-1}(p_R + \sin(t) W^* W)  } \\
& \qqq \q { \cd (\sin(t) p_R + W^* W) T_j^* \cd (1 - P) } \\
& \q = U_t \cd (\sin(t) p_R + W^* W)  T_j^* \cd (1 - P) 
+ U_t \cd \cos(t) (1 \ot T_j^*) \cd  p_R  . \qedhere
\end{split}
\]
\end{proof}

\begin{lemma}\label{l:cont_op_norm}
Let $K \in M_2(\B K \hot \B T) \subseteq M_2\big( \B L (F \hot F) \big)$. The map $t \mapsto H_t^* K$ is continuous in operator norm on the interval $[ 0, \pi/2]$.
\end{lemma}
\begin{proof}
Since the map $t \to H_t$ is continuous in operator norm on the interval $(0,\pi/2]$, it is enough to check continuity at $t= 0$. 

We recall that 
\[
\begin{split}
H_t^* & = U_t^* (1 - P) - W^* P \\
 & = \big( - \cos(t) W^* + (p_R + \sin(t) W^* W)(1 - \cos(t) W)^{-1}(p_L + \sin(t) W W^*) \big)(1 - P)
- W^* P
\end{split}
\]
for $t \in (0,\pi/2]$ whereas 
\[
\begin{split}
H_0^* = -W^* P  -W^*(1 - P)  + \ma{cc}{ (Q_0 \ot 1_F )\Si & 0 \\ 0 & 0 } .
\end{split}
\]

We remark that $\lim_{N \to \infty}( \sum_{k = 0}^N Q_k \ot 1_{F \op F} ) K = K$, where the convergence takes place in operator norm. Next, we recall from Proposition \ref{p:orthinadj} that $P \in M_2\big(\B L( F \hot \B T)\big)$ and moreover that $M_2(\B K \hot \B T) \su M_2(\B L(F \hot \B T))$ is an ideal. Because of the structure of the involved operators, we may then focus on proving that
\[
\begin{split}
& \lim_{t \to 0} (p_R + \sin(t) W^* W)(1 - \cos(t) W)^{-1}(p_L + \sin(t) WW^*) \cd (Q_k \ot 1_{F \op F}) \cd (1 - P) \\
& \q = {(Q_0 \ot 1_F) \Si (Q_k \ot 1_F) \op 0 }  .
\end{split}
\]
for every fixed $k \in \nn_0$. However, by \eqref{eq:douushift} we have that
\[
\begin{split}
& \lim_{t \to 0} (p_R + \sin(t) W^* W)(1 - \cos(t) W)^{-1}(p_L + \sin(t) WW^*) (Q_k \ot 1_{F \op F}) \cd (1 - P) \\
& \q = { \lim_{t \to 0} } (p_R + \sin(t) W^* W) \sum_{j = 0}^k (\cos(t) W )^j (p_L + \sin(t) WW^*) (Q_k \ot 1_{F \op F}) \cd (1 - P) \\
& \q = p_R \sum_{j = 0}^k W^j (Q_k \ot 1_{F \op F}) p_L
= p_R W^k (Q_k \ot 1_{F \op F}) p_L
= { (Q_0 \ot 1_F) \Si (Q_k \ot 1_F) \op 0 }.
\end{split}
\]
This proves the result of the lemma.
\end{proof}

\begin{prop}\label{p:Ht_cont}
Let $x \in \B T$. The difference
\[ 
H_t^* (\psi_+(x) \ot 1_F) H_t - (\psi_-(x) \ot 1_F)
\]
defines a compact operator on $(F \op F) \hot \B T$ for all $t \in [0,\pi/2]$ and the map
\[
[0,\pi/2] \to  \B L( (F \op F) \hot \B T) 
\q t \mapsto H_t^* (\psi_+(x) \ot 1_F) H_t
\]
is norm-continuous. In particular, we have the identity
\[
{\bf 1}_{\B T} = \big[ H_{\pi/2}^* ( \psi_+ \ot 1_{\B T} ) H_{\pi/2} , \psi_- \ot 1_{\B T} \big]
\]
inside $KK_0^{SU(2)}(\B T,\B T)$.
\end{prop}
\begin{proof}
We start by proving the statement on compactness. For $t = 0$ we know from Lemma \ref{l:intertwine} that
\[
H_0^* (\psi_+(x) \ot 1_F) H_0 - (\psi_-(x) \ot 1_F) = p_R (1_{F \op F} \ot x) ,
\]
which belongs to $M_2( \B K(F \hot \B T))$ since $p_R =( Q_0 \ot 1_{\B T}) \op 0$. For $t \in (0,\pi/2]$ we see from Lemma \ref{l:WinM2}, Proposition \ref{p:compactdouu} and Proposition \ref{p:orthinadj} that $[ \psi_+(x) \ot 1_{\B T} , H_t ] \in M_2( \B K(F \hot \B T))$. An application of Lemma \ref{l:heiuni} and Proposition \ref{p:quasi} then yields that
\[
H_t^* (\psi_+(x) \ot 1_{\B T}) H_t \sim \psi_+(x) \ot 1_{\B T} \sim \psi_-(x) \ot 1_{\B T}
\] 
hence proving the statement regarding compactness. 

We now focus on proving norm-continuity. Using standard density arguments, we may restrict our attention to the case where $x$ is one of the generators $x = T_j^*$ for some $j \in \{0,1,\ldots,n\}$. Once more, we use the shorthand notation $T_j^* := \psi_+(T_j^*) \ot 1_F$. We already know from Proposition \ref{p:cont_op_norm} that the path $t \mapsto H_t^* T_j^* H_t (1 - P)$ is continuous in operator norm on $[0,\pi/2]$. Now, for $t \in [0,\pi/2]$ we have that 
\[
H_t^* T_j^* H_t P = - H_t^* T_j^* W P 
= - H_t^* W P T_j^* - H_t^* [T_j^*,W P]
= T_j^* - H_t^* [T_j^*,WP] . 
\]
Since the commutator $[T_j^*, W P]$ belongs to $M_2(\B K \hot \B T)$ by Proposition \ref{p:compactdouu} and Proposition \ref{p:orthinadj}, it follows from Lemma \ref{l:cont_op_norm} that $t \mapsto H_t^* T_j^* H_t P$ is norm-continuous as well. This proves the statement regarding continuity.

The remaining claim on classes in $SU(2)$-equivariant $KK$-theory now follows from the above considerations upon remarking that all the involved quasi-homomorphisms are $SU(2)$-equivariant. Indeed, we then have the string of identities
\[
\begin{split}
{\bf 1_{\B T}} 
& = [ \psi_- \ot 1_{\B T} + p_R (1_{F \op F} \ot M_{\B T}), \psi_- \ot 1_{\B T}]
= [ H_0^*(\psi_+ \ot 1_F) H_0, \psi_- \ot 1_{\B T}] \\
& = [H_{\pi/2}^* (\psi_+ \ot 1_{\B T}) H_{\pi/2}, \psi_- \ot 1_{\B T} ]
\end{split}
\]
inside $KK_0^{SU(2)}( \B T,\B T)$.
\end{proof}

%%%%%%%%%%%%%%%%%%%%%%%%%%%%%%%%%%%%%%%%%%%%%%%%%%%%%%%%%%%%%%%%%%8<%%%%%%%%%%%%%%%%%%%%%%%%%%%%%%%%%%%%%%%%%%%%%%%%%%%%%%%%%%%%%%

%%%%%%%%%%%%%%%%%%%%%%%%%%%%%%%%%%%%%%%%%%%%%%%%%%%%%%%%%%%%%%%%%%8<%%%%%%%%%%%%%%%%%%%%%%%%%%%%%%%%%%%%%%%%%%%%%%%%%%%%%%%%%%%%%%

%%%%%%%%%%%%%%%%%%%%%%%%%%%%%%%%%%%%%%%%%%%%%%%%%%%%%%%%%%%%%%%%%%8<%%%%%%%%%%%%%%%%%%%%%%%%%%%%%%%%%%%%%%%%%%%%%%%%%%%%%%%%%%%%%%

%%%%%%%%%%%%%%%%%%%%%%%%%%%%%%%%%%%%%%%%%%%%%%%%%%%%%%%%%%%%%%%%%%8<%%%%%%%%%%%%%%%%%%%%%%%%%%%%%%%%%%%%%%%%%%%%%%%%%%%%%%%%%%%%%%

\subsection{Second step: everything else}\label{ss:everything}
For each $t \in [0,1]$ we define the $SU(2)$-equivariant bounded adjointable operator
\[
y_t := 1 - P + \ma{cc}{ (1-t)^{1/2} v^{TT} & v^{TB} \\ v^{BT} & (1 - t)^{1/2} v^{BB}} P
: (F \op F) \hot \B T \to (F \op F) \hot \B T .
\]
Since the assignment $t \mapsto y_t$ is continuous in operator norm we obtain a bounded adjointable operator
\[
y : (F \op F) \hot C( [0,1], \B T) \to (F \op F) \hot C( [0,1], \B T) ,
\]
which acts as $y_t$ on the fibre $(F \op F) \hot \B T$ associated with the evaluation at the point $t \in [0,1]$.

We shall see in this subsection that both $y$ and $y^*$ have dense images and that the corresponding unitary operator (obtained via polar decomposition)
\[
I : (F \op F) \hot C( [0,1], \B T) \to (F \op F) \hot C( [0,1], \B T)
\]
yields the next step of our homotopy. 

More precisely, it is the aim of this subsection to prove the following:

\begin{prop}
For each $x \in \B T$, it holds that the path $t \mapsto I_t^* (\psi_+(x) \ot 1_{\B T}) I_t - (\psi_+(x) \ot 1_{\B T})$ is a norm-continuous path of compact operators on $(F \op F) \hot \B T$. In particular, we have the identity
\[
{\bf 1}_{\B T} = [ I_1^* (\psi_+ \ot 1_{\B T}) I_1, \psi_- \ot 1_{\B T}]
\]
inside the $SU(2)$-equivariant $KK$-group, $KK_0^{SU(2)}(\B T,\B T)$.
\end{prop}

The proof of this proposition relies on the results in Appendix \ref{a:quasi}. Aligning with the notation applied in Appendix \ref{a:quasi} we define
\begin{equation}\label{eq:ydefinit}
\begin{split}
X := (F \op F) \hot C( [0,1], \B T) \q \s X := (\Falg \op \Falg) \ot C( [0,1], \B T) \\
x_j := P \big( \psi_+(T_j^*) \ot 1_{ C( [0,1],\B T)} \big) P \q  K := (D^{-1} \op D^{-1}) \ot 1_{ C( [0,1],\B T)} ,
\end{split}
\end{equation}
for all $j \in \{0,1,2,\ldots,n\}$. Remark here that $P : X \to X$ is the orthogonal projection which agrees with $P \in \B L\big( (F \op F) \hot \B T \big)$ in each fibre (corresponding to the evaluations at the points $t \in [0,1]$). We notice that $x_j : X \to X$ is a bounded adjointable operator for every $j \in \{0,1,2\ldots,n\}$ whereas $K : X \to X$ is a compact operator.

\begin{lemma}\label{l:ydense}
The bounded adjointable operators $y$ and $y^* : X \to X$ both have norm-dense image. Moreover, it holds for each $j \in \{0,1,2\ldots,n\}$ that $x_j(\s X) , x_j^*(\s X), y(\s X), y^*(\s X) \su \s X$.
\end{lemma}
\begin{proof}
We first remark that $\Pi (Q_k \ot Q_m) = (Q_k \ot Q_m) \Pi$ for all $k,m \in \nn_0$ and this implies that $\Pi$ preserves the dense submodule $\Falg \ot \B T \su F \hot \B T$. The fact that $x_j, x_j^*, y$ and $y^*$ all preserve the dense submodule $\s X = (\Falg \op \Falg) \ot C( [0,1],\B T)$ is then a consequence of Lemma \ref{l:WinM2} and the definition of the Toeplitz operators $T_j$ and $T_j^* \in \B T$.

We continue by focusing on the claim regarding the images of $y$ and $y^*$. Since the path $t \mapsto y_t$ is norm-continuous it suffices to verify that $y_t$ and $y_t^* : (F \op F) \hot \B T \to (F \op F ) \hot \B T$ both have norm-dense image for each $t \in [0,1]$. Applying Lemma \ref{l:WMT} we obtain that
\begin{equation}\label{eq:y*yyy*}
\begin{split}
y_t^* {y_t} & = { \ma{cc}{ (1 - \Pi) + (1 - t + t \cd (v^{BT})^* v^{BT} ) \Pi & 0 \\ 0 &  1 - t + t  \cd (v^{TB})^* v^{TB}  }} \q \T{and} \\
y_t y_t^* & = { \ma{cc}{ (1 - \Pi) + ( 1 - t + t \cd v^{TB} (v^{TB})^* ) \Pi & 0 \\ 0 &  1 - t + t \cd v^{BT} (v^{BT})^* } }
\end{split}
\end{equation}
for all $t \in [0,1]$. For $t \in [0,1)$ we see from these identities that $y_t$ and $y_t^*$ are in fact invertible as bounded adjointable operators (and they are therefore in particular surjective). 

For $t = 1$ we obtain from \eqref{eq:adjointTB} that 
\[
\begin{split}
y_1 & = \ma{cc}{1 - \Pi & v^{TB} \\ v^{BT} & 0} 
= \ma{cc}{1 - \Pi & v^{TB} \\ (\Phi^{-1/2} \ot \Phi^{1/2}) \cd (v^{TB})^* & 0}
\q \T{and} \\
y_1^* & = \ma{cc}{1 - \Pi & (v^{BT})^* \\ (v^{TB})^* & 0}
= \ma{cc}{1 - \Pi & v^{TB} \cd (\Phi^{-1/2} \ot \Phi^{1/2}) \\ (v^{TB})^* & 0} .
\end{split}
\]
We recall that $\Phi : F \to F$ is an invertible element in $\B T \su \B L(F)$. The fact that $y_1$ and $y_1^*$ have dense images then follows from an application of Lemma \ref{l:gammadense} and Lemma \ref{l:image}.
\end{proof}

In order to achieve a better understanding of the bounded adjointable operator $y^* y : X \to X$ we apply the decomposition from Theorem \ref{t:iso_fusion}. This decomposition allows us for each $k,m \in \nn_0$ to introduce the bounded operator
\[
\begin{split}
& \De_{k,m} : E_k \ot E_m \to E_k \ot E_m \\
& \De_{k,m}\big(\si^j \io_{k-j,m-j}(\xi)\big) := \fork{ccc}{ 0 & \T{for} & j = 0 \\ 
\frac{d_k d_{m-1}}{d_{k-1} d_m} \cd (1 - \frac{d_{k-j} d_{m-j-1}}{d_k d_{m-1}}) \cd \si^j \io_{k-j,m-j}(\xi) 
& \T{for} & 0 < j \leq k,m }
\end{split}
\]
defined whenever $0 \leq j \leq k,m$ and $\xi \in E_{k+m-2j}$. We notice that 
\[
\| \De_{k,m} \| \leq \frac{d_k d_{m-1}}{d_{k-1} d_m} \leq n + 1
\]
for all $k,m \in { \nn}$ and we therefore obtain a bounded operator 
\[
\De : F \hot F \to F \hot F \q \De(\xi) := \De_{k,m}(\xi) \, \, , \, \, \, k,m \in \nn_0 .
\]
Remark also that $\De_{k,m} = 0$ for $k = 0$ or $m = 0$.

\begin{lemma}\label{l:delta}
We have the identity $(v^{BT})^* v^{BT} = \De$. In particular, it holds that $\De \in \B L(F \hot F)$ belongs to the unital $C^*$-subalgebra $\B L(F \hot \B T) \su \B L(F \hot F)$.
\end{lemma}
\begin{proof}
The identity holds trivially on $E_k \ot E_m$ for $k = 0$ or $m = 0$. Thus, let $k,m \in \nn$. From the identities in \eqref{eq:adjointTB}, Lemma \ref{l:sigteebot} and the definition in \eqref{eq:defphi} we obtain that
\begin{equation}\label{eq:vbottop}
((v^{BT})^* v^{BT})(\eta) = v^{TB} (\Phi^{-1} \ot \Phi) (v^{TB})^*(\eta)
= \frac{d_k d_{m-1} }{\mu_k \cd d_{k-1} d_m} \si_{k-1,m-1} \si_{k-1,m-1}^*(\eta)
\end{equation}
for all $\eta \in E_k \ot E_m$. Let $0 \leq j \leq k,m$ and let $\xi \in E_{k+m-2j}$ be given. Using Theorem \ref{t:iso_fusion} we only need to verify that
\[
((v^{BT})^* v^{BT})( \si^j \io_{k-j,m-j}(\xi)) = \De_{k,m}( \si^j \io_{k-j,m-j}) .
\]
The case where $j = 0$ follows since $v^{BT}(1 - \Pi) = 0$ and the remaining cases follow from \eqref{eq:vbottop} upon applying Lemma \ref{l:sigma_propI} and Lemma \ref{l:sigma_propII}. 
\end{proof}

The next lemma is a straightforward consequence of Lemma \ref{l:gamma}, Lemma \ref{l:delta} and \eqref{eq:y*yyy*}.

\begin{lemma}\label{l:y*ydelgam}
Let $t \in [0,1]$. We have the identity
\[
\begin{split}
y_t^* y_t 
& = \ma{cc}{ 1 -  \Pi + \big( (1 - t) + t \cd \De \big) \cd \Pi & 0 \\
0 &  1 - t + t  \cd \Ga } .
\end{split}
\]
\end{lemma}

\begin{lemma}\label{l:y*ydense}
The norm-dense submodule $\s X = (\Falg \op \Falg) \ot C( [0,1],\B T) { \su X}$ is contained in the image of $y^* y : X \to X$.
\end{lemma}
\begin{proof}
For $k \in \nn_0$ we sometimes apply the identification $Q_k := Q_k \ot 1_{\B T} : F \hot \B T \to F \hot \B T$. It follows from Lemma \ref{l:y*ydelgam} and the definition of the involved operators that
\[
y^*_t y_t (Q_k \op Q_k) = (Q_k \op Q_k) y^*_t y_t
\]
for all $t \in [0,1]$. In particular, upon identifying $Q_k \in \B L( F \hot \B T)$ with the constant path with value $Q_k$ for all $t \in [0,1]$, we obtain that
\[
\T{Im}\Big( (Q_k \op Q_k) \cd \big( y^* y (Q_k \op Q_k) + (1 - Q_k) \op (1 - Q_k) \big) \Big) 
\su \T{Im}(y^* y)
\]
for all $k \in \nn_0$. Since $\T{Im}(Q_k \op Q_k) = (E_k \op E_k) \ot C( [0,1],\B T)$ it therefore suffices to show that
\[
y^* y(Q_k \op Q_k) : (Q_k \op Q_k) X \to (Q_k \op Q_k) X 
\]
is invertible. In other words, we have to show that the fibre 
\[
(y^*_t y_t)(Q_k \op Q_k) : Q_k (F \hot \B T) \op Q_k (F \hot \B T) \to Q_k (F \hot \B T) \op Q_k(F \hot \B T)
\]
is invertible for each $t \in [0,1]$ and that
\[
\sup_{t \in [0,1]}\big\| \big( (y^*_t y_t)(Q_k \op Q_k) \big)^{-1} \big\| < \infty .
\] 
As we did in Lemma \ref{l:gammadense} we may switch over and solve the corresponding problem on the Hilbert space $(Q_k F \hot F) \op (Q_k F \hot F)$. We apply Lemma \ref{l:y*ydelgam} and deal with each component separately, namely
\[
1 - \Pi + \big( (1 - t) + t \cd \De \big) \cd \Pi \q \T{and} \q
1 - t + t \cd \Ga : F \hot F \to F \hot F .
\] 

Let $k \in \nn_0$ be fixed. We saw in the proof of Lemma \ref{l:gammadense} that $(1 - t + t \cd \Ga ) (Q_k \ot 1_F) : E_k \ot F \to E_k \ot F$ for all $t \in [0,1]$ is invertible and that 
\[
\sup_{t \in [0,1]} \big\| \big( (1 - t + t \cd \Ga ) (Q_k \ot 1_F) \big)^{-1} \big\| < \infty .
\]
Remark that we are here also applying that $\Ga : F \hot F \to F \hot F$ is a positive bounded operator.

We now consider the problematic part of the other component of $y_t^* y_t(Q_k \op Q_k)$: \[(1 - t + t \De) \Pi (Q_k \ot 1_F) :  \Pi( E_k \ot F ) \to \Pi( E_k \ot F).\] Remark in this respect that $(Q_k \ot Q_m)\Pi = \Pi (Q_k \ot Q_m)$ for all $m \in \nn_0$.

For each $t \in [0,1]$ and $m \in \nn_0$ we are interested in the invertible operator
\[
(1 - t + t \De) \Pi (Q_k \ot Q_m) : \Pi( E_k \ot E_m) \to \Pi(E_k \ot E_m)
\]
For $k = 0$ or $m = 0$ we have that $\Pi( E_k \ot E_m)  = \{0\}$ so suppose that $k,m \in \nn$. In this case, we have that the bounded operator $\De_{k,m} \Pi : \Pi (E_k \ot E_m) \to \Pi (E_k \ot E_m)$ is invertible with
\begin{equation}\label{eq:boudelpi}
\| (\De_{k,m} \Pi )^{-1} \| 
= \frac{d_{k-1} d_m}{d_k d_{m-1}} \left( 1 - \frac{d_{k-1} d_{m-2}}{ d_k d_{m-1}} \right)^{-1} 
\leq \frac{d_{k-1} d_1}{d_k} \left( 1 - \frac{d_{k-1}}{ d_k} \right)^{-1}
= d_1 \cd \left( \frac{d_k}{d_{k-1}} - 1 \right)^{-1}  .
\end{equation}
Since this norm-bound is independent of $m \in \nn$ we conclude that
\[
( t + (1 - t)\De) \Pi (Q_k \ot 1_F)
: \Pi( E_k \ot F) \to \Pi(E_k \ot F)
\]
is invertible for all $t \in [0,1]$ and that
\[
\sup_{t \in [0,1]} \big\| \big( ( t + (1 - t)\De) \Pi  (Q_k \ot 1_F)  \big)^{-1} \big\|  < \infty .
\]
We are here also relying on the positivity of the bounded operator $\De : F \hot F \to F \hot F$. 
\end{proof}

Recall the definition of the bounded adjointable operators $x_j, y$ and $K : X \to X$ from \eqref{eq:ydefinit}. 

\begin{lemma}\label{l:yfact}
Let $j \in \{0,1,2,\ldots,n\}$. There exist bounded adjointable operators $L, \ov{L}, M, \ov{M} : X \to X$ such that
\[
K^{1/2} L K^{1/2} = [x_j,y] = MK \q \mbox{and} \q K^{1/2} \ov{L} K^{1/2} = [x_j,y^*] = K \ov{M} 
\]
\end{lemma}
\begin{proof}
To ease the notation we put $T_j^* := T_j^* \ot 1_{\B T}$. For each $t \in [0,1]$ we apply Lemma \ref{l:WMT} and compute that
\[
\begin{split}
[x_j,y_t] & = \ma{cc}{ (1 - t)^{1/2} \Pi \cd [ T_j^*,v^{TT} ] \cd \Pi & \Pi   [ T_j^*, v^{TB} ]   \\ 
\, [ T_j^*,v^{BT} ] \cd  \Pi & (1 - t)^{1/2}  [ T_j^*,v^{BB} ]  } \q \T{and} \\
[x_j,y_t^*] & = \ma{cc}{ (1 - t)^{1/2} \Pi \cd [T_j^*,(v^{TT})^*] \cd \Pi & \Pi \cd [T_j^*,(v^{BT})^*] \\ 
\, [T_j^*,(v^{TB})^*] \cd \Pi & (1 - t)^{1/2} [T_j^*,(v^{BB})^*] } .
\end{split}
\]
We consider the inverses $K^{-1/2}$ and $K^{-1}$. These positive and regular unbounded operators both have $\s X$ as a core and on this core they are given by
\[
D^{1/2} \ot 1_{\cc^2 \ot C([0,1],\B T)} \T{ and } D \ot 1_{\cc^2 \ot C([0,1],\B T)} : \s X \to X ,
\]
respectively. The result of the lemma now follows from Proposition \ref{p:DpcommIII}. Indeed, $L$ and $\ov{L}$ are the bounded adjointable extensions of $D^{1/2} [x_j,y] D^{1/2}$ and $D^{1/2}[x_j,y^*]D^{1/2}$, respectively. Whereas $M$ and $\ov{M}$ are the bounded adjointable extensions of $[x_j, y] D$ and $D[x_j, y^*]$, respectively. We remark that all of these four unbounded operators are understood to be defined on the algebraic tensor product $\s X = (\Falg \op \Falg) \ot C([0,1],\B T)$. Indeed, this algebraic tensor product works well in this respect since it is a core for both $D$ and $D^{1/2}$ and since it is invariant under $x_j, y$ and $y^*$.
\end{proof}

For each $\la > 0$ we put $R_\la := (\la + y^* y)^{-1/2}$.

\begin{lemma}\label{l:ysupfin}
It holds that $K y^* y = y^* y K$. Moreover, there exists a constant such that
\[
\| K  R_\la \| \leq C \q \mbox{and} \q \| K^{1/2} y R_\la \| \leq C
\]
for all $\la > 0$.
\end{lemma}
\begin{proof}
The fact that $K y^* y = y^* y K$ follows since $y^* y$ leaves the submodule $(E_k \op E_l) \ot C([0,1],\B T)$ invariant for all $k,l \in \nn_0$. Moreover, writing $y : X \to X$ as a $2 \ti 2$-matrix in the following fashion
\[
y = \ma{cc}{y^{TT} & y^{TB} \\ y^{BT} & y^{BB}} \in M_2\big( \B L( F \hot C([0,1],\B T) ) \big)
\]
we see from the argument given in the proof of Proposition \ref{p:DpcommIII} that $K^{1/2} y K^{-1/2} : \s X \to X$ extends to the bounded adjointable
\[
\ma{cc}{1 - \Pi + (\Phi^{-r} \ot 1)y^{TT} \Pi & y^{TB} (\Phi^r \ot 1) \\ (\Phi^{-r} \ot 1)y^{BT} & y^{BB} (\Phi^r \ot 1)}
\]
operator on $X$. Since each component in 
\[
R_\la = \ma{cc}{R_\la^{TT} & 0 \\ 0 & R_\la^{BB}} = \ma{cc}{(1 - \Pi) R_\la^{TT}(1 - \Pi) + \Pi R_\la^{TT} \Pi & 0 \\ 0 & R_\la^{BB}}
\]
commutes with $\Phi \ot 1$, it suffices to find a constant $C > 0$ such that $\| K R_\la \| \leq C$ for all $\la > 0$. Using the description of $y^* y : X \to X$ from Lemma \ref{l:y*ydelgam}, together with the definitions of $\Ga$ and $\De : F \hot F \to F \hot F$, we may focus on showing that
\[
\sup_{k,m \in \nn_0} \| d_k^{-1} \Ga_{k,m}^{-1} \| < \infty \q \T{and} \q
\sup_{k,m \in \nn} \| d_k^{-1} (\De_{k,m} \Pi)^{-1} \| < \infty ,
\]
where we consider $\De_{k,m} \Pi$ as a bounded invertible operator on the Hilbert space $\Pi(E_k \ot E_m)$ for $k,m \in \nn$. The first estimate was already established in the proof of Lemma \ref{l:DGamma_inv_sqrt} and the second estimate follows from Lemma \ref{l:quocon} and the estimate in \eqref{eq:boudelpi}. Indeed, we have that
\[
\| d_k^{-1} (\De_{k,m} \Pi)^{-1} \| \leq \frac{d_1}{d_k} \left( \frac{d_k}{d_{k-1}} - 1 \right)^{-1} 
\leq d_1 \cd \frac{d_{k-1}}{d_k} \leq (n+1) \cd \ga_n
\]
for all $k,m \in \nn$. 
\end{proof}

For each $t \in [0,1]$, define $I_t : (F \op F) \hot \B T \to (F \op F) \hot \B T$ as the bounded adjointable extension of
\[
y_t |y_t|^{-1} : \T{Im}( |y_t|) \to (F \op F) \hot \B T .
\]
We emphasise that
\begin{equation}\label{eq:Iide}
I_0 = y_0 = H_{\pi/2} \, \, \T{and} \, \, \, I_1 = \ma{cc}{1 - \Pi & \Te \\ \Te^* & 0} 
: (F \op F) \hot \B T \to (F \op F) \hot \B T ,
\end{equation}
where the bounded adjointable isometry $\Te : F \hot \B T \to F \hot \B T$ was introduced in Lemma \ref{l:image}.

We are now ready to prove the main result of this subsection:

\begin{prop}\label{p:I_t}
The map $t \mapsto I_t$ is a strictly continuous path of $SU(2)$-equivariant unitary operators on $(F \op F) \hot \B T$. Moreover, for every $x \in \B{T}$, the map $t \mapsto I_t^* (\psi_+(x) \ot 1_{\B T}) I_t -\psi_+(x) \ot 1_{\B T}$ is a norm-continuous path of compact operators on $(F \op F) \hot \B{T}$. In particular, we have the identity
\[
{\bf 1}_{\B T} = [ I_1^* (\psi_+ \ot 1_{\B T}) I_1, \psi_- \ot 1_{\B T}]
\]
inside $KK_0^{SU(2)}(\B T,\B T)$.
\end{prop}
\begin{proof}
By Lemma \ref{l:ydense} the operator $y |y|^{-1}: \T{Im}(|y|) \to X$ extends to a unitary operator $I$ on $X = (F \op F) \hot C([0,1],\B T)$. The fibres of this unitary operator are exactly the unitary operators $I_t : (F \op F) \hot \B T \to (F \op F) \hot \B T$, $t \in [0,1]$. This means that the path $t \mapsto I_t$ is a strictly continuous path of unitary operators on $(F \op F) \hot \B T$. Moreover, since $y_t \in \B L\big( (F \op F) \hot \B T\big)$ is $SU(2)$-equivariant we obtain that $I_t \in \B L\big( (F \op F) \hot \B T\big)$ is $SU(2)$-equivariant as well.

Next, a combination of Proposition \ref{p:polar}, Lemma \ref{l:ydense}, Lemma \ref{l:y*ydense}, Lemma \ref{l:yfact} and Lemma \ref{l:ysupfin} shows that the commutators $[x_j,I]$ and $[x_j^*,I]$ belong to the compact operators on $(F \op F) \hot C( [0,1],\B T)$ for every $j \in \{0,1,2,\ldots,n\}$. Now, put $T_j^* := \psi_+(T_j^*) \ot 1_{C([0,1],\B T)}$ and remark that
\[
T_j^* = x_j + (1 - P) T_j^* (1 - P) + (1 - P) T_j^* P .
\]
We know from Proposition \ref{p:orthinadj} that $(1 - P) T_j^* P$ is a compact operator on $(F \op F) \hot C( [0,1],\B T)$ and we moreover have that
\[
[I,T_j^*] = [I,x_j] + [I, (1 - P) T_j^* P]
\]
and similarly for with $I^*$ instead of $I$. This shows that $[I,T_j^*]$ and $[I^*,T_j^*]$ are compact operators on $(F \op F) \hot C( [0,1],\B T)$ for all $j \in \{0,1,2,\ldots,n\}$ and hence that
\[
I^* (\psi_+(x) \ot 1_{C([0,1],\B T)}) I - \psi_+(x) \ot 1_{C([0,1],\B T)}
\]
is a compact operator on $(F \op F) \hot C([0,1],\B T)$ for all $x \in \B T$. But this means that the path
\[
t \mapsto I_t^* (\psi_+(x) \ot 1_{\B T}) I_t - \psi_+(x) \ot 1_{\B T}
\]  
is a norm-continuous path of compact operators on $(F \op F) \hot \B T$. Since $\psi_+(x) \ot 1_{\B T} - \psi_-(x) \ot 1_{\B T}$ is a compact operator as well (for every $x \in \B T$) we obtain the identity
\[
[ I_0^* (\psi_+ \ot 1_{\B T}) I_0, \psi_- \ot 1_{\B T}] = [I_1^* (\psi_+ \ot 1_{\B T}) I_1, \psi_- \ot 1_{\B T}]
\]
inside the $SU(2)$-equivariant $KK$-group $KK^{SU(2)}_0(\B T,\B T)$. Since $I_0 = H_{\pi/2}$ we obtain the result of the present proposition by an application of Proposition \ref{p:Ht_cont}.
\end{proof}

\begin{rem}
For $n>1$, it can be established that both $y^*y$ and $yy^*$ are invertible as bounded adjointable operators on $X$. The proof of Proposition \ref{p:I_t} therefore simplifies a lot for $n > 1$. For $n = 1$ it only holds that $y^* y$ and $y y^*$ have dense images in $X$ and this is the reason for some of the more detailed analysis carried out in this subsection. 
\end{rem}

%%%%%%%%%%%%%%%%%%%%%%%%%%%%%%%%%%%%%%%%%%%%%%%%%%%%%%%%%%%%%%%%%%8<%%%%%%%%%%%%%%%%%%%%%%%%%%%%%%%%%%%%%%%%%%%%%%%%%%%%%%%%%%%%%%

%%%%%%%%%%%%%%%%%%%%%%%%%%%%%%%%%%%%%%%%%%%%%%%%%%%%%%%%%%%%%%%%%%8<%%%%%%%%%%%%%%%%%%%%%%%%%%%%%%%%%%%%%%%%%%%%%%%%%%%%%%%%%%%%%%

%%%%%%%%%%%%%%%%%%%%%%%%%%%%%%%%%%%%%%%%%%%%%%%%%%%%%%%%%%%%%%%%%%8<%%%%%%%%%%%%%%%%%%%%%%%%%%%%%%%%%%%%%%%%%%%%%%%%%%%%%%%%%%%%%%

%%%%%%%%%%%%%%%%%%%%%%%%%%%%%%%%%%%%%%%%%%%%%%%%%%%%%%%%%%%%%%%%%%8<%%%%%%%%%%%%%%%%%%%%%%%%%%%%%%%%%%%%%%%%%%%%%%%%%%%%%%%%%%%%%%

\subsection{Third step: proof of $KK$-equivalence}
We are now ready to finish the proof of Theorem \ref{t:KKequiv} establishing that $\B T$ and $\B C$ are $KK^{SU(2)}$-equivalent. 

\begin{proof}[Proof of Theorem \ref{t:KKequiv}]
From Proposition \ref{p:I_t} we have the identity
\[
{\bf 1}_{\B T} = [ I_1^* (\psi_+ \ot 1_{\B T}) I_1, \psi_- \ot 1_{\B T}]
\]
inside the $SU(2)$-equivariant $KK$-group $KK_0^{SU(2)}(\B T,\B T)$. Thus in order to prove the identity
\[
{\bf 1}_{\B T} = { [\psi_+,\psi_-] \hot_{\B C} [i]}
\]
we only need to show that
\begin{equation}\label{eq:rota}
[ I_1^* (\psi_+ \ot 1_{\B T}) I_1, \psi_- \ot 1_{\B T}] = [ \psi_+ \ot 1_{\B T}, \psi_- \ot 1_{\B T}] .
\end{equation}
We recall from \eqref{eq:Iide} that 
\[
I_1 = \ma{cc}{ 1 - \Pi & \Te \\ \Te^* & 0}
\]
and hence that $I_1 \in \B L\big( (F \op F) \hot \B T \big)$ is an $SU(2)$-equivariant selfadjoint unitary operator.

For each $t \in [0,1]$, define 
\[ 
J_t := \frac{1+I_1}{2} + \exp (\pi i t)  \cd \frac{I_1 -1}{2}
\]
so that $J_t \in \B L\big( (F \op F) \hot \B T \big)$ is an $SU(2)$-equivariant unitary operator and $t \mapsto J_t$ is a norm continuous path with $J_0=I_1$ and $J_1=1$. Moreover, for every $x \in \B{T}$, the assignment 
\[
[0,1] \ni t \mapsto { J_t^* (\psi_+(x) \ot 1_{\B T}) J_t - \psi_+(x) \ot 1_{\B T}}
\]
yields a norm continuous path of compact operators on the module $(F \op F) \hot \B{T}$. Indeed, the last claim on compactness follows immediately from Proposition \ref{p:I_t}. 

The existence of the path $t \mapsto J_t$ with the above properties establishes the identity in \eqref{eq:rota} and we have proved our main theorem.
\end{proof}

\section{The Gysin sequence}\label{s:gysin}
Throughout this section, we fix a strictly positive integer $n$ and consider the irreducible representation $\rho_n : SU(2) \to U(L_n)$. We apply the notation
\[
\B K(F) := \B K(F(\rho_n,L_n))  \, , \, \, \B T := \B T(\rho_n,L_n) \, \, \T{and} \, \, \, \B O := \B T(\rho_n,L_n) / \B K(F(\rho_n,L_n))
\]
for the associated compact operators, Toeplitz algebra and Cuntz--Pimsner algebra. By construction, we have the exact sequence
\[
\begin{CD}
0 @>>> \B K( F(\rho_n,L_n)) @>{j}>> \B T(\rho_n,L_n) @>{q}>> \B O(\rho_n,L_n) @>>> 0
\end{CD}
\]
of $C^*$-algebras. This exact sequence in turn results in the following six term exact sequence of $K$-groups:
\[
\begin{CD}
K_0( \B K( F) ) @>{j_*}>> K_0( \B T) @>{q_*}>> K_0( \B O ) \\
@A{\pa}AA & & @VV{\pa}V \\ 
K_1( \B O)  @<{q_*}<< K_1( \B T) @<{j_*}<< K_1( \B K(F))
\end{CD}
\]
We recall that the compact operators $\B K( F)$ are strongly Morita equivalent to the complex number via the $C^*$-correspondence $F = F(\rho_n,L_n)$ from $\B K(F(\rho_n,L_n))$ to $\cc$. In particular, this $C^*$-correspondence together with its dual $F(\rho_n,L_n)^*$ implements a $KK$-equivalence between $\B K(F)$ and $\cc$. We denote the corresponding classes in $KK$-theory by
\[
[F] \in KK_0( \B K(F), \cc) \q \T{and} \q [F^*] \in KK_0(\cc, \B K(F) ) .
\]
Combining these observations with the $KK$-equivalence from Theorem \ref{t:KKequiv}, we obtain the exact sequence
\[
\xymatrix{
& K_1(\B O) \ar[rr]^{([F] \hot_{\B K(F)} \cd ) \ci \pa} & & K_0( \cc ) \ar[rrr]^{ [F^*] \hot_{\B K(F)} { [j]} \hot_{\B T}  [\psi_+,\psi_-] } & & & K_0( \cc) \ar[r]^{(q \ci i)_* } & K_0( \B O ) \ar[dlll] \\
& & & & \{0\} \ar[ulll] & & 
}
\]
We recall that $i : \cc \to \B{T}$ denotes the unital inclusion of $\cc$ into the Toeplitz algebra and remark that $q \ci i : \cc \to \B{O}$ agrees with the unital inclusion of the complex numbers into $\B{O}$. We will abuse notation and denote the latter inclusion with the same symbol $i$. 

In the next proposition we compute the composition $[F^*] \hot_{\B K(F)} [j] \hot_{\B T}  [\psi_+,\psi_-] $, which we identify with the Euler class of the irreducible representation $\rho_n : SU(2) \to U(L_n)$, i.e. the alternating sum of $KK$-classes ${\bf 1}_\cc - [L_n] + [ \T{det}(\rho_n,L_n) ] \in KK_0(\cc,\cc)$.

\begin{prop}
We have the identity
\[
[j] \hot_{\B T}  [\psi_+,\psi_-]  =   [F] \hot_{\cc} \big( {\bf 1}_\cc -[L_n]+[\T{det}(\rho_n, L_n)]\big)  
\]
in $KK_0(\B K(F), \cc)$.
\end{prop}
\begin{proof}
By Proposition \ref{p:irrdet} we have that $\T{det}(\rho_n, L_n)$ is a one-dimensional complex vector space and hence that $[\T{det}(\rho_n, L_n)]={\bf 1}_\cc$ inside $KK_0(\cc,\cc)$. Hence we have to show that 
\begin{equation}\label{eq:lastKK}
[j] \hot_{\B T}  [\psi_+,\psi_-]  = 2 \cd [F] - [F] \hot_{\cc} [E_1] .
\end{equation}

Since $j : \B K(F) \to \B T$ is the inclusion we have that both $\psi_+ \ci j$ and $\psi_- \ci j : \B K(F) \to \B L( F \op F)$ factorises through the compact operators on $F \op F$ and the left hand side of \eqref{eq:lastKK} is therefore given by
\[
[j] \hot_{\B T} [\psi_+, \psi_-] = [\psi_+ \ci j, 0] - [\psi_- \ci j, 0] .
\] 
Now, letting $\phi : \B K(F) \to \B L(F)$ denote the inclusion of the compact operators into the bounded operators we have that $\psi_+ \ci j = \phi \op \phi : \B K(F) \to \B L(F \op F)$ and hence that $[\psi_+ \ci j,0] = 2 \cd [F]$ inside $KK_0(\B K(F),\cc)$.

Next, recall that $\psi_-(x) = W_R ( x \ot 1_{E_1} ) W_R^* : F \op F \to F \op F$ for all $x \in \B T$, where $W_R : F \ot E_1 \to F \op F$ is the isometry defined in \eqref{eq:WR}. In particular, we have that $W_R$ implements a unitary isomorphism between $F \ot E_1$ and $W_R W_R^*(F \op F)$.

We define the $*$-homomorphism $\phi_- : \B K(F) \to \B L\big( W_R W_R^* (F \op F) \big)$ by 
\[
\phi_-(x)(\xi) = (\psi_- \ci j)(\xi)
\]
for all $\xi \in W_R W_R^* (F \op F)$. We then have that $(\phi_-,0)$ is unitarily equivalent to the quasi-homomorphism $(\phi \ot 1_{E_1},0)$. Moreover, we see that the quasi-homomorphisms $(\psi_- \ci j,0)$ and $(\phi_-,0)$ agree up to addition of a degenerate quasi-homomorphism. We therefore obtain the identities
\[
[\psi_- \ci j,0] = [\phi_-,0] = [\phi \ot 1_{E_1},0] = [F] \hot_{\cc} [E_1]
\]
inside the $KK$-group $KK_0(\B K(F),\cc)$. %This ends the proof of the lemma.
\end{proof}

Combining the above results we obtain the $KK$-theoretic Gysin-sequence associated with the irreducible representation $\rho_n : SU(2) \to U(L_n)$:

\begin{thm}
The following sequence of $K$-groups is exact:
\[
\xymatrix{
& K_1(\B O) \ar[rr]^{([F] \hot_{\B K(F)} \cdot ) \ci \pa} & & K_0( \cc ) \ar[rrr]^{ {\bf 1}_\cc -[L_n]+[\T{det}(\rho_n, L_n)]} & & & K_0( \cc) \ar[r]^{ i_*} & K_0( \B O ) \ar[dlll] \\
& & & & \{0\} \ar[ulll] & & 
}
\]
\end{thm}

\begin{cor}
For every $n \in \nn$ we have 
\begin{equation}
K_0(\B O(\rho_n,L_n)) \cong \zz / (n-1)\zz \qquad  K_1(\B O(\rho_n,L_n)) \cong \begin{cases} \zz & n=1,\\
\{0\} & \mbox{otherwise.}
\end{cases}
\end{equation}
%In particular, it holds that $K_0( O(\rho_2,L_2)) = \{0\} = K_1(O(\rho_2,L_2))$.
\end{cor}
%\todo[inline]{How about the extension class in \cite{CaMe} vs the class of the defining extension?}
%\todo[inline]{Shall we conclude by mentioning that one should go from sphere to sphere bundles? It might be better to put this kind of comments in the introduction.}

\section{Concluding remarks and open problems}
The present paper raises a number of questions and open problems and we would like to conclude by listing a few of them:

\begin{enumerate}
\item It is relevant to consider the case where the representation $\tau : SU(2) \to U(H)$ is no longer irreducible, but where $H$ remains finite dimensional. We expect however that a lot of the considerations appearing in the present paper could be carried over to this more general context without too much trouble. In this direction we have so far only computed the determinant of the representation, see Proposition \ref{p:compdetredu}. 
\item In the present work we have only been studying $SU(2)$-subproduct systems in a Hilbert space context, meaning that we have in some sense been looking at $SU(2)$-bundles with a one-point parameter space. In order to find a noncommutative analogue of the classical $K$-theoretic Gysin sequence arising from a complex hermitian vector bundle of rank $2$, \cite{Ka78}, it is necessary to extend our work to $SU(2)$-subproduct systems with a non-trivial parameter space. This means that an interesting starting point could be a general $SU(2)$-$C^*$-correspondence where the left action factorises through the compact operators. In this context it could be relevant to compare the corresponding extension class with the class appearing in \cite{CaMe}.
\item In analogy with the case of Cuntz--Pimsner algebras arising from a $C^*$-correspondence it is an important problem to settle the universal properties both for the Toeplitz algebras and the Cuntz--Pimsner algebras coming from our $SU(2)$-equivariant data. In particular, it would be worthwhile to look for an $SU(2)$-gauge invariant uniqueness theorem as obtained in the $U(1)$-setting by Katsura in \cite[Theorem 6.4]{Kat}. % in the $U(1)$-setting.
\end{enumerate}

\appendix
\section{Commutators and polar decompositions}\label{a:quasi}
Throughout this appendix we let $X$ be a countably generated Hilbert $C^*$-module over a $C^*$-algebra $B$. 

\begin{prop}\label{p:polar}
Suppose that $x,y : X \to X$ are bounded adjointable operator and that there exists a norm-dense submodule $\s X \su X$ such that
\[
\s X \su \T{Im}(y^* y) \, \, \mbox{and} \, \, \, x( \s X) , x^*(\s X) , y^*( \s X) \su \s X .
\] 
Suppose moreover that $K : X \to X$ is a positive compact operator and that $L,\ov{L}, M, \ov{M} : X \to X$ are bounded adjointable operators such that 
\begin{enumerate}
\item $K^{1/2} L K^{1/2} = [x,y]$ and $K^{1/2} \ov{L} K^{1/2} = [x,y^*]$;
\item $M K = [x,y]$ and $K \ov{M} = [x,y^*]$.
\end{enumerate}
Suppose finally that there exists a constant $C > 0$ such that
\[
\| K^{1/2} (\la + y^* y)^{-1/2} \|  \, , \, \, \| K (\la + y^* y)^{-1} \|  \, ,  \, \,  \| K^{1/2} y (\la + y^* y)^{-1} \|  \leq C 
\]
for all $\la > 0$. Then we may conclude that the unbounded operator $y |y|^{-1} : \T{Im}( |y|) \to X$ extends to a bounded adjointable isometry $\te : X \to X$ satisfying that $[x,\te]$ and $[x^*,\te]$ both lie in $\B K(X)$.
\end{prop}
\begin{proof}
We start by recording that since $|y| : X \to X$ is positive and has dense image we know that $|y|^{-1} : \T{Im}( |y|) \to X$ is a well-defined unbounded positive and regular operator. The unbounded operator $y |y|^{-1} : \T{Im}(|y|) \to X$ then extends to an isometry $\te : X \to X$ and this isometry is adjointable since $|y|^{-1} y^*$ is densely defined as well (the domain of $|y|^{-1} y^*$ contains $\s X$ and the adjoint $\te^* : X \to X$ is the unique bounded extension of $|y|^{-1} y^*$). 

It follows from the identities in $(1)$ and compactness of $K : X \to X$ that both $[x,y]$ and $[x,y^*]$ lie in $\B K(X)$. 

For each $\la > 0$ we put $R_\la := (\la + y^* y)^{-1}$. For every $\xi \in \T{Im}(y^*y)$ we have that $|y|^{-1} \xi = \frac{1}{\pi} \int_0^\infty \la^{-1/2} R_\la \xi d \la$ where the integral converges absolutely (using the norm on $X$). We compute that
\[
\begin{split}
[x, R_\la] 
& = - R_\la [x,y^* y] R_\la = - R_\la [x,y^*] y R_\la - R_\la y^* [x,y] R_\la \\
& = - R_\la K^{1/2} \ov{L} K^{1/2} y R_\la - R_\la y^* M K R_\la .
\end{split}
\]
This in particular implies that $[x,R_\la] \in \B K(X)$. Notice now that $\| y^* y R_\la \| \leq 1$ for all $\la > 0$. Combining this estimate with our assumptions we obtain that
\begin{equation}\label{eq:estimate}
\begin{split}
\big\| y [x,R_\la] \big\| 
& \leq \| y R_\la K^{1/2} \|  \cd \| \ov{L} \| \cd \| K^{1/2} y R_\la \| 
+ \| y R_\la y^* \| \cd \| M \| \cd \| K R_\la \| \\
& \leq { C^2 \cd \| \ov{L} \|  + C \cd \| M \| }
\end{split}
\end{equation}
for all $\la > 0$. 

%Similarly, we obtain that
%\[
%\big\| [x,R_\la] y^* \big\| \leq C^2 \cd ( \| L \| + \| \ov{M} \| ) 
%\]
%for all $\la > 0$. 

Remark now that the integral $\int_1^\infty \la^{-1/2} y [x,R_\la] d\la$ converges absolutely in operator norm since $\| R_\la \| \leq \la^{-1}$ for all $\la > 0$. Moreover, we obtain from the estimate in \eqref{eq:estimate} that the integral $\int_0^1 \la^{-1/2} y [x,R_\la] d\la$ converges absolutely in operator norm as well. The whole integral
\[
\int_0^\infty \la^{-1/2} y [ x, R_\la  ] d \la 
\]
therefore converges absolutely in operator norm and since the integrand is a continuous map $(0,\infty) \to \B K(X)$ we conclude that
\[
\frac{1}{\pi} \int_0^\infty \la^{-1/2} y [ x, R_\la ] d \la  \in \B K(X) .
\]
We may likewise show that the integral
\[
\frac{1}{\pi} \int_0^\infty \la^{-1/2} K R_\la d\la
\]
converges absolutely to a compact operator.

The claim that $[x,\te]$ is a compact operator is now verified by noting that
\[
\begin{split}
[x, \te] \xi 
& = [x,y] |y|^{-1} \xi + y [x, |y|^{-1}] \xi \\ 
& = M \frac{1}{\pi} \int_0^\infty \la^{-1/2} K  R_\la(\xi) d\la + \frac{1}{\pi} \int_0^\infty \la^{-1/2} y [ x, R_\la ] (\xi) d \la ,  
\end{split}
\]
for all $\xi \in \s X$.

Since our assumptions are symmetric in $x$ and $x^*$ it follows immediately that $[x^*,\te]$ is a compact operator as well.
\end{proof}

%Suppose now that $A$ is a $C^*$-algebra acting on $X$ via a $*$-homomorphism $\psi : A \to \B L(X)$ so that $X$ is a $C^*$-correspondence from $A$ to $B$. We may conjugate the $*$-homomorphism $\psi : A \to \B L(X)$ by the unitary $u : X \to X$ and obtain another $*$-homomorphism $\phi : A \to \B L(X)$ with $\phi(a) = u \psi(a) u^*$ for all $a \in A$.
%
%We say that a subset $S \su A$ generates $A$ when $A$ agrees with the smallest $C^*$-subalgebra of $A$ which contains $S$.
%
%\begin{prop}
%Suppose that $S \su A$ generates $A$ and that $\psi(s)$ satisfies the conditions of Proposition \ref{p:polar} for all $s \in S$. Then the pair of $*$-homomorphisms $(\psi, \phi)$ is a quasi-homomorphism.  
%\end{prop}
%\begin{proof}
%We need to show that $\psi(a) - u \psi(a) u^* \in \B K(X)$ for all $a \in A$. By Proposition \ref{p:polar} this holds for all $s \in S$ since $\psi(s) - u \psi(s) u^* = [\psi(s),u] u^*$. It then suffices to remark that the set of elements $a \in A$ satisfying that $\psi(a) - u \psi(a) u^* \in \B K(X)$ form a $C^*$-subalgebra of $A$.
%\end{proof}
%\todo{References in grey are not cited (yet)}

\subsubsection*{Acknowledgements} This work has benefited from conversation with our colleagues and collaborators Adam Dor-On, Evgenios Kakariadis, Magnus Goffeng, Bram Mesland, Ryszard Nest, Adam Rennie, Michael Skeide, Wojciech Szyma\'nski, Mike Whittaker. The authors would like to express their gratitude to Giovanni Landi for the many inspiring conversations and discussions that lead us to considering this problem.

{\bf This work is part of the research programme VENI with project number 016.192.237, which is (partly) financed by the Dutch Research Council (NWO).} 

{\bf JK gratefully acknowledges the financial support from the Independent Research Fund Denmark through grant no. 7014-00145B and grant no. 9040-00107B.}

\begin{figure}[h]
    \centering
    \begin{minipage}{0.45\textwidth}
        \centering
        \includegraphics[width=0.4\textwidth]{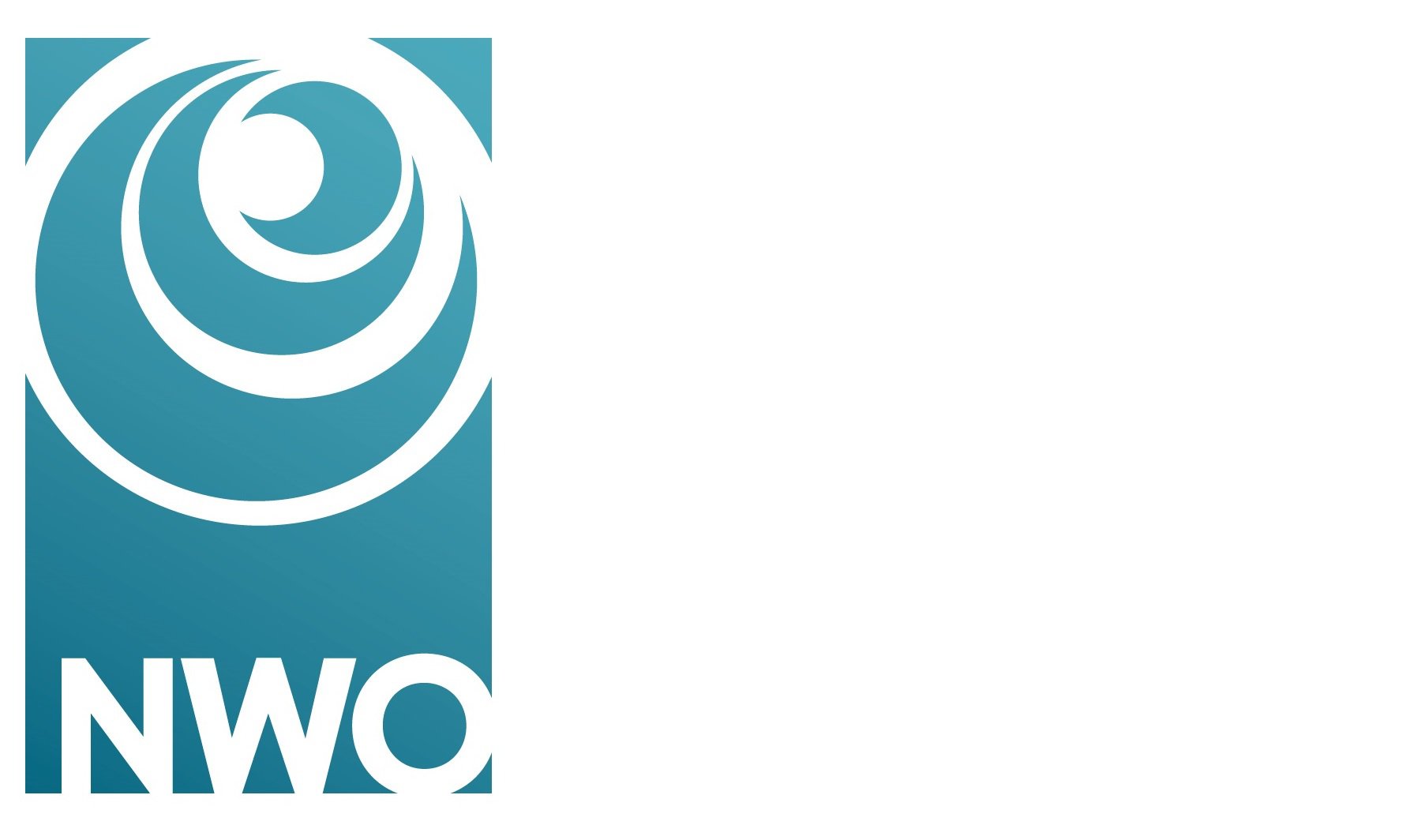} % first figure itself
 %       \caption*{first figure}
    \end{minipage}\hfill
    \begin{minipage}{0.45\textwidth}
        \centering
        \includegraphics[width=0.4\textwidth]{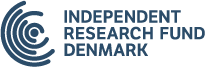} % second figure itself
%        \caption*{second figure}
    \end{minipage}
\end{figure}

\end{document}